\documentclass[11pt]{amsart}
\usepackage{amsmath,amsfonts,amssymb,mathrsfs}
\usepackage{amssymb,mathrsfs,graphicx,enumerate,mathabx}
\usepackage{amssymb,amscd,amsthm,bbm}
\usepackage{mathtools}
\mathtoolsset{showonlyrefs}
\usepackage[retainorgcmds]{IEEEtrantools}
\usepackage{colortbl}
\usepackage{lipsum}
\usepackage{graphicx, subfigure}
\usepackage{graphicx}
\usepackage{hyperref}
\mathtoolsset{showonlyrefs}

\title[]{Stability of equilibria of an aggregation-diffusion energy on sphere}

\author[Fetecau]{Razvan C. Fetecau}
\address[Razvan C. Fetecau]{\newline Department of Mathematics, Simon Fraser University, 8888 University Dr., Burnaby, BC V5A 1S6, Canada}
\email{van@math.sfu.ca}

\author[Park]{Hansol Park}
\address[Hansol Park]{\newline Department of Mathematics, National Tsing Hua University, Section 2, Kuang-Fu Road, Hsinchu 30013, Taiwan}
\email{hansolpark@math.nthu.edu.tw}

\newtheorem{theorem}{Theorem}[section]
\newtheorem{lemma}{Lemma}[section]

\newtheorem{proposition}{Proposition}[section]
\newtheorem{remark}{Remark}[section]

\newcommand{\bbr}{\mathbb R}

\newcommand{\bbs}{\mathbb S}

\newcommand{\calA}{\mathcal{A}}
\newcommand{\calC}{\mathcal{C}}
\newcommand{\calF}{\mathcal{F}}

\newcommand{\calP}{\mathcal{P}}

\def\d{\mathrm{d}}
\newcommand{\dS}{\mathrm{d}S} 
\newcommand{\dx}{\mathrm{d}S (x)}
\newcommand{\dy}{\mathrm{d}S (y)}
\newcommand{\dz}{\mathrm{d}S (z)}

\newcommand{\rhou}{\rho_{\mathrm{uni}}}

\newcommand{\supp}{\mathrm{supp}(\rho)}

\newcommand{\dm}{d} 

\begin{document}


\begin{abstract}
We consider an aggregation-diffusion energy on the sphere and investigate the stability of its equilibria. The energy consists of a porous-medium type nonlinear entropy $\frac{1}{m-1}\int \rho(x)^m\dx$ with $m>1$, together with an interaction energy modeled by a quadratic interaction potential. The energy generalizes the Onsager free energy with dipolar potential, which models polymer orientation. Our study complements the authors' previous work [Nonlinearity {\bf 39} (2026), 055011], where the ground states of the energy functional were investigated. In the current paper we extend the previous results on existence of equilibria for $m>2$ from $\dm =2$ to arbitrary dimension $\dm \geq 2$. This allows us to present the bifurcation structure (with respect to the interaction strength) of all equilibria in any dimension $\dm \geq 2$, for all $m>1$.  Furthermore, we provide a complete classification of the stability of all equilibria of the energy, by deriving a criterion for stability that can be checked explicitly. In particular, for $m>2$ we identify a saddle-node bifurcation of strictly supported equilibria, and a subcritical pitchfork bifurcation where the uniform distribution loses stability.
\end{abstract}

\maketitle 


\section{Introduction}
\label{sect:intro}
This paper is concerned with the stability of equilibria of the following free energy functional:
\begin{equation}
\label{energy-sphere}
E[\rho]=\frac{1}{m-1}\int_{\bbs^\dm}\rho(x)^m \dx+\frac{\kappa}{4}\iint_{\bbs^\dm\times \bbs^\dm}\|x-y\|^2\rho(x)\rho(y) \dx\dy,
\end{equation}
where $\bbs^\dm$ represents the $\dm$-dimensional unit sphere, $\|\cdot\|$ denotes the Euclidean distance in $\bbr^{\dm+1}$, and the integration is with respect to the surface area measure $\dS$ of the sphere. The energy functional is defined on the space $\mathcal{P}_{ac}(\bbs^\dm)$ of absolutely continuous (with respect to $\dS$) probability measures on $\bbs^\dm$. Also, $m>1$ represents the diffusion exponent, and $\kappa>0$ is the interaction strength. 

The free energy \eqref{energy-sphere} belongs to a large class of functionals called aggregation-diffusion energies. The first term in \eqref{energy-sphere}, which favours spreading, represents the entropy, while the second term is a nonlocal interaction energy that promotes aggregation. Consequently, the two terms have competing effects, and the equilibria and the global minimizers depend on the balance between them. Due to this interpretation, such energies and their associated gradient flow dynamics have a wide range of applications to modelling collective behaviour in sciences and engineering \cite{CaMcVi2006, HoPu2005, JiEgerstedt2007, KoSuUmBe2011, LeToBe2009, MotschTadmor2014}. Given the extensive literature on this topic, it would be impractical to present a complete account of this literature here, and we refer the reader to several review  articles \cite{BailoCarrilloCastro2024, CarrilloCraigYao2019, CarrilloVecil2010}, as well as to the foundational work from \cite{AGS2005, CaMcVi2006} on gradient flows associated to this type of energies. Among many related works, we also refer to \cite{BurgerDiFrancescoFranek, BuFeHu14, CaHiVoYa2019, CaHoMaVo2018, DelgadinoXukaiYao2022} for studies on qualitative properties of equilibria, as well as on their uniqueness or lack thereof.

The energy in \eqref{energy-sphere} corresponds specifically to nonlinear diffusion with exponent $m>1$ in the slow diffusion (porous media) regime, and nonlocal attraction modelled by the quadratic interaction potential $W:\bbs^\dm \times \bbs^\dm \to \bbr$ given by $W(x,y) = \frac{1}{2} \| x-y \|^2$. Nonlinear diffusion was considered in various applications of this class of models, such as swarming  \cite{BurgerDiFrancescoFranek, BuFeHu14, Kaib17, TBL}, granular media \cite{CaMcVi2006}, machine learning \cite{PeletierShalova2025}, and opinion formation \cite{FagioliRadici2021}. Also, interaction potentials in quadratic (and more generally, in power-law) form have been considered extensively in studies on interaction energies, with or without diffusion \cite{Balague_etalARMA, BertozziCarilloLaurent,BertozziLaurent,CarrilloChipotHuang2014,CaHoMaVo2018,ChFeTo2015,FeHu13,HoPu2005,LeToBe2009}.  Note however that the vast majority of these references consider the model set up on the Euclidean space $\bbr^\dm$. Work on aggregation-diffusion energies and their related evolution equations on general Riemannian manifolds is more recent and much more limited \cite{CaFePa2025a, FeHaPa2021, FePa2023a, FePa2024b, FePa2024a, HaHwKiKiMi, ha2022emergent, PeletierShalova2025, WuSlepcev2015}. 

A particularly important application of the free energy \eqref{energy-sphere} is to rod-like polymers \cite{ConstantinKevrekidisTiti2004, fatkullin2005critical}, where $\rho$ represents the probability distribution function for the orientation of a polymer interpreted as a rigid rod of unit vector $x \in \bbs^2$. The special case $m=1$ of \eqref{energy-sphere} (in this case the entropic term is given by $\int_{\bbs^\dm} \rho(x) \log \rho(x) \dx$) corresponds to the celebrated Onsager free energy  \cite{Onsager1949}, used to study isotropic-to-nematic phase transitions in systems of hard rods. For this  specific application, the quadratic interaction potential, or its equivalent expression $- x \cdot y $, is called the dipolar potential. As the interaction strength $\kappa$ increases through a critical value, the uniform distribution (the isotropic state) loses stability to a bell-shaped distribution (the nematic state) \cite{DegondFrouvelleLiu2014, fatkullin2005critical, FrouvelleLiu2012}. The time evolution equation (i.e., the gradient flow of the Onsager free energy) was studied in \cite{FrouvelleLiu2012}, where explicit rates of convergence of the solutions to global energy minimizers were derived and investigated.

Also in applications to rod-shaped polymers, as the molecular geometry becomes more complex and exhibits additional symmetries, the interaction potential can also be modified accordingly. A classical example is the Maier--Saupe model, where the interaction potential is given by $W_{\mathrm{MS}}(x,y)=-(x\cdot y)^2$. Unlike the dipolar interaction, this potential possesses head-to-tail symmetry, which comes with a very different set of equilibria for the corresponding Onsager energy. The mathematical analysis of the free energy and its associated gradient flow with the Maier--Saupe interaction, including the classification of equilibria, stability, and bifurcation phenomena, has been carried out for example in \cite{ball2021axisymmetry, fatkullin2005critical, liu2005axial, zhou2007characterization}.
 
The extension of the Onsager free energy to nonlinear diffusion was considered recently in \cite{FePaVa2025} for the slow diffusion regime ($m>1$) and in \cite{FePa2026-fast} for fast diffusion ($0<m<1$). In both of the papers \cite{FePaVa2025} and \cite{FePa2026-fast}, the focus is to find and characterize the global minimizers of the energy functional \eqref{energy-sphere}. The aim of the present work is to complement the findings from \cite{FePaVa2025}, where a very interesting bifurcation/transition structure of energy equilibria has been identified -- but not investigated, in terms of the size of the interaction strength $\kappa$. In the current research we provide a complete characterization of the bifurcations in energy equilibria that occur with varying the parameter $\kappa$. In addition, we make an important extension to the results in \cite{FePaVa2025}, by generalizing them from $\dm =2$ to arbitrary dimension $\dm\geq 2$ for the range $m>2$ of the diffusion exponent (which is the most interesting in terms of the bifurcation structure).

One of the key novelties of the slow diffusion regime $m>1$, compared to the linear diffusion, is the existence and emergence of equilibria supported on a strict subset of the sphere.  For all $m>1$, at a critical value of $\kappa$ (denoted by $\kappa_1$ in the paper), the uniform distribution loses stability. We show that a distinct family of fully supported equilibria forms at $\kappa=\kappa_1$, via a supercritical pitchfork bifurcation (for $1<m<2$) or a subcritical pitchfork bifurcation (for $m>2)$. In addition, at a second critical value $\kappa=\kappa_2$, these equilibria transition from being fully supported to having strict support on the sphere. Finally, in the case $m>2$ we show that a saddle-node bifurcation of strictly supported equilibria occurs at a third critical value $\kappa=\kappa_3$.

To establish the stability of the equilibria we investigate the sign of the second variation of the energy, which in turn reduces to a constrained minimization problem -- see the functional $\calF[\psi]$ and the admissible set $\calA$ introduced in Section \ref{sect:2nd-order}. Minimizing this functional, which is a delicate issue that required considerable effort, is an interesting problem in itself. The difficulty comes from the fact that for strictly supported equilibria, the critical points of this functional, and hence the possible minimizers, do not belong to the admissible set. To deal with this issue, we first investigate the minimization problem on a larger set, and then connect the two minimization problems, to derive an explicit criterion for the stability of the equilibria.

In closing, we mention a very recent line of research related to our paper, in the area of machine learning. Using the same setup on the sphere, but a different interaction potential, free energies appear in applications to large language models, specifically to noisy transformers \cite{Geshkovski_etal2025}. Similar to our work, phase transitions have been studied for the equilibria of such energies, where in particular, the uniform distribution loses stability and undergoes a bifurcation at a certain critical noise strength \cite{Balasubramanian_etal2025, Gerber_etal2025, mun2026phase2, mun2026phase1, ShalovaSchlichting2025}. This is a very active research area at the moment, which is developing very fast.

The remaining part of the paper is organized as follows. In Section \ref{sect:prelims}, we present known results on the equilibria of the energy functional. These results were established in \cite{FePaVa2025}, where the global minimizers of the energy functional were investigated. In Section \ref{sect:existence-mg2}, we consider the case $m>2$ in an arbitrary dimension $\dm$, which was not addressed in \cite{FePaVa2025} due to a technical issue. The main result of this section is Theorem \ref{thm:equigenerald}, which extends to arbitrary dimension $\dm \geq 2$ the existence of equilibria established in \cite{FePaVa2025} for $\dm=2$. In Section \ref{sect:first-order}, we present a stability analysis based on the first variation of the energy. In Sections \ref{sect:2nd-order} and \ref{sect:2ndorder-fscs}, we analyze the stability of equilibria using the second variation of the energy. In Section \ref{sect:2nd-order} we study the minimization of a certain functional, which helps us establish the sign of the second variation. The main result in this section is Theorem \ref{thm:minF}, which shows how the stability of an equilibrium is determined by the minimization of this functional. Finally, in Section \ref{sect:2ndorder-fscs}, we apply this theorem and make explicit calculations to provide a complete classification of the stability of equilibria (Theorems \ref{thm:mg2:fs-stab} and \ref{thm:mg2:cs-stab}). Some supporting calculations are presented in the Appendix.


\section{Preliminaries and background}
\label{sect:prelims}
In this section, we introduce the equilibria and present some existing results from \cite{FePaVa2025}. Since for $x,y \in \bbs^\dm$, we have $\| x-y \|^2 = 2 - 2 \, x\cdot y$, the energy \eqref{energy-sphere} can also be written as
\begin{equation}
\label{eqn:energy-s}
\begin{aligned}
E[\rho]&=\frac{1}{m-1}\int_{\bbs^\dm}\rho(x)^m\dx-\frac{\kappa}{2}\iint_{\bbs^\dm\times\bbs^\dm} (x \cdot y) \rho(x)\rho(y)\dx \dy+\frac{\kappa}{2} \\
&=\frac{1}{m-1}\int_{\bbs^\dm}\rho(x)^m\dx - \frac{\kappa}{2} \| c_\rho\|^2 +\frac{\kappa}{2},
\end{aligned}
\end{equation}
where $c_\rho \in \bbr^{\dm+1}$ denotes the centre of mass of $\rho \in \calP_{ac}(\bbs^\dm)$:
\begin{equation}
\label{eqn:CM}
c_\rho = \int_{\bbs^\dm} x\rho(x) \dx.
\end{equation}

Note that since $\rho$ has unit mass, its centre of mass coincides with the first moment. For general density functions $f$ on $\bbs^\dm$, we will use the notation
\begin{equation}
\label{eqn:CM-general}
c_f = \frac{\int_{\bbs^\dm} x f(x) \dx}{\int_{\bbs^\dm} f(x) \dx}.
\end{equation}


\subsection{Critical points of the energy} 
\label{subsect:cp} 
We present here the critical points of the energy functional, as derived and discussed in \cite{FePaVa2025}. Take $\rho \in \calP_{ac}(\bbs^\dm)$ and a family of densities  $\rho^\epsilon \in \calP_{ac}(\bbs^\dm)$
of the form
\[
\rho^\epsilon = \rho + \epsilon \delta \rho.
\]
Note that in particular, $\int_{\bbs^\dm} \delta \rho(x) \dx =0$, so that $\rho^\epsilon$ are probability densities. Then, by \eqref{eqn:energy-s}, we compute
\begin{align*}
  \frac{\d}{\d \epsilon} E[\rho^\epsilon] &= 
  \frac{m}{m-1} \int_{\bbs^\dm} \rho^\epsilon(x)^{m-1} \delta \rho(x) \dx -\kappa \,  c_{\rho^\epsilon} \cdot
  \int_{\bbs^\dm} x \delta \rho(x) \dx.
\end{align*}

At a critical point, we have $\left. \frac{\d}{\d \epsilon} E[\rho^\epsilon] \right|_{\epsilon=0} = 0$, and hence,
\[
\int_{\bbs^\dm}\left(\frac{m}{m-1} \rho(x)^{m-1}-\kappa  c_\rho \cdot x \right)  \delta \rho(x) \dx =0,
\]
for all $\delta \rho$ of zero mass. Furthermore, $\rho \in \calP_{ac}(\bbs^\dm)$ is a critical point of the energy (alternatively, an equilibrium density) if the expression
\begin{equation}
\label{eqn:fvE}
\frac{\delta E}{\delta \rho}= \frac{m}{m-1} \rho(x)^{m-1}-\kappa \, c_\rho \cdot x
\end{equation}
is constant on each closed connected component of the support of $\rho$ \cite{CaDePa2019, Kaib17}. The value of the constant can be different in different components of the support. 

Several lemmas, presented and proved in Section \ref{sect:first-order}, restrict significantly the set of equilibria that can be stable. First, by Lemma \ref{lem:constant-lambda}, for an equilibrium $\rho$ to be stable, the expression $\frac{\delta E}{\delta \rho}$ from \eqref{eqn:fvE} must necessarily take the {\em same} constant value everywhere in $\mathrm{supp}(\rho)$, i.e., $\rho$ satisfies
\begin{equation}
\label{eqn:EL}
\frac{m}{m-1} \rho(x)^{m-1}-\kappa \, c_\rho \cdot x = \lambda, \qquad \forall x\in \mathrm{supp}(\rho),
\end{equation} 
for a constant $\lambda$. As shown in Lemma \ref{lem:constant-lambda}, if different components of $\supp$ correspond to different values of $\lambda$, then the energy can be decreased by transporting mass between the components. 

Second, by Lemma \ref{lem:supportcondi}, a stable equilibrium $\rho$ must necessarily have {\em connected} support, given by
\begin{equation}
\label{eqn:support}
\mathrm{supp}(\rho)=\{x:\kappa\, c_\rho \cdot x+\lambda\geq0\},
\end{equation}
meaning that $\rho$ is fully supported in the geodesic disk of $\bbs^\dm$ given by $\kappa \, c_\rho \cdot x \geq -\lambda$. We defer the presentation of Lemmas \ref{lem:constant-lambda} and \ref{lem:supportcondi} to Section \ref{sect:first-order}, where we investigate the stability of equilibria by using the first variation of the energy only. For the purpose of this section, we assume the results of the two lemmas, and note that the equilibria that are of interest for the current work satisfy \eqref{eqn:EL} and \eqref{eqn:support}.

The simplest equilibrium that satisfies \eqref{eqn:EL} and \eqref{eqn:support} is the uniform distribution on the sphere,
\begin{equation}
\label{eqn:rho-uni}
\rhou(x) = \frac{1}{|\bbs^\dm|}, \qquad \forall x \in \bbs^\dm,
\end{equation}
where $|\bbs^\dm|$ denotes the area of $\bbs^\dm$. Indeed, $\rhou$ is a solution of \eqref{eqn:EL} for all $\kappa>0$, where $c_{\rhou} =0$, and the constant $\lambda$ is given by
\begin{equation}
\label{eqn:lambda-uni}
\lambda_{\text{uni}} = \left(\frac{m}{m-1}\right) \frac{1}{|\bbs^\dm|^{m-1}}.
\end{equation}

By \eqref{eqn:EL}, we write the equilibria as
\begin{equation}
\label{eqn:equil}
\rho(x)=\begin{cases}
\left(\frac{m-1}{m} \right)^{\frac{1}{m-1}} \left(\lambda+\kappa\, c_\rho \cdot x\right)^{\frac{1}{m-1}},\quad &\text{for }x\in \mathrm{supp}(\rho),\\[5pt]
0,&\text{otherwise}.
\end{cases}
\end{equation}
Note that by its explicit expression, $\rho$ is axially symmetric with respect to the direction of its centre of mass. Without loss of generality, set 
\begin{equation}
\label{eqn:x0}
c_\rho=\|c_\rho\| x_0,
\end{equation}
for some fixed (but arbitrary) unit vector $x_0 \in \bbs^\dm$. Also define $\theta_x=\arccos ( x_0 \cdot x) \in [0,\pi]$, to write $\, c_\rho \cdot x=\|c_\rho\|\cos\theta_x$. The condition $\lambda+\kappa\, c_\rho \cdot x\geq0$ on $\mathrm{supp}(\rho)$ can then be simplified into
\begin{equation}
\label{ineq:cos}
-\frac{\lambda}{\kappa\|c_\rho\|}\leq\cos\theta_x, \qquad \forall x \in \mathrm{supp}(\rho).
\end{equation}

Based on the relative size of $\lambda$, two types of equilibria in the form \eqref{eqn:equil} were identified in \cite{FePaVa2025}:
\smallskip

{\em a) Equilibria fully supported on $\bbs^\dm$ ($\lambda \geq  \kappa \|c_\rho\|)$.} As in this case $-\frac{\lambda}{\kappa\|c_\rho\|}\leq -1$, $\theta_x$ can take any value in $[0,\pi]$. Hence, $\supp = \bbs^\dm$, and the equilibrium from \eqref{eqn:equil} is given by
\begin{equation}
\label{eqn:equil-fs}
\rho(x)=\left(\frac{m-1}{m}\right)^{\frac{1}{m-1}}\left(\lambda+\kappa\|c_\rho\|\cos\theta_x\right)^{\frac{1}{m-1}},\qquad\forall x\in \bbs^\dm.
\end{equation}

For consistency, $\rho$ given by \eqref{eqn:equil-fs} has to have centre of mass at $c_\rho$. Using the definition of $c_\rho$ and hyperspherical coordinates, we write
\begin{equation}
\label{eqn:cm-sq}
\begin{aligned}
\|c_\rho\|^2&= c_\rho \cdot\int_{\bbs^\dm}x\rho(x)\dS(x)\\
&= \int_{\bbs^{\dm}}\, c_\rho \cdot x \rho(x)\dS(x)\\
&=\dm w_{\dm}\|c_\rho\|\int_0^\pi \cos\theta \left(\frac{m-1}{m}\right)^{\frac{1}{m-1}}\left(\lambda+\kappa\|c_\rho\|\cos\theta\right)^{\frac{1}{m-1}}\sin^{\dm-1}\theta\d\theta,
\end{aligned}
\end{equation}
where for the last equal sign we also used that $|\bbs^{\dm-1}|=\dm w_{\dm}$, with $w_{\dm}$ denoting the volume of the $\dm$-dimensional unit ball.

Also, $\rho$ has unit mass. Using again hyperspherical coordinates, together with the calculation above, we arrive  at the following two equations: 
\begin{equation}
\label{eqn:system-fs-mg1}
\begin{aligned}
1&=\dm w_\dm\int_0^\pi \left(\frac{m-1}{m}\right)^{\frac{1}{m-1}}\left(\lambda+\kappa\|c_\rho\|\cos\theta\right)^{\frac{1}{m-1}}\sin^{\dm-1}\theta \, \d\theta, \\
\|c_\rho\|& =\dm w_\dm\int_0^\pi \left(\frac{m-1}{m}\right)^{\frac{1}{m-1}}\left(\lambda+\kappa\|c_\rho\|\cos\theta\right)^{\frac{1}{m-1}}\sin^{\dm-1}\theta \cos\theta \, \d\theta.
\end{aligned}
\end{equation}
The system of equations \eqref{eqn:system-fs-mg1} needs to be solved for $\lambda$ and $\|c_\rho\|$, which in turn determine $\rho$.

\medskip

{\em b) Equilibria supported on a strict subset of $\bbs^\dm$ ($-\kappa \|c_\rho\| < \lambda < \kappa \|c_\rho\|$).} In this case we have $-1<-\frac{\lambda}{\kappa\|c_\rho\|}<1$, and by \eqref{ineq:cos} we have $\theta_x \in \left[0,\arccos\left(-\frac{\lambda}{\kappa\|c_\rho\|}\right)\right]$. The equilibrium $\rho$ from \eqref{eqn:equil} is given by
\begin{equation}
\label{eqn:equil-cs}
\rho(x)=\begin{cases}
\left(\frac{m-1}{m}\right)^{\frac{1}{m-1}}\left(\lambda+\kappa\|c_\rho\|\cos\theta_x\right)^{\frac{1}{m-1}},\qquad&\text{ if }0\leq\theta_x\leq \arccos\left(-\frac{\lambda}{\kappa\|c_\rho\|}\right),\\[5pt]
0,\qquad&\text{ otherwise}.
\end{cases}
\end{equation}

Similarly to case a), $\lambda$ and $\|c_\rho\|$ can be found by requiring that $\rho$ has unit mass and centre of mass at $c_\rho$. Denote by 
\begin{equation}
\label{eqn:phi-notation}
    \phi = \arccos\left(-\frac{\lambda}{\kappa\|c_\rho\|}\right),
\end{equation}
and follow a similar process as in case a), to get
\begin{equation}
\label{eqn:system-cs-mg1}
\begin{aligned}
1&=\dm w_\dm\int_0^\phi \left(\frac{m-1}{m}\right)^{\frac{1}{m-1}}\left(\lambda+\kappa\|c_\rho\|\cos\theta\right)^{\frac{1}{m-1}}\sin^{\dm-1}\theta \, \d\theta, \\
\|c_\rho\|& =\dm w_\dm\int_0^\phi \left(\frac{m-1}{m}\right)^{\frac{1}{m-1}}\left(\lambda+\kappa\|c_\rho\|\cos\theta\right)^{\frac{1}{m-1}}\sin^{\dm-1}\theta \cos\theta \, \d\theta.
\end{aligned}
\end{equation}
System \eqref{eqn:system-cs-mg1} has to be solved for $\lambda$ and $\|c_\rho\|$.

Finally, we note that \eqref{ineq:cos} has no solutions if $\lambda < -\kappa\|c_\rho\|$. Also, for $\lambda=-\kappa\|c_\rho\|$, the only solution to \eqref{ineq:cos} is $\theta_x=0$; however this corresponds to a Dirac mass concentrated at one point, which is not an admissible density. 

\subsection{Background: existing results on equilibria of the energy}
\label{subsect:equil-review}

The uniform distribution \eqref{eqn:rho-uni} is a critical point of the energy for any $\kappa>0$. This equilibrium is very important in applications of the model to polymer orientation, as it corresponds to the isotropic state. The stability of the uniform distribution has been established in \cite{FePaVa2025}, we simply list the result here.

\begin{proposition} \cite[Proposition 3.1]{FePaVa2025}
\label{prop:stab-unif}
The uniform distribution $\rhou$ is a stable critical point of the energy if $\kappa \leq \kappa_1$, where
\begin{equation}
\label{eqn:kappa1}
\kappa_1:= \frac{m (\dm+1)}{|\bbs^\dm|^{m-1}},
\end{equation}
and unstable if $\kappa>\kappa_1$.
\end{proposition}

For the equilibria \eqref{eqn:equil-fs} and \eqref{eqn:equil-cs} we distinguish two ranges of $m$: (i) $1<m<2$, and (ii) $m>2$. The stability of equilibria in case (i) was essentially established in \cite{FePaVa2025}, but we include it here for completeness.

\subsubsection{\underline{Case (i) $1<m<2$}}
\label{subsubsect:prelim:ml2}
In this case, a fully supported equilibrium in the form \eqref{eqn:equil-fs} bifurcates from the uniform distribution at $\kappa=\kappa_1$.  Moreover, there exists a second critical value $\kappa_2>\kappa_1$ where the fully supported equilibrium transitions to a strictly supported equilibrium in the form \eqref{eqn:equil-cs}. The value of $\kappa_2$ can be computed explicitly in terms of $m$ and $\dm$ \cite{FePaVa2025} -- see \eqref{eqn:kappa2}. The precise results are given by the following proposition. 

\begin{proposition} \cite[Proposition 4.1 \& Theorem 4.1]{FePaVa2025}
\label{prop:bif-mg1}
For any $1<m<2$ and $\dm\geq 2$, there exist two critical values $\kappa_2>\kappa_1$ of $\kappa$ such that:

\noindent i) At $\kappa=\kappa_1$, a fully supported equilibrium in the form \eqref{eqn:equil-fs} bifurcates from the uniform distribution $\rhou$. For any $\kappa \in (\kappa_1,\kappa_2)$, there exists a unique equilibrium in the form \eqref{eqn:equil-fs}, given by
\begin{equation}
\label{eqn:rhok-fs}
\rho_\kappa(x)=\left(\frac{m-1}{m}\right)^{\frac{1}{m-1}}\left(\lambda_\kappa+\kappa \|c_{\rho_\kappa}\| \cos\theta_x\right)^{\frac{1}{m-1}},\qquad\forall x\in \bbs^\dm,
\end{equation}
where $\|c_{\rho_\kappa}\|$ and $\lambda_\kappa$ are uniquely determined by $\kappa$ from  \eqref{eqn:system-fs-mg1}. 
\smallskip

\noindent ii) At $\kappa=\kappa_2$, the fully supported equilibria from part i) turn into equilibria with support strictly contained in $\bbs^\dm$.  For any $\kappa \in (\kappa_2,\infty)$, there exists a unique equilibrium in the form \eqref{eqn:equil-cs}, given by
\begin{equation}
\label{eqn:rhok-cs}
\rho_\kappa(x)=\begin{cases}
\left(\frac{m-1}{m}\right)^{\frac{1}{m-1}}\left(\lambda_\kappa+\kappa \|c_{\rho_\kappa}\| \cos\theta_x\right)^{\frac{1}{m-1}},\qquad&\text{ if }0\leq\theta_x\leq \phi_\kappa, \\[5pt]
0,\qquad&\text{ otherwise},
\end{cases} 
\end{equation}
with $\phi_\kappa$, $\|c_{\rho_\kappa}\|$ and $\lambda_\kappa$ uniquely determined by $\kappa$ -- see \eqref{eqn:phi-notation} and \eqref{eqn:system-cs-mg1}.

Also, the global energy minimizer is $\rhou$ (for $0<\kappa<\kappa_1$), $\rho_\kappa$ given by \eqref{eqn:rhok-fs} (for $\kappa_1<\kappa<\kappa_2$), or $\rho_\kappa$ given by \eqref{eqn:rhok-cs} (for $\kappa>\kappa_2$).
\end{proposition}

We summarize briefly the main ideas used to establish the existence of equilibria from Proposition \ref{prop:bif-mg1}.

{\em a) Equilibria with full support.} We denote 
\begin{equation}
\label{eqn:eta-notation}
\lambda=-\kappa \| c_\rho\| \eta.
\end{equation}
As $\lambda \geq \kappa \| c_\rho\|$ for these equilibria, $\eta$ lies in the range $\eta\leq -1$. With this notation, system \eqref{eqn:system-fs-mg1} has now to be solved for $\eta$ and $\|c_\rho\|$.
From \eqref{eqn:system-fs-mg1}, we can express $\| c_\rho\|$ and $\kappa$ in terms of $\eta$ as follows:
\begin{subequations}
\label{eqn:skappa-eta}
\begin{align}
 \| c_\rho\| &=\frac{\int_0^\pi (\cos\theta-\eta)^{\frac{1}{m-1}}\sin^{\dm-1}\theta \cos\theta\d\theta}{\int_0^\pi (\cos\theta-\eta)^{\frac{1}{m-1}}\sin^{\dm-1}\theta\d\theta}, \label{eqn:s-eta} \\[2pt]
 \kappa^{-1}&= H(\eta), \label{eqn:kappa-eta}
\end{align}
\end{subequations}
where $H(\eta)$ is defined for $\eta\leq -1$ by
\begin{equation}
\label{eqn:H}
H(\eta)= \frac{m-1}{m} (\dm w_\dm)^{m-1} \left(\int_0^\pi (\cos\theta-\eta)^{\frac{1}{m-1}}\sin^{\dm-1}\theta \cos\theta\d\theta\right)
\left(\int_0^\pi (\cos\theta-\eta)^{\frac{1}{m-1}}\sin^{\dm-1}\theta\d\theta\right)^{m-2}.
\end{equation}

In spite of its complicated expression, the function $H(\eta)$ has some remarkable monotonicity properties, as given by the following lemma.
\begin{lemma}\cite[Lemma 4.1]{FePaVa2025}
\label{lem:H-monotone}
For any dimension $\dm \geq 2$, the function $H(\eta)$ given by \eqref{eqn:H} is decreasing for $1<m<2$, is constant for $m=2$, and increasing for $m>2$.    
\end{lemma}

It can be shown \cite{FePaVa2025} that 
\begin{equation}
\label{eqn:Hlim}
\Bigl( \lim_{\eta \to -\infty } H(\eta)\Bigr)^{-1} = \kappa_1,
\end{equation}
where $\kappa_1$ is the critical $\kappa$ identified in Proposition \ref{prop:stab-unif}. Also denote
\begin{equation}
\label{eqn:kappa2}
\kappa_{2}:= (H(-1))^{-1}.
\end{equation}
This is the second critical value of $\kappa$ from Proposition \ref{prop:bif-mg1}. Note that $\kappa_1<\kappa_2$ since $H(\eta)$ is decreasing when $1<m<2$. Also, by Lemma \ref{lem:H-monotone} and equations  \eqref{eqn:Hlim} and \eqref{eqn:kappa2}, for any $\kappa \in (\kappa_{1},\kappa_{2})$, there exists a unique $\eta_\kappa <-1$ that satisfies \eqref{eqn:kappa-eta}. Corresponding to this $\eta_\kappa$, we then find $\|c_{\rho_\kappa}\|$ from \eqref{eqn:s-eta}, and set $\lambda_\kappa = -\kappa \|c_{\rho_\kappa}\| \eta_\kappa$. This determines uniquely a fully supported equilibrium $\rho_\kappa$ given by \eqref{eqn:rhok-fs}. 

\begin{remark}
\label{rmk:kappac}
As $\kappa \searrow \kappa_1$ (which corresponds to $\eta_\kappa \searrow -\infty$ and $\|c_{\rho_\kappa}\| \searrow 0$, along with $\lambda_\kappa \to \lambda_{\text{uni}}$ defined in \eqref{eqn:lambda-uni}), $\rho_\kappa$ approaches the uniform distribution.  At $\kappa = \kappa_1$, $\rho_\kappa$ gets born from the uniform distribution $\rhou$. 
\end{remark}
\medskip

{\em b) Equilibria with strict support.} Recall notation \eqref{eqn:phi-notation}. By writing $\lambda = - \kappa \|c_\rho\| \cos \phi $, the unknowns to be solved for in system \eqref{eqn:system-cs-mg1} are $\phi$ and $\|c_\rho\|$. System \eqref{eqn:system-cs-mg1} can be rewritten as
\begin{subequations}
\label{eqn:skappa-phi}
\begin{align}
 \| c_\rho \|&=\displaystyle\frac{\int_0^\phi (\cos\theta-\cos\phi)^{\frac{1}{m-1}}\sin^{\dm-1}\theta \cos\theta\d\theta}{\int_0^\phi (\cos\theta-\cos\phi)^{\frac{1}{m-1}}\sin^{\dm-1}\theta\d\theta}, \label{eqn:s-phi}\\[5pt]
 \kappa^{-1}& = F(\phi),  \label{eqn:kappa-phi}
\end{align}
\end{subequations}
where $F$ is a function defined on $0<\phi<\pi$ given by
\begin{equation}
\label{eqn:F}
F(\phi)= \frac{m-1}{m} (\dm w_\dm)^{m-1} \left(\int_0^\phi (\cos\theta-\cos\phi)^{\frac{1}{m-1}}\sin^{\dm-1}\theta \cos\theta\d\theta\right)\left(\int_0^\phi (\cos\theta-\cos\phi)^{\frac{1}{m-1}}\sin^{\dm-1}\theta \d\theta\right)^{m-2}.
\end{equation}
The monotonicity of the function $F$ has a key role on the bifurcation theory considered here. For any $1<m<2$ fixed, $F(\phi)$ is an increasing function on $(0,\pi)$ \cite[Lemma 4.3]{FePaVa2025}. Also, $F(0)=0$ and $F(\pi) = H(-1) = \kappa_2^{-1}$ (see \eqref{eqn:H} and \eqref{eqn:kappa2}), and hence the range of $F$ is $\left(0,\frac{1}{\kappa_2}\right)$. Therefore, equation \eqref{eqn:kappa-phi} has a unique solution $\phi_\kappa$ for every $\kappa_2<\kappa<\infty$. This leads to a unique strictly supported equilibrium $\rho_\kappa$ given by \eqref{eqn:rhok-cs}, where $\| c_{\rho_\kappa}\|$ is found from \eqref{eqn:s-phi} and $\lambda_\kappa = - \kappa \|c_{\rho_\kappa}\| \cos \phi_\kappa $, for every $\kappa_2<\kappa<\infty$. As $\kappa \nearrow \infty$, $\phi_\kappa \searrow 0$ and $\rho_\kappa$ approaches a Dirac delta concentration.

Given the information in Propositions \ref{prop:stab-unif} and \ref{prop:bif-mg1}, the stability of the equilibrium solutions can be inferred immediately. The uniform distribution is stable for $\kappa<\kappa_1$ and unstable for $\kappa>\kappa_1$. On the other hand, for $\kappa>\kappa_1$ there exists only one other equilibrium density $\rho_\kappa$, in addition to $\rhou$. The equilibrium $\rho_\kappa$ is a global minimizer, and hence stable. 

\begin{remark}\label{rmk:m12}
For $m$ in the range $1<m<2$, the bifurcation at $\kappa=\kappa_1$ is a supercritical pitchfork bifurcation, as the uniform distribution becomes unstable (see Proposition~\ref{prop:stab-unif}), while a branch of stable equilibria $\rho_\kappa$ emerges; see also Figure~\ref{fig:m12-srho}(a).
\end{remark}

An illustration is provided in Figure \ref{fig:m12-srho}. In Figure \ref{fig:m12-srho}(a) we plot the norms of the centre of mass of the equilibria; blue corresponds to $\rhou$, red to fully supported $\rho_\kappa$ ($\kappa_1<\kappa<\kappa_2$, see \eqref{eqn:rhok-fs}), and black to strictly supported $\rho_\kappa$ ($\kappa>\kappa_2$, see \eqref{eqn:rhok-cs}). Note the bifurcation at $\kappa = \kappa_1$, and the transition of $\rho_\kappa$ at $\kappa=\kappa_2$ from a fully supported equilibrium density to an equilibrium with strict support. Figure \ref{fig:m12-srho}(b) shows several profiles of the equilibrium densities $\rho_\kappa$ for various values of $\kappa$ (at $\kappa=\kappa_1$, $\rho_\kappa$ coincides with the uniform distribution). The equilibria concentrate as $\kappa$ increases. For these numerical results we used $m=1.2$ and $\dm = 4$, for which $\kappa_1 \approx 3.1196$ and $\kappa_2 \approx 4.3888$.

\begin{figure}[htbp]
 \begin{center}
 \begin{tabular}{cc}
 \includegraphics[width=0.45\textwidth]{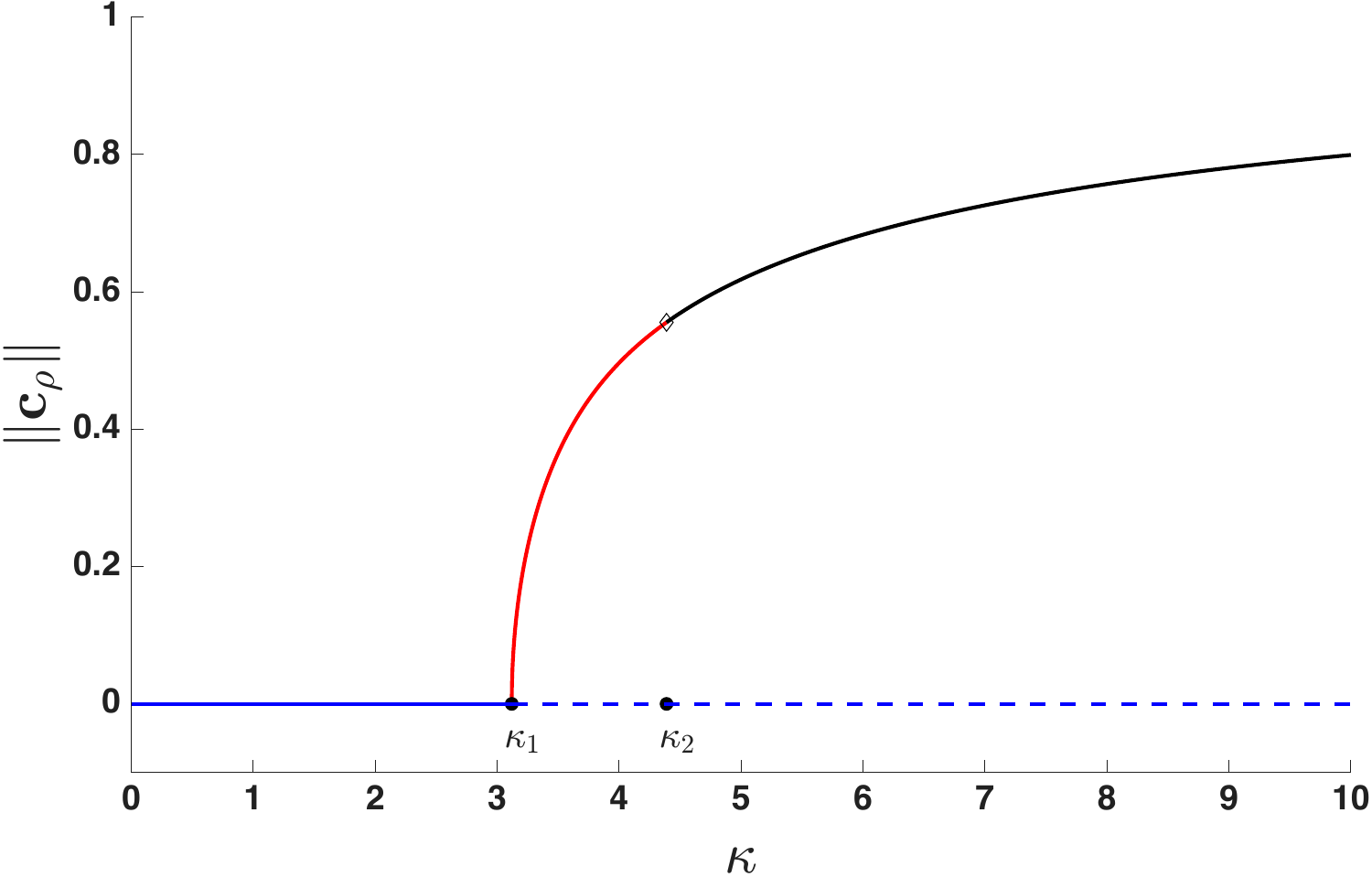} & 
 \includegraphics[width=0.48\textwidth]{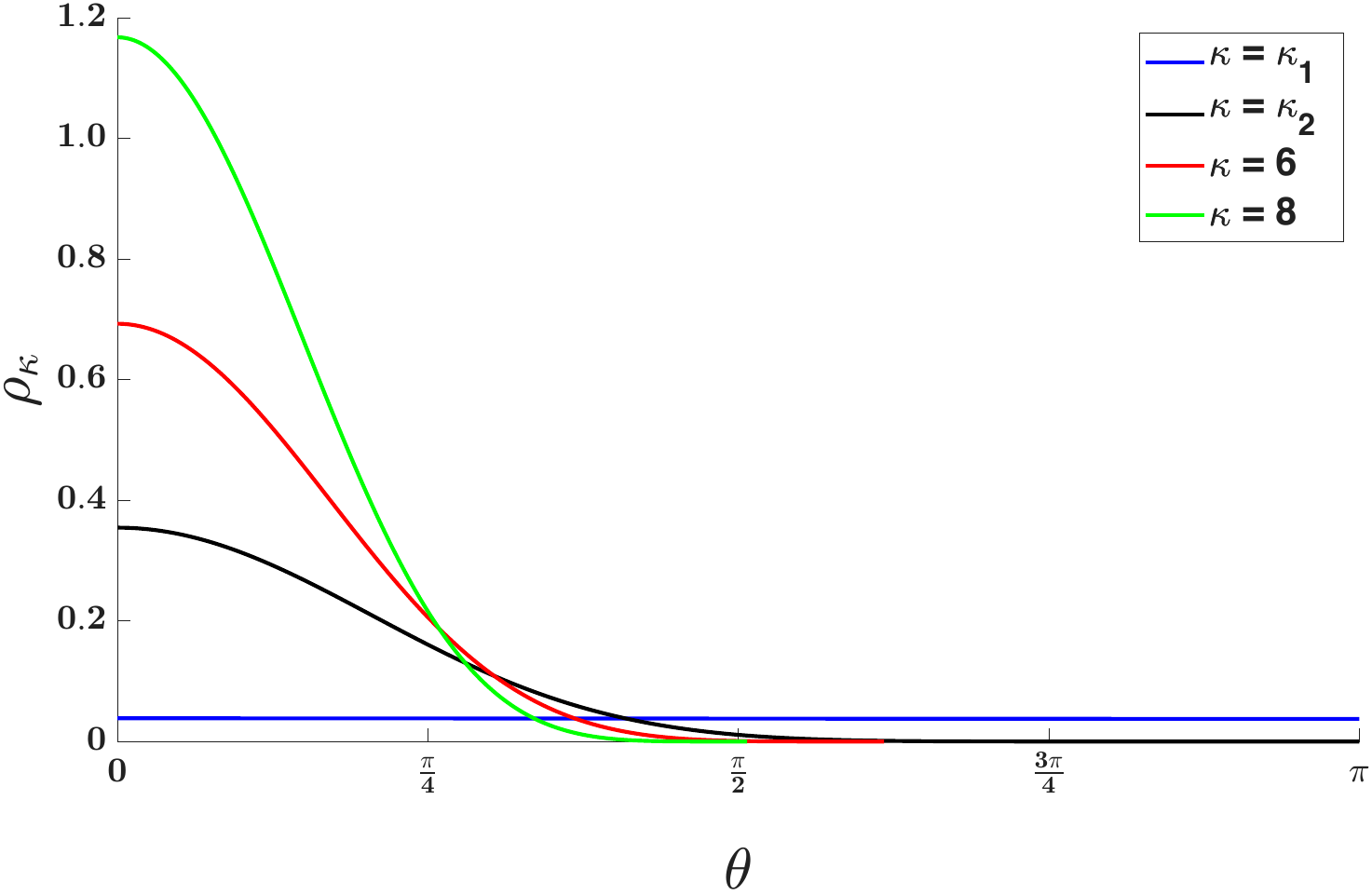} \\
 (a) & (b)
\end{tabular}
\caption{Case $1<m<2$. (a) Plot of the norm of the centre of mass of equilibria. Blue corresponds to $\rhou$, red to the fully supported $\rho_\kappa$ given by \eqref{eqn:rhok-fs}, and black to the strictly supported $\rho_\kappa$ from \eqref{eqn:rhok-cs}  -- see Proposition \ref{prop:bif-mg1}. At $\kappa=\kappa_1$, the fully supported equilibrium $\rho_\kappa$ emerges from the uniform distribution. At $\kappa=\kappa_2$, $\rho_\kappa$ changes from being fully supported to strictly supported on $\bbs^\dm$ -- this transition is indicated by a black diamond. (b) Plot of equilibria $\rho_\kappa$ (see \eqref{eqn:rhok-fs} and \eqref{eqn:rhok-cs}) for various values of $\kappa$. At $\kappa=\kappa_1$, the equilibrium is the uniform distribution. The equilibria $\rho_\kappa$ become more and more concentrated as $\kappa$ increases. We used $m=1.2$ and $\dm = 4$.}
\label{fig:m12-srho}
\end{center}
\end{figure}

\subsubsection{\underline{Case (ii) $m>2$}}
\label{subsubsect:prelim:mg2}
This is a significantly more complex case, which is in fact the focus of the present paper. In this case, only the stability of $\rhou$ is known, and the stability of the other equilibria needs to be established. 

The range $m>2$ was only considered in dimension $\dm=2$ in \cite{FePaVa2025}, due to technical difficulties in assessing the monotonicity of the function $F(\phi)$ for arbitrary dimensions. For $\dm =2$ and $m>2$, there exist three critical values $\kappa_3<\kappa_2<\kappa_1$ of $\kappa$. The critical value $\kappa_3$ is novel, as it does not exist for $1<m<2$. At $\kappa = \kappa_3$, a pair of strictly supported equilibria in the form \eqref{eqn:rhok-cs} gets born -- see Figure \ref{fig:m35-splot} for an illustration. One of them (black dashed line) transitions to a fully supported equilibrium at $\kappa=\kappa_2$, and then merges with the uniform distribution at $\kappa=\kappa_1$ (red dashed line in Figure \ref{fig:m35-splot}). On the other hand, the other equilibrium that gets born at $\kappa=\kappa_3$ (black solid line) undergoes no further transitions and exists for all $\kappa>\kappa_3$. The precise statement on the existence of equilibria is the following.

\begin{proposition} \cite[Proposition 6.1]{FePaVa2025}
\label{prop:bif-mg2}
For $\dm=2$ and any $m>2$, there exist three critical values of $\kappa$ ($\kappa_1>\kappa_2>\kappa_3$) such that:

\noindent i) At $\kappa=\kappa_3$, a pair of equilibria with support strictly contained in $\bbs^\dm$ gets born. For any $\kappa \in (\kappa_3,\kappa_2)$, there exist two solutions $0<\phi_{\kappa,1}<\phi_{\kappa,2}<\pi$ of \eqref{eqn:kappa-phi}; at $\kappa=\kappa_3$, the two solutions $\phi_{\kappa,1}$ and $\phi_{\kappa,2}$ coincide, with common value $\bar{\phi}$. Correspondingly, there exist two equilibria $\rho_{\kappa,1}$ and $\rho_{\kappa,2}$ of the form (see \eqref{eqn:equil-cs}):
\begin{equation}
\label{eqn:rhok-cs-mg2}
\rho_{\kappa,i}(x)=\begin{cases}
\left(\frac{m-1}{m}\right)^{\frac{1}{m-1}}\left(\lambda_{\kappa,i}+\kappa \|c_{\rho_{\kappa,i}}\| \cos\theta_x\right)^{\frac{1}{m-1}},\qquad&\text{ if }0\leq\theta_x\leq \phi_{\kappa,i},\\[5pt]
0,\qquad&\text{ otherwise},
\end{cases}
\end{equation}
where $\phi_{\kappa,i}$, $\|c_{\rho_{\kappa,i}}\|$ and $\lambda_{\kappa,i}$ are determined by $\kappa$ ($i=1,2$) -- see \eqref{eqn:skappa-phi}, and also \eqref{eqn:phi-notation}. 
\smallskip

\noindent ii) At $\kappa=\kappa_2$, the equilibrium $\rho_{\kappa,2}$ from part i) becomes fully supported on $\bbs^\dm$ (i.e., $\phi_{\kappa_2,2} = \pi$). For $\kappa_2<\kappa<\kappa_1$, $\rho_{\kappa,2}$ has the form (see  \eqref{eqn:equil-fs}):
\begin{equation}
\label{eqn:rhok-fs-mg2}
\rho_{\kappa,2}(x)=\left(\frac{m-1}{m}\right)^{\frac{1}{m-1}}\left(\lambda_{\kappa,2}+\kappa \|c_{\rho_{\kappa,2}} \|\cos\theta_x\right)^{\frac{1}{m-1}},\qquad\forall x\in \bbs^\dm,
\end{equation}
where $\|c_{\rho_{\kappa,2}}\|$ and $\lambda_{\kappa,2}$ are uniquely determined by $\kappa$ -- see \eqref{eqn:skappa-eta}, and also \eqref{eqn:eta-notation}.  On the other hand, $\rho_{\kappa,1}$ undergoes no transition at $\kappa=\kappa_2$, and retains the form \eqref{eqn:rhok-cs-mg2}.
\smallskip

\noindent iii) At $\kappa=\kappa_1$, the fully supported equilibrium $\rho_{\kappa,2}$ from part ii) merges with the uniform distribution $\rhou$, and vanishes. And again,  $\rho_{\kappa,1}$ undergoes no transition at $\kappa=\kappa_1$; as $\kappa$ increases to infinity, $\phi_{\kappa,1}$ approaches $0$ and $\rho_{\kappa,1}$ concentrates into a Dirac delta.
\end{proposition}

We summarize again the key ideas behind Proposition \ref{prop:bif-mg2}, distinguishing the fully and strictly supported equilibria.

{ \em a) Equilibria of full support.} With the same notations as in the case $1<m<2$, the problem reduces to finding solutions to \eqref{eqn:kappa-eta} in the range $\eta \leq -1$. Recall \eqref{eqn:Hlim} and the notation \eqref{eqn:kappa2}. For $m>2$, however, the function $H(\eta)$ is increasing, and hence, $\kappa_2<\kappa_1$. For any $\kappa_2<\kappa<\kappa_1$, \eqref{eqn:kappa-eta} has a unique solution -- see Figure \ref{fig:HF-mg2}(a) for an illustration. Together with \eqref{eqn:s-eta} and \eqref{eqn:eta-notation}, this determines a unique equilibrium given by \eqref{eqn:rhok-fs-mg2}. The main distinction from the case $1<m<2$ is that for $m>2$, $\rho_\kappa$ emerges from the uniform distribution as $\kappa$ {\em decreases} through $\kappa_1$. Here, $\rho_\kappa$ approaches $\rhou$ as $\kappa \nearrow \kappa_1$.

\begin{figure}[!htbp]
 \begin{center}
 \begin{tabular}{cc}
 \includegraphics[width=0.48\textwidth]{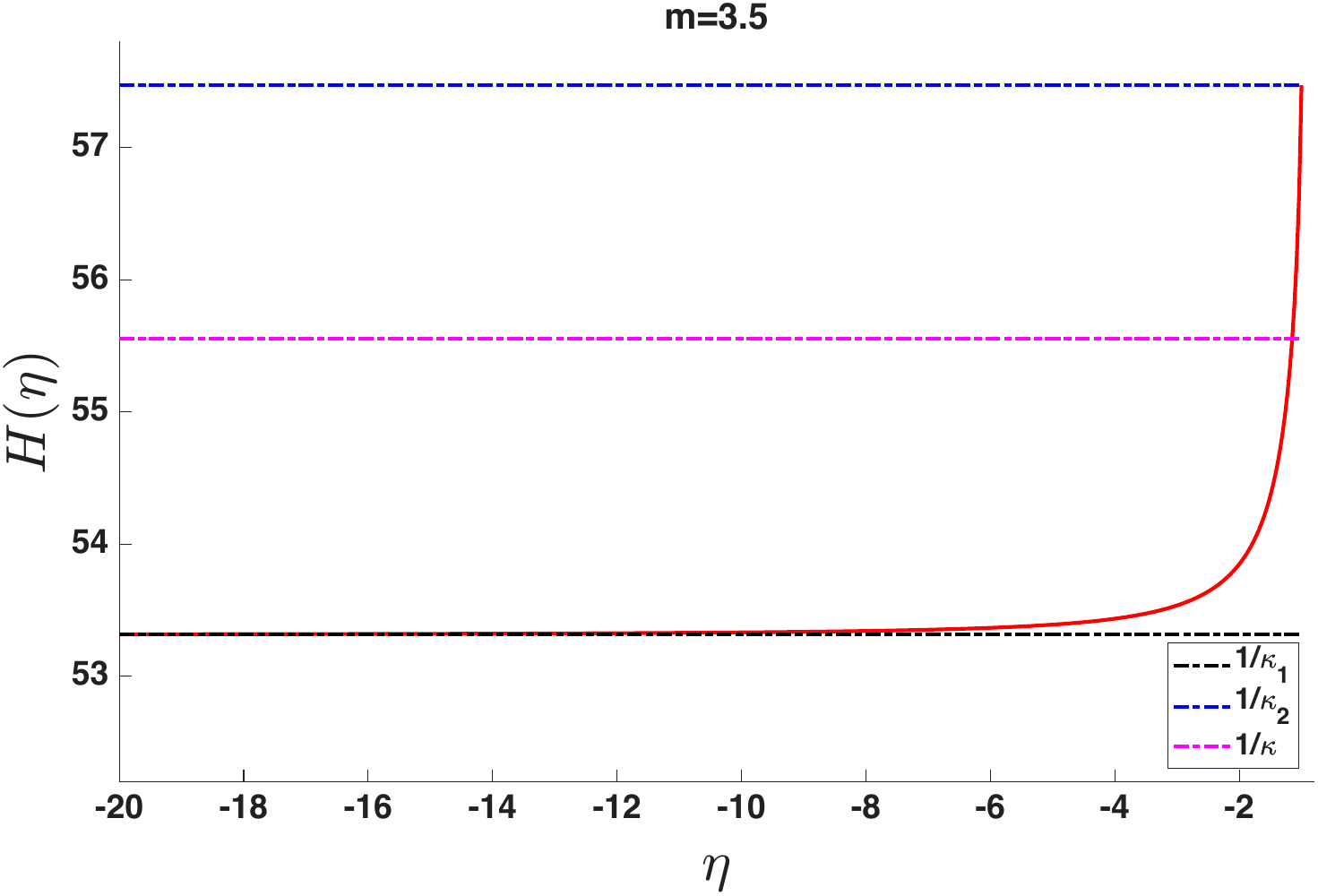} & 
 \includegraphics[width=0.48\textwidth]{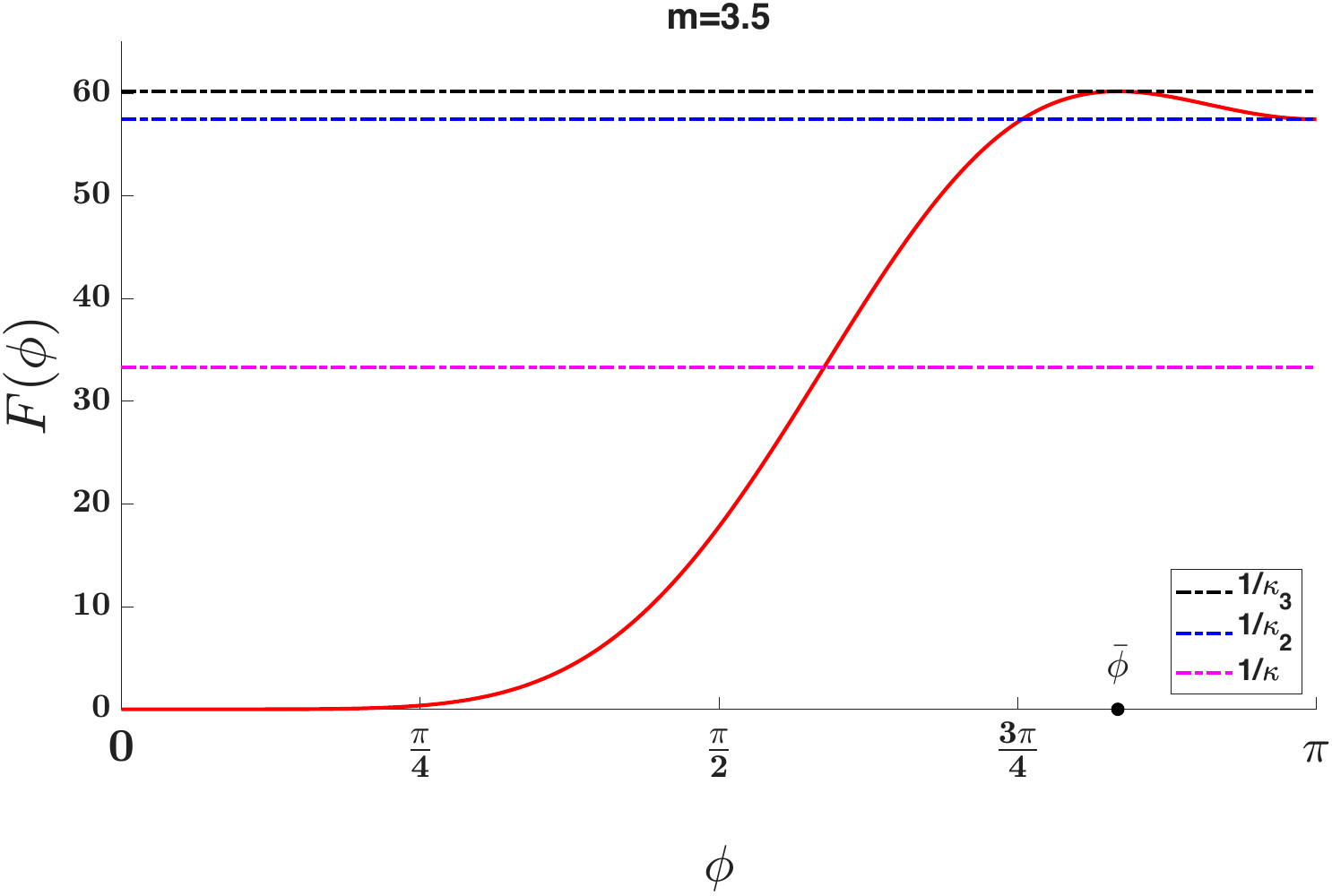} \\
 (a) & (b)
\end{tabular}
\caption{Case $\dm =2$, $m>2$. (a) Plot of function $H$ defined in \eqref{eqn:H}; for $m>2$, $H(\eta)$ is increasing (see Lemma \ref{lem:H-monotone}). For any $\kappa_2<\kappa<\kappa_1$, there exists a unique solution $\eta_\kappa <-1$ to \eqref{eqn:kappa-eta}. (b) Plot of function $F$ defined in \eqref{eqn:F}. The function changes monotonicity at $\bar{\phi}$, where it has a global maximum; see Lemma \ref{lem:F-monotone-mg2}. For $\kappa_3<\kappa<\kappa_2$, there exist two solutions $\phi_{\kappa,1} \in (0,\bar{\phi})$ and $\phi_{\kappa,2}\in(\bar{\phi},\pi)$ of \eqref{eqn:kappa-phi}. 
For $\kappa>\kappa_2$, there exists a unique solution $\phi_{\kappa,1} \in (0,\bar{\phi})$ of \eqref{eqn:kappa-phi}.  For both plots, $m=3.5$.}
\label{fig:HF-mg2}
\end{center}
\end{figure}

{ \em b) Equilibria of strict support.} Using again previous notations, one needs to solve first \eqref{eqn:kappa-phi} to find the possible values $\phi$ of the size of the support of the equilibria. For $m>2$, the monotonicity of the function $F(\phi)$ was not assessed in general however, as it was only established for dimension $\dm=2$. The result is the following.

\begin{lemma}\cite[Lemma 6.1]{FePaVa2025}
\label{lem:F-monotone-mg2}
Let $m>2$, $\dm=2$, and the function $F(\phi)$ defined on $0<\phi<\pi$ given by \eqref{eqn:F}. Then, there is $\bar{\phi} \in (0,\pi)$ such that $F$ is increasing on $(0,\bar{\phi})$ and decreasing on $(\bar{\phi}, \pi)$. 
\end{lemma}

The third critical value $\kappa_3$ from Proposition \ref{prop:bif-mg2} is given by
\begin{equation}
\label{eqn:kappa3}
\kappa_3:= (F(\bar{\phi}))^{-1}.
\end{equation}
By Lemma \ref{lem:F-monotone-mg2}, at $\kappa=\kappa_3$, the equation $F(\phi) = \kappa^{-1}$ has one solution given by $\bar{\phi}$ -- see the black dash-dotted line in Figure \ref{fig:HF-mg2}(b). Recall that by their expressions in \eqref{eqn:H} and \eqref{eqn:F}, the functions $H$ and $F$ satisfy 
\begin{equation}
\label{eqn:HF-rel}
H(-1)=F(\pi),
\end{equation}
and therefore $\kappa_2 = F(\pi)^{-1}$ (see \eqref{eqn:kappa2}). We then have $\kappa_2 > \kappa_3$, and for any $\kappa_3<\kappa<\kappa_2$, the equation \eqref{eqn:kappa-phi} has two solutions: $\phi_{\kappa,1} \in (0, \bar{\phi})$ and $\phi_{\kappa,2} \in (\bar{\phi},\pi)$. Also, for $\kappa >\kappa_2$, there exists only one solution $\phi_{\kappa,1} \in (0, \bar{\phi})$ of \eqref{eqn:kappa-phi} -- see the magenta dash-dotted line in Figure \ref{fig:HF-mg2}(b). As $\kappa$ increases to $\infty$, $\phi_{\kappa,1}$ decreases to $0$, and the corresponding equilibrium $\rho_{\kappa,1}$ approaches a delta Dirac.

 Figure \ref{fig:m35-splot} shows an illustration of the norm of the centres of mass of equilibria. It shows the emergence of $\rho_{\kappa,1}$ and $\rho_{\kappa,2}$ at $\kappa=\kappa_3$ (solid black and dashed black lines, respectively), the transition of $\rho_{\kappa,2}$ from fully supported to strictly supported at $\kappa=\kappa_2$ (note the black diamond), and the merging of the fully supported $\rho_{\kappa,2}$ (dashed red line) with $\rhou$ at $\kappa= \kappa_1$. We point out that similar to case $1<m<2$, a fully supported equilibrium emerges from the uniform distribution at $\kappa=\kappa_1$, and then it changes from fully supported to strictly supported at $\kappa=\kappa_2$, but now these transitions happen as $\kappa$ {\em decreases} through $\kappa_1$ and $\kappa_2$, respectively. The numerical results correspond to $m=3.5$ and $\dm=2$, for which $\kappa_3 \approx 0.0166$, $\kappa_2 \approx 0.0174$, and $\kappa_1 \approx 0.0188$.

\begin{figure}[htbp]
 \begin{center}
 \begin{tabular}{cc}
 \includegraphics[width=0.48\textwidth]{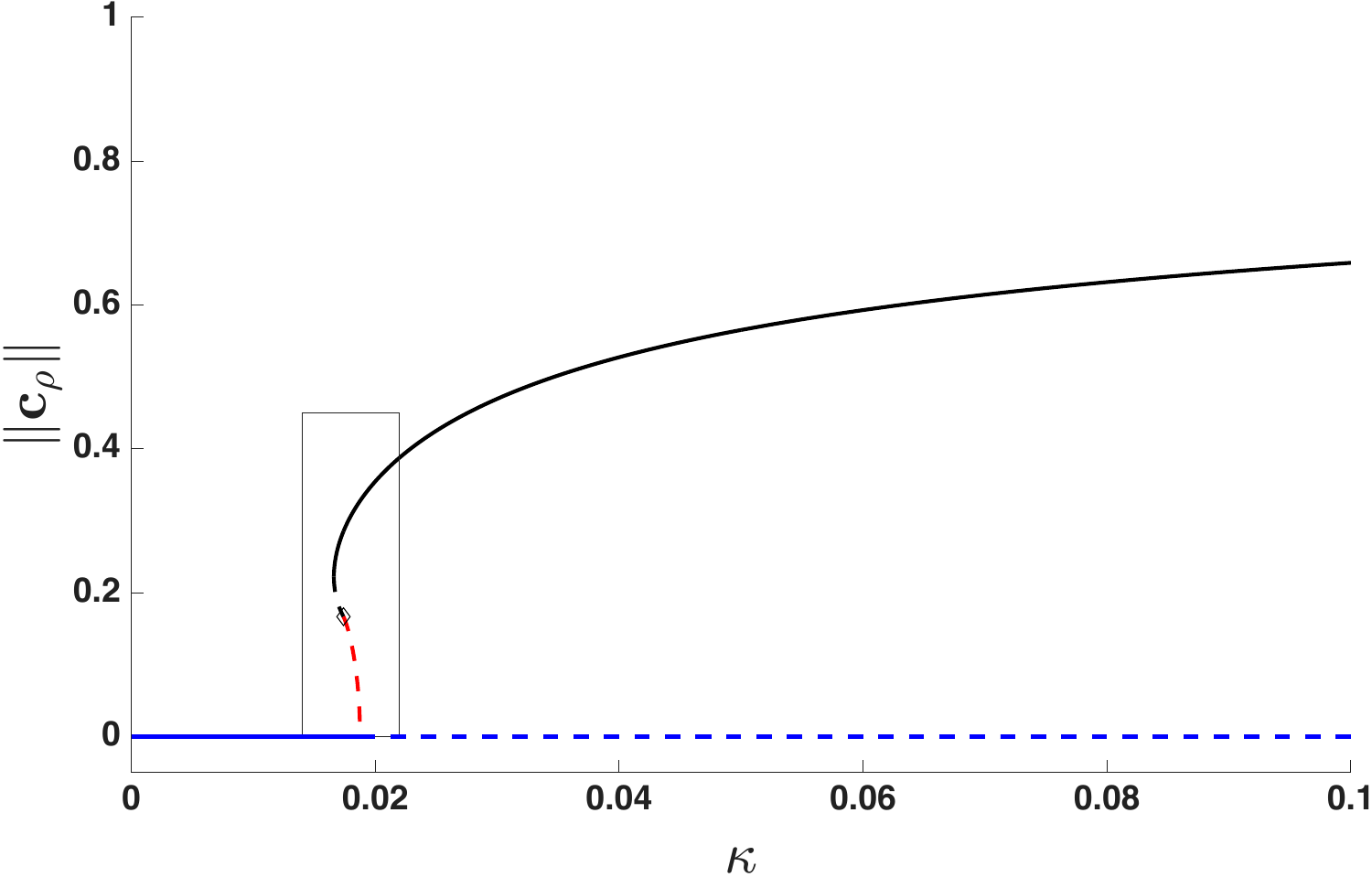} & 
 \includegraphics[width=0.48\textwidth]{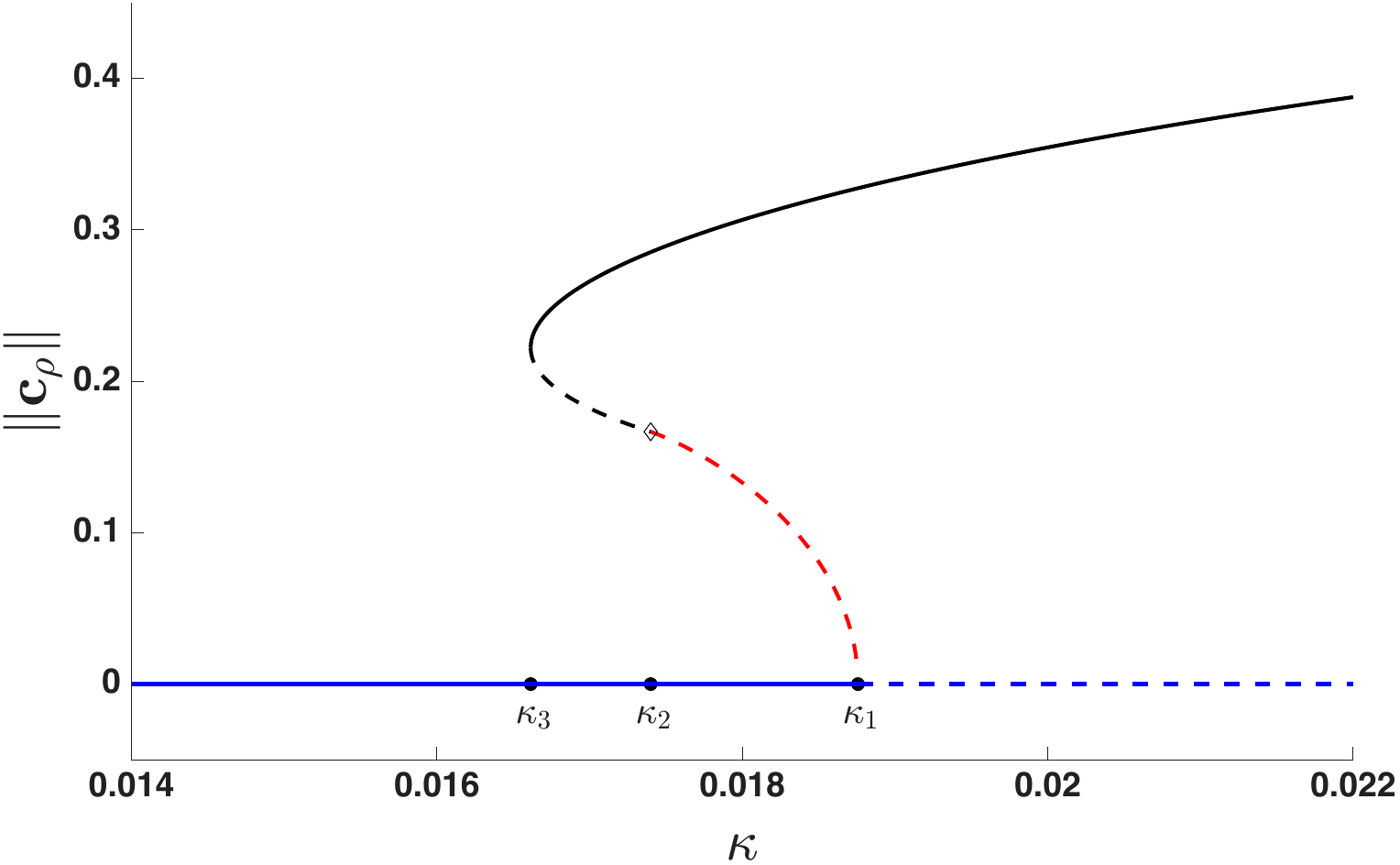} \\
 (a) $0<\kappa<0.1$ & (b) inset of rectangle from plot (a)
\end{tabular}
\caption{Case $\dm =2$, $m>2$. Plot of the norm of the centre of mass of  
$\rhou$ (blue), $\rho_{\kappa,1}$ (solid black) and $\rho_{\kappa,2}$ (dashed black for $\kappa_3<\kappa<\kappa_2$, and dashed red for $\kappa_2<\kappa<\kappa_1$) -- see Proposition \ref{prop:bif-mg2}. Plot (b) is the inset of the rectangle from plot (a). At $\kappa=\kappa_3$, a pair of equilibria ($\rho_{\kappa,1}$ and $\rho_{\kappa,2}$) with strict support in $\bbs^\dm$ gets born. The equilibrium $\rho_{\kappa,1}$ undergoes no further transitions and concentrates into a Dirac delta as $\kappa \to \infty$. The equilibrium $\rho_{\kappa,2}$ becomes fully supported at $\kappa = \kappa_2$ (transition indicated by a black diamond), and then merges with the uniform distribution at $\kappa=\kappa_1$. The numerical simulations correspond to $m=3.5$.}
\label{fig:m35-splot}
\end{center}
\end{figure}


\section{Existence of equilibria and bifurcation structure ($m>2$)}
\label{sect:existence-mg2}

From here on, this paper deals exclusively with the range $m>2$. In previous work \cite{FePaVa2025}, the existence and bifurcation structure of equilibria for the case $m>2$ were investigated; however, due to technical difficulties arising in general dimension, a complete classification was obtained only for the two-dimensional case ($\dm=2$). In this section, we extend these results to arbitrary dimensions $\dm\geq 2$ and provide a complete classification of equilibria for the case $m>2$.

The following theorem extends Proposition \ref{prop:bif-mg2} to arbitrary dimensions.

\begin{theorem}[Existence of equilibria]
\label{thm:equigenerald}
Proposition \ref{prop:bif-mg2} holds with an identical statement for general $\dm \geq 2$ and $m>2$.
\end{theorem}

As discussed in Section \ref{subsubsect:prelim:mg2}, the existence of fully supported equilibria reduces to finding solutions to \eqref{eqn:kappa-eta} in the range $\eta \leq -1$. The existence and uniqueness of such solutions  can be inferred from the monotonicity of the function $H(\eta)$ from \eqref{eqn:H}. Note that by Lemma \ref{lem:H-monotone} (see also Figure \ref{fig:HF-mg2}(a)), $H(\eta)$ is strictly increasing for $m>2$, in {\em any} dimension $\dm \geq 2$. For this reason, all considerations from Proposition \ref{prop:bif-mg2} that refer to fully supported equilibria extend immediately to general dimension.

To prove Theorem~\ref{thm:equigenerald}, it only remains to address the  equilibria of strict support in general dimension. The existence of such equilibria follows from solving \eqref{eqn:kappa-phi}, which yields the possible values $\phi$ of the size of their support. For $\dm=2$, the monotonicity of the function $F(\phi)$ was established in Lemma \ref{lem:F-monotone-mg2} (see also Figure \ref{fig:HF-mg2}(b)), which then enables straight away the conclusion of Proposition \ref{prop:bif-mg2}. Following an identical argument, to prove Theorem \ref{thm:equigenerald} it only remains to show that $F(\phi)$ has the same monotonicity properties in any $\dm \geq 2$ as it has for $\dm=2$. 

For convenience of calculations, we make the following notations throughout this paper. For an equilibrium $\rho$, we define
\begin{equation}
\label{eqn:ABC}
\begin{cases}
\displaystyle A(\rho)=\int_{\supp}\rho(x)^{2-m}\dx,\quad \\[2pt]
\displaystyle B(\rho)=\int_{\supp}\rho(x)^{2-m}(x_0\cdot x)\dx,\quad \\[2pt]
\displaystyle C(\rho)=\int_{\supp}\rho(x)^{2-m}(x_0\cdot x)^2\dx.
\end{cases}
\end{equation}
There are various relationships between $A(\rho)$, $B(\rho)$, and $C(\rho)$ that we derive and use in our analysis; the supporting calculations for these relationships are presented in Appendices~\ref{appendix:fs} and \ref{appendix:ps} for fully supported and strictly supported $\rho$, respectively. 

For the arguments that follow, it is more convenient to use notation \eqref{eqn:phi-notation} and write the strictly supported equilibrium \eqref{eqn:equil-cs}  as
\begin{align}
\label{eqn:equil-cs-eta}
\rho(x)=\begin{cases}
\displaystyle\left(\frac{(m-1)\kappa\|c_\rho\|}{m}\right)^{\frac{1}{m-1}}(\cos\theta_x-\cos\phi)^{\frac{1}{m-1}},&\quad\text {if }0\leq \theta_x\leq\phi,\\
0,&\quad \text{otherwise}.
\end{cases}
\end{align}
Furthermore, conditions \eqref{eqn:system-cs-mg1} can be rewritten as
\begin{subequations}
\label{eqn:equil-cs-1s}
\begin{align}
1&=\dm w_\dm \left(\frac{(m-1)\kappa \|c_\rho\|}{m}\right)^{\frac{1}{m-1}}\int_0^\phi\left(\cos\theta-\cos\phi\right)^{\frac{1}{m-1}}\sin^{\dm-1}\theta \, \d\theta, \label{eqn:equil-cs-1}\\
\|c_\rho\|& =\dm w_\dm \left(\frac{(m-1)\kappa \|c_\rho\|}{m}\right)^{\frac{1}{m-1}}\int_0^\phi\left(\cos\theta-\cos\phi\right)^{\frac{1}{m-1}}\sin^{\dm-1}\theta \cos\theta \, \d\theta.\label{eqn:equil-cs-s}
\end{align}
\end{subequations}

\begin{lemma}
\label{lemma:ABC-cs}
Let $\rho$ be a strictly supported equilibrium in the form \eqref{eqn:equil-cs-eta}. Then there exist the following relationships between $A(\rho)$, $B(\rho)$ and $C(\rho)$ defined in \eqref{eqn:ABC}:
\begin{align}
\label{ABC-system}
\begin{cases}
\displaystyle B(\rho)=\frac{m}{(m-1)\kappa\|c_\rho\|}+\cos\phi A(\rho),\vspace{0.3cm}\\
\displaystyle C(\rho)=\frac{m}{(m-1)\kappa}+\cos\phi B(\rho),\vspace{0.3cm}\\
\displaystyle A(\rho)-C(\rho)=\frac{\dm m}{\kappa}.
\end{cases}
\end{align}
Also, we can solve the above system of linear equations to get
\begin{align}\label{eqn:cs-A}
A(\rho)=\frac{m}{\kappa\sin^2\phi}\left(\frac{1}{m-1}+\frac{\cos\phi}{(m-1)\|c_\rho\|}+\dm\right).
\end{align}
\end{lemma}
\begin{proof}
See Appendix \ref{appendix:ps}.
\end{proof}

The following proposition is the generalization of Lemma \ref{lem:F-monotone-mg2} to general dimension, which is the missing ingredient to conclude Theorem \ref{thm:equigenerald}.

\begin{proposition}
\label{prop:sgnFp}
Let $m>2$, $\dm \geq 2$, and consider the function $F(\phi)$ defined on $0<\phi<\pi$, given by \eqref{eqn:F}. Then, there exists a unique $\bar{\phi} \in (0,\pi)$ such that $F'(\bar{\phi})=0$, where $F'(\phi)>0$ on $(0,\bar{\phi})$ and $F'(\phi)<0$ on $(\bar{\phi},\pi)$.
\end{proposition}
\begin{proof}
Using \eqref{eqn:F} we compute
\begin{equation}
\label{eqn:dlnFdtheta-1}
\begin{aligned}
\frac{\d}{\d\phi}\ln F(\phi)&=\frac{\sin\phi}{m-1}\left(\frac{\int_0^\phi (\cos\theta-\cos\phi)^{\frac{1}{m-1}-1}\sin^{\dm-1}\theta \cos\theta\d\theta}{\int_0^\phi (\cos\theta-\cos\phi)^{\frac{1}{m-1}}\sin^{\dm-1}\theta \cos\theta\d\theta} \right. \\
&\quad \left. +(m-2)\frac{\int_0^\phi (\cos\theta-\cos\phi)^{\frac{1}{m-1}-1}\sin^{\dm-1}\theta \d\theta}{\int_0^\phi (\cos\theta-\cos\phi)^{\frac{1}{m-1}}\sin^{\dm-1}\theta \d\theta}
\right)\\
&=\frac{\sin\phi}{m-1}\left(
\frac{B(\rho)}{C(\rho)-\cos\phi B(\rho)}+(m-2)\frac{A(\rho)}{B(\rho)-\cos\phi A(\rho)}
\right),
 \end{aligned}
 \end{equation}
where for the second equal sign we used \eqref{eqn:ABC} and \eqref{eqn:equil-cs-eta}; in particular for the denominators we wrote
\[
(\cos\theta-\cos\phi)^{\frac{1}{m-1}} = (\cos\theta-\cos\phi)^{\frac{2-m}{m-1}} \, (\cos\theta-\cos\phi).
\]
From \eqref{ABC-system}, we can further simplify \eqref{eqn:dlnFdtheta-1} into
\begin{align*}
\frac{\d}{\d\phi}\ln F(\phi)&=\frac{\sin\phi}{m-1}\left(
\frac{(m-1)\kappa B(\rho)}{m}+\frac{(m-2)(m-1)\kappa \|c_\rho\|A(\rho)}{m}
\right)\\
&=\frac{\kappa \sin\phi}{m}\left(B(\rho)+(m-2)\|c_\rho\|A(\rho)\right).
\end{align*}

Since $F(\phi)>0$ and $\frac{\kappa\sin\phi}{m}>0$, we infer that
\begin{equation}
\label{eqn:sgnF-1}
\mathrm{sgn}\left(F'(\phi)\right)=\mathrm{sgn}\left(B(\rho)+(m-2)\|c_\rho\|A(\rho)\right).
\end{equation}
Now, we use \eqref{ABC-system} (first equation) and \eqref{eqn:cs-A} to get
\begin{align*}
B(\rho)+(m-2)\|c_\rho\|A(\rho)&=\frac{m}{(m-1)\kappa\|c_\rho\|}+\cos\phi A(\rho)+(m-2)\|c_\rho\|A(\rho)\\
&=\frac{m}{(m-1)\kappa\|c_\rho\|}+\left(\cos\phi+(m-2)\|c_\rho\|\right)A(\rho)\\
&=\frac{m}{(m-1)\kappa\|c_\rho\|} \\
& \quad +\left(\cos\phi+(m-2)\|c_\rho\|\right)\times\frac{m}{\kappa\sin^2\phi}\left(\frac{1}{m-1}+\frac{\cos\phi}{(m-1)\|c_\rho\|}+\dm\right)\\
&=\frac{m}{(m-1)\kappa\|c_\rho\|\sin^2\phi} \\
& \quad \times \left(
\sin^2\phi+(\cos\phi+(m-2)\|c_\rho\|)\left(\|c_\rho\|+\cos\phi+\dm(m-1)\|c_\rho\|\right)
\right).
\end{align*}
Since $\frac{m}{(m-1)\kappa\|c_\rho\|\sin^2\phi}>0$, we combine the above with \eqref{eqn:sgnF-1} to get
\begin{equation}
\label{eqn:sgnF-2}
\mathrm{sgn}\left(F'(\phi)\right)=\mathrm{sgn}\left(
\sin^2\phi+(\cos\phi+(m-2)\|c_\rho\|)\left(\|c_\rho\|+\cos\phi+\dm(m-1)\|c_\rho\|\right)
\right).
\end{equation}

Next, we investigate the r.h.s. of \eqref{eqn:sgnF-2}. Some of the supporting calculations will be presented in Appendix \ref{appendix:signG}, but we summarize the main steps here. Dividing \eqref{eqn:equil-cs-s} by \eqref{eqn:equil-cs-1}, we find the following expression for$\|c_\rho\|$:
\begin{align}\label{def:s}
\| c_\rho \|= s(\phi) :=
\frac{ \int_0^\phi
(\cos\theta-\cos\phi)^{\frac{1}{m-1}}
\sin^{d-1}\theta \cos\theta\d\theta}
{\int_0^\phi
(\cos\theta-\cos\phi)^{\frac{1}{m-1}}
\sin^{d-1}\theta\d\theta}.
\end{align}
Also denote (see r.h.s. of \eqref{eqn:sgnF-2}):
\begin{equation}
\label{eqn:G}
G(\phi)=\sin^2\phi+
(\cos\phi+(m-2)s(\phi))
(s(\phi)+\cos\phi+d(m-1)s(\phi)),
\end{equation}
for all $\phi\in(0, \pi)$. 

We want to show that the equation $G(\phi)=0$ has a unique solution $\bar{\phi}\in(0,\pi)$ such that
\begin{equation}
\label{eqn:signG}
\begin{cases}
G(\phi)>0,\qquad\forall \phi\in(0, \bar{\phi}),\\
G(\phi)<0,\qquad\forall \phi\in(\bar{\phi}, \pi).
\end{cases}
\end{equation}
The claim in the proposition follows then from \eqref{eqn:signG} together with \eqref{eqn:sgnF-2}.

The key results proved in Appendix \ref{appendix:signG} are:

a) Lemma \ref{lem:negativeprime}: Let $G(\phi_0)=0$ for some $\phi_0\in(0, \pi)$. Then, $G'(\phi_0)<0$.

b) Lemma \ref{lem:Gnearpi}: The function $G(\phi)$ satisfies
\[
\lim_{\phi\to\pi-}\frac{G(\phi)}{\sin^2\phi}<0.
\]

\smallskip
With these two lemmas established, to prove \eqref{eqn:signG} we argue as follows.
\smallskip

\noindent {(\it Existence of a zero of $G$)} We know (see \eqref{slimvalue}) that
\[
\lim_{\phi\to0+}s(\phi)=1,
\]
which implies
\[
\lim_{\phi\to0+}G(\phi)=(m-1)(2+\dm(m-1))>0.
\]
Furthermore, by Lemma \ref{lem:Gnearpi}, there exists $\epsilon>0$ such that
\[
G(\phi)<0, \quad \text{ for all }\pi-\epsilon<\phi<\pi.
\]
Therefore, by the intermediate value property there exists $\phi_0\in(0, \pi)$ satisfying $G(\phi_0)=0$.\\

\noindent {(\it Uniqueness of the zero of $G$)} To prove that there is a  unique $\bar{\phi} \in (0,\pi)$ such that $G(\bar{\phi})=0$, let $Z$ be the set of zeros of $G$ in $(0,\pi)$. For any $\phi\in Z$, by Lemma \ref{lem:negativeprime}) we have
\[
G(\phi)=0\qquad\text{and}\qquad G'(\phi)<0.
\]
Therefore any element of $Z$ is an isolated point. For this reason, we can consider two consecutive elements $\phi_1 < \phi_2$ of $Z$, i.e., $Z\cap(\phi_1,\phi_2)=\emptyset$. Then, we have
\[
G(\phi_1)=G(\phi_2)=0,\quad G'(\phi_1)<0,\quad G'(\phi_2)<0.
\]
From the above, we infer that there exists $\delta\in(0, \phi_2-\phi_1)$ such that
\[
G(\phi)<0, \quad\forall \phi\in(\phi_1,\phi_1+\delta), \quad \text { and } \quad G(\phi)>0, \quad\forall \phi\in(\phi_2-\delta, \phi_2).
\]
By the intermediate value property, we know then that there exists $\phi_3\in(\phi_1+\delta, \phi_2-\delta)$ such that
\[
G(\phi_3)=0,
\]
which yields a contradiction. We can conclude that there exists a unique zero $\bar{\phi}\in(0, \pi)$ of $G$.

\noindent {(\it Concluding step)} From
\[
\lim_{\phi\to0+}G(\phi)>0, \quad \text{ the fact that } G(\phi)<0 \text{ for all } \phi\in(\pi-\epsilon,\pi),
\]
and the uniqueness of $\bar{\phi}$ satisfying $G(\bar{\phi})=0$, we conclude \eqref{eqn:signG}.
\end{proof}

Once the monotonicity of $F(\phi)$ has been established by Proposition \ref{prop:sgnFp}, the existence of strictly supported equilibria $\rho_{\kappa,1}$ and $\rho_{\kappa,2}$ (see statement of Proposition \ref{prop:bif-mg2}) can be inferred immediately. Indeed, recall that $\kappa_2 = (F(\pi))^{-1}$ and $\kappa_3 = (F(\bar{\phi}))^{-1}$, where $\kappa_2>\kappa_3$. For any $\kappa_3<\kappa<\kappa_2$, the equation \eqref{eqn:kappa-phi} has two solutions: $\phi_{\kappa,1} \in (0, \bar{\phi})$ and $\phi_{\kappa,2} \in (\bar{\phi},\pi)$ -- see Figure \ref{fig:HF-mg2}(b). On the other hand, for $\kappa >\kappa_2$, there exists only one solution $\phi_{\kappa,1} \in (0, \bar{\phi})$ of \eqref{eqn:kappa-phi}; see the magenta dash-dotted line in Figure \ref{fig:HF-mg2}(b). All the considerations above conclude now the proof of Theorem \ref{thm:equigenerald}.

\section{First order stability}
\label{sect:first-order}
In this section, we will investigate the first order stability of all possible equilibria. By \eqref{eqn:fvE}, an equilibrium density $\rho$ satisfies
\begin{align}\label{eqn:FV}
\frac{m}{m-1}\rho(x)^{m-1}-\kappa~ c_\rho\cdot x=\lambda_i,
\qquad
\forall x\in\mathcal{C}_i,\quad i\in\mathcal{I},
\end{align}
where $\{\mathcal{C}_i\}_{i\in\mathcal{I}}$ represents the collection of the connected components of $\supp$, i.e.,
\[
\bigcup_{i\in\mathcal{I}}\mathcal{C}_i=\supp,
\]
where each $\mathcal{C}_i$ is connected, and the sets $\mathcal{C}_i$ are disjoint.

The next lemma shows that unless all the values $\lambda_i$ are equal, the equilibrium is unstable.
\begin{lemma}
\label{lem:constant-lambda}
Let $\rho$ be an equilibrium satisfying \eqref{eqn:FV}. If there exists $j_1,j_2\in\mathcal{I}$ satisfying $\lambda_{j_1}\neq \lambda_{j_2}$, then $\rho$ is unstable.
\end{lemma}

\begin{proof}
Let assume that there are two distinct connected components $\mathcal{C}_1$ and $\mathcal{C}_2$ of $\supp$ such that
\[
\frac{m}{m-1}\rho(x)^{m-1}-\kappa~ c_\rho\cdot x=\lambda_i
\qquad
\forall x\in\mathcal{C}_i,\quad i=1, 2.
\]
Furthermore assume that $\lambda_1\neq \lambda_2$, and without loss of generality, assume $\lambda_1<\lambda_2$. 

Define
\begin{equation}
\label{eqn:Mi}
M_i:=\int_{\mathcal{C}_i}\rho(x)\dx, \qquad\forall i=1, 2.
\end{equation}
Also, set 
\begin{equation}
\label{eqn:psi-connected}
\psi(x)=\left(\frac{1}{M_1}\chi_{\mathcal{C}_1}(x)-\frac{1}{M_2}\chi_{\mathcal{C}_2}(x)\right)\rho(x),
\end{equation}
where $\chi_{\mathcal{C}_i}$ denotes the characteristic function of the set $\mathcal{C}_i$. 

Note that since $\mathcal{C}_1$ and $\mathcal{C}_2$ are disjoint, we have $\int_{\bbs^\dm} \psi(x) \dx =0$. Then, for any $0\leq \epsilon<\min(M_1, M_2)$, the density
\[
\rho^\epsilon=\rho+\epsilon\psi
\]
is a probability distribution. The effect of the perturbation $\psi$ is to transfer mass from the component with larger $\lambda$ to the one with smaller $\lambda$, i.e., it transfers mass from $\mathcal{C}_2$ to $\mathcal{C}_1$. We will show that the energy strictly decreases by this mass transfer.

From
\[
\partial_\epsilon\rho^\epsilon=\psi(x), \\
\]
along with \eqref{eqn:Mi} and \eqref{eqn:psi-connected}, we get
\begin{align*}
\frac{\d}{\d\epsilon}E[\rho^\epsilon]\bigg|_{\epsilon=0+}&=\int_{\supp}\frac{\delta E}{\delta\rho}(x) \left(\frac{1}{M_1}\chi_{\mathcal{C}_1}(x)-\frac{1}{M_2}\chi_{\mathcal{C}_2}(x)\right)\rho(x)\dx\\[3pt]
&=\frac{1}{M_1}\int_{\mathcal{C}_1}\lambda_1\rho(x)\dx-\frac{1}{M_2}\int_{\mathcal{C}_2}\lambda_2\rho(x)\dx\\[3pt]
&=\lambda_1-\lambda_2<0,
\end{align*}
where we also used
\[
\frac{\delta E}{\delta \rho}(x)=\lambda_i, \qquad\forall x\in \mathcal{C}_i, \quad i=1, 2.
\]
Since the energy decays strictly in this direction, we  conclude that $\rho$ is unstable.
\end{proof}

By Lemma \ref{lem:constant-lambda}, any disconnected equilibria with nonidentical values of $\lambda$ in its components, is unstable. Therefore, we restrict our attention to equilibria that satisfy \eqref{eqn:EL}, for some constant $\lambda$.

Now note that since an equilibrium $\rho$ satisfying \eqref{eqn:EL} must be non-negative on its support, then necessarily we have
\begin{align}\label{eqn:suppcondi}
\supp\subseteq\left\{x\in\bbs^\dm:\kappa \, c_\rho\cdot x+\lambda\ge0\right\}.
\end{align}

In the following lemma, we show that $\rho$ satisfying \eqref{eqn:EL} is unstable if the inclusion in \eqref{eqn:suppcondi} is strict.
\begin{lemma}
\label{lem:supportcondi}
Let $\rho$ be an equilibrium satisfying \eqref{eqn:EL} for which the inclusion in \eqref{eqn:suppcondi} is strict. Then $\rho$ is unstable.
\end{lemma}
\begin{proof}
Let $\rho$ be an equilibrium as in the assumption of the lemma, and define
\[
\mathcal{S} := \left\{x \in \bbs^\dm:\kappa\, c_\rho\cdot x+\lambda\ge0\right\}.
\]
Since the inclusion in \eqref{eqn:suppcondi} is strict, $\mathcal{S}\backslash \supp$ is a non-zero measure set. 

Consider a probability density function $ \displaystyle \psi\in\mathcal{P}_{ac}(\mathcal{S}\backslash \supp)$ such that
\begin{align}\label{eqn:supppsi}
\mathrm{supp}(\psi)\subset \left\{x\in\mathbb \bbs^\dm:\kappa \, c_\rho\cdot x+\lambda\ge\delta\right\},
\end{align}
for some $\delta>0$. Then define the family of probability densities:
\[
\rho^\epsilon=(1-\epsilon)\rho+\epsilon \psi, \qquad\forall \epsilon\in[0, 1).
\]
The effect of the perturbation above is to transfer mass from the support of $\rho$ into its complement. We will show that by making this mass transfer, the energy decreases.

As
\[
\partial_\epsilon \rho^\epsilon=-\rho+\psi,
\]
we compute
\begin{equation}
\label{eqn:dEdeps-connected}
\begin{aligned}
\frac{\d}{\d\epsilon}E[\rho^\epsilon]\bigg|_{\epsilon=0+}&=\int_{\bbs^\dm}\frac{\delta E}{\delta\rho}(x) \left(-\rho(x)+\psi(x)\right)\dx\\[2pt]
&=-\int_{\supp}\frac{\delta E}{\delta\rho}(x)\rho(x)\dx+\int_{\bbs^\dm \setminus \supp }\frac{\delta E}{\delta\rho}(x)\psi(x)\dx\\[2pt]
&=-\lambda -\int_{\bbs^\dm}\kappa \, c_\rho\cdot x \, \psi(x)\dx,
\end{aligned}
\end{equation}
where for the last equal sign we used $\frac{\delta E}{\delta\rho}(x) = \lambda$ on $\supp$, and  $\frac{\delta E}{\delta\rho}(x) = - \kappa \, c_\rho\cdot x$ on its complement $\bbs^\dm \setminus \supp$ -- see \eqref{eqn:fvE} and \eqref{eqn:EL}, and also note that $\psi$ is supported on $\bbs^\dm \setminus \supp$. 

Furthermore, by \eqref{eqn:supppsi} we have
\begin{equation*}
\kappa \, c_\rho\cdot x\geq -\lambda+\delta, \qquad\forall x\in\mathrm{supp}(\psi).
\end{equation*}
Now multiply the above by $\psi$ and integrate on $\bbs^\dm$ (note that $\psi$ is a probability density function, hence it is non-negative), to get
\begin{equation}
\label{eqn:glpd}
\begin{aligned}
\int_{\bbs^\dm}\kappa \, c_\rho\cdot x \, \psi(x)\dx &\geq (-\lambda+\delta)\int_{\bbs^\dm}\psi(x)\dx\\
&=-\lambda+\delta.
\end{aligned}
\end{equation}

Finally, combine \eqref{eqn:dEdeps-connected} and \eqref{eqn:glpd} to find
\[
\frac{\d}{\d\epsilon}E[\rho^\epsilon]\bigg|_{\epsilon=0+}\leq -\lambda-(-\lambda+\delta)=-\delta<0.
\]
Since the energy decreases, we infer that $\rho$ is unstable.
\end{proof}

Due to Lemmas \ref{lem:constant-lambda} and \ref{lem:supportcondi}, we only need to investigate the stability of equilibria $\rho$ that satisfy \eqref{eqn:EL} and have connected support given by \eqref{eqn:support}. These are the equilibria identified in \cite{FePaVa2025}, as given by \eqref{eqn:equil-fs} or \eqref{eqn:equil-cs}, depending on whether they are supported on the full sphere or on a strict subset of it. When convenient, we will refer to these equilibria in the common notation \eqref{eqn:equil}.

Consider now an equilibrium $\rho$ in the form \eqref{eqn:equil}. For a fixed positive number $\epsilon_0>0$, take a family of perturbed states $\rho^\epsilon \in \calP_{ac} (\bbs^\dm)$ in the form
\begin{equation}
\label{eqn:rhoeps}
\rho^\epsilon=\rho+\epsilon\psi,\qquad  0\leq \epsilon\leq\epsilon_0,
\end{equation}
where $\psi \in L^1(\bbs^\dm)$ satisfies:
\begin{equation}
\label{eqn:psi-L1}
\int_{\bbs^\dm}\psi(x)\dx=0,
\end{equation}
together with
\begin{equation}
\label{eqn:rhoeps-pos}
\psi(y)\geq-\frac{1}{\epsilon_0}\rho(y),\quad \forall 
 y\in\supp, \qquad \text{ and } \qquad \psi(y)\geq0, \quad\forall  y\in\supp^c.
\end{equation}
The first condition guarantees that $\rho^\epsilon$ has unit mass, while the other two conditions are needed for $\rho^\epsilon$ to be non-negative both in the support and outside the support of $\rho$.

\begin{proposition}[First-order stability]
\label{prop:fo-stab}
Let $\rho$ be an equilibrium in the form \eqref{eqn:equil} and consider general linear perturbations in the form \eqref{eqn:rhoeps}, where $\psi \in L^1(\bbs^\dm)$ satisfies \eqref{eqn:psi-L1} and \eqref{eqn:rhoeps-pos}. Then, $\rho$ is stable with respect to perturbations $\psi$ for which $\mathrm{supp}(\psi)$ is not a subset of $\supp$.
\end{proposition}
\begin{proof}
Using \eqref{eqn:energy-s}, we compute
\begin{align}\label{First-var}
\frac{\d}{\d\epsilon}E[\rho^\epsilon]=\frac{m}{m-1}\int_{\bbs^\dm}\rho^\epsilon(x)^{m-1}\psi(x)\dx-\kappa \, c_{\rho^\epsilon} \cdot \int_{\bbs^\dm} x \psi(x) \dx,
\end{align}
and then evaluate \eqref{First-var} at $\epsilon = 0+$, to find
\begin{equation}
\label{eqn:dEdeps-0}
\frac{\d}{\d\epsilon}E[\rho^\epsilon]\bigg|_{\epsilon=0+}=\frac{m}{m-1}\int_{\bbs^\dm}\rho(x)^{m-1}\psi(x)\dx-\kappa \, c_{\rho} \cdot \int_{\bbs^\dm} x \psi(x) \dx.
\end{equation}
We will investigate the sign of the r.h.s. of \eqref{eqn:dEdeps-0}. 

Using \eqref{eqn:EL}, we compute the r.h.s. of \eqref{eqn:dEdeps-0} as
\begin{align*}
&\frac{m}{m-1}\int_{\bbs^\dm}\rho(x)^{m-1}\psi(x)\dx-\kappa \, c_{\rho} \cdot \int_{\bbs^\dm} x \psi(x) \dx \\[2pt]
& \hspace{2cm} =\int_{\bbs^\dm}\left(\frac{m}{m-1}\rho(x)^{m-1}-\kappa \, c_\rho\cdot x\right)\psi(x)\dx\\[2pt]
&\hspace{2cm} =\int_{\supp}\left(\frac{m}{m-1}\rho(x)^{m-1}-\kappa \, c_\rho\cdot x\right)\psi(x)\dx \\
&\hspace{2cm} \quad +\int_{\supp^c}\left(\frac{m}{m-1}\rho(x)^{m-1}-\kappa \, c_\rho\cdot x\right)\psi(x)\dx\\
&\hspace{2cm} =\int_{\supp}\lambda \psi(x)\dx+\int_{\supp^c}\left(\frac{m}{m-1}\rho(x)^{m-1}-\kappa \, c_\rho\cdot x\right)\psi(x)\dx\\
&\hspace{2cm} =\int_{\supp^c}\left(\frac{m}{m-1}\rho(x)^{m-1}-\kappa \, c_\rho\cdot x-\lambda\right)\psi(x)\dx,
\end{align*}
where in the last equality we used
\[
0=\int_{\bbs^\dm}\psi(x)\dx=\int_{\supp}\psi(x)\dx+\int_{\supp^c}\psi(x)\dx.
\]

For any $ x \in \supp^c$ we have $\rho(x)=0$, $-\kappa \, c_\rho\cdot x -\lambda>0$ (see \eqref{eqn:support}), and $\psi(x) \geq 0$. Hence, provided $\mathrm{supp}(\psi)$ is not a subset of $\supp$, then $\frac{\d}{\d\epsilon}E[\rho^\epsilon]\big|_{\epsilon=0+}>0$ and $\rho$ is stable with respect to the perturbation $\psi$. 
\end{proof}

By the calculation above, if $\mathrm{supp}(\psi)\subset\supp$ then the first variation \eqref{eqn:dEdeps-0} vanishes, and the stability of $\rho$ needs to be determined by higher-order variations. We will investigate this in the next section.

\section{Second order stability}
\label{sect:2nd-order}
For an equilibrium $\rho$ in the form \eqref{eqn:equil}, we consider here only perturbations in the form \eqref{eqn:rhoeps} with  
$\mathrm{supp}(\psi)\subset \mathrm{supp}(\rho)$.
The set of admissible perturbations $\psi$ is denoted by
\begin{equation}
\label{eqn:calA}
\mathcal{A}:=\left\{\psi \in L^1:\mathrm{supp}(\psi)\subset \supp,\quad\int_{\supp}\psi(x)\dx=0,\quad \psi(x)\geq-\frac{1}{\epsilon_0}\rho(x)\right\}.
\end{equation}

From \eqref{First-var}, we have
\begin{equation*}
\frac{\d^2}{\d\epsilon^2}E[\rho^\epsilon]=m\int_{\bbs^\dm}\rho^\epsilon(x)^{m-2}\psi(x)^2\dx-\kappa \left \| \int_{\bbs^\dm} x \psi(x) \dx \right \|^2,
\end{equation*}
and hence,
\begin{equation}
\label{eqn:d2Edeps-0}
\left.\frac{\d^2}{\d\epsilon^2}E[\rho^\epsilon]\right|_{\epsilon=0+}=m\int_{\bbs^\dm}\rho(x)^{m-2}\psi(x)^2\dx-\kappa \left \| \int_{\bbs^\dm} x \psi(x) \dx \right \|^2.
\end{equation}

To establish the stability of the equilibrium $\rho$ we need to assess the sign of \eqref{eqn:d2Edeps-0}. For this purpose, we define the following functional:
\begin{equation}
\label{eqn:calF}
\mathcal{F}[\psi]=\frac{\int_{\supp}\rho(x)^{m-2}\psi(x)^2\dx}{\left \| \int_{\bbs^\dm} x \psi(x) \dx \right \|^2},
\end{equation}
and look into minimizing $\mathcal{F}[\psi]$ over $\psi \in \mathcal{A}$. However, for reasons which we will explain later (see Remark \ref{rmk:minA}), instead of finding directly the minimum of $\mathcal{F}[\psi]$ on $\mathcal{A}$, we investigate the minimum of the functional in the larger set $\tilde{\mathcal{A}}$ defined by
\begin{equation}
\label{eqn:calAtilde}
\tilde{\mathcal{A}}:=\left\{\psi \in L^1:\mathrm{supp}(\psi)\subset \supp,\quad\int_{\supp}\psi(x)\dx=0\right\}.
\end{equation}
Note that functions in the larger class $\tilde{\mathcal{A}}$ are no longer required to be bounded below by $-\frac{1}{\epsilon_0}\rho(x)$.

\subsection{Critical points of the functional $\calF[\psi]$} We look first into the critical points of the functional $\mathcal{F}$. Consider perturbations $\psi^\epsilon = \psi + \epsilon \delta \psi$, and set
\begin{equation}
\label{eqn:fvar-calF}
0= \left. \frac{\d}{\d\epsilon}\mathcal{F}[\psi+\epsilon\delta\psi] \right |_{\epsilon =0}=\int_{\supp}\frac{\delta\mathcal{F}[\psi]}{\delta\psi(x)}\delta\psi(x)\dx,
\end{equation}
where the functional derivative of $\mathcal{F}[\psi]$ is computed as
\[
\frac{\delta\mathcal{F}}{\delta\psi}(x)=\frac{2\rho(x)^{m-2}\psi(x)}{\left \| \int_{\bbs^\dm} y \psi(y) \dy \right \|^2}- \left( \int_{\supp}\rho(y)^{m-2}\psi(y)^2\dy \right) \frac{2 \int_{\bbs^\dm} y \psi(y) \dy \cdot x}{\left \| \int_{\bbs^\dm} y \psi(y) \dy \right \|^4}.
\]
Since admissible perturbations have zero mass, then necessarily $\int_{\supp} \delta \psi(x) \dx =0$.

Note that for a critical point $\psi$ of $\mathcal{F}$, it holds that $\frac{\delta\mathcal{F}[\psi]}{\delta\psi}(x)$ is constant on $\supp$. Indeed, assuming the contrary, one can choose two subsets $A_1, A_2\subset \supp$ satisfying $\frac{\delta\mathcal{F}[\psi]}{\delta\psi(x_1)}<\frac{\delta\mathcal{F}[\psi]}{\delta\psi(x_2)}$ for any $x_1\in A_1$ and $x_2\in A_2$. Then, take
\[
\delta\psi(x)=\frac{1}{|A_1|}\chi_{A_1}(x)-\frac{1}{|A_2|}\chi_{A_2}(x), \qquad \forall~x \in \supp,
\]
and find
\[
\int_{\supp}\frac{\delta\mathcal{F}[\psi]}{\delta\psi(x)}\delta\psi(x)\dx<0,
\]
which contradicts \eqref{eqn:fvar-calF}.

For convenience of notations, denote by $\mu_{\psi}$ the non-normalized first moment of $\psi$:
\[
\mu_{\psi} = \int_{\bbs^\dm} x \psi(x) \dx.
\]
By the argument above, for a critical point $\psi$ of $\mathcal{F}$, we have 
\begin{equation}
\label{eqn:EL-calF}
\frac{2\rho(x)^{m-2}\psi(x)}{\|\mu_{\psi}\|^2}-\int_{\supp}\rho(y)^{m-2}\psi(y)^2\dy \, \frac{2\mu_{\psi}\cdot x}{\|\mu_{\psi}\|^4}=2\Lambda,\qquad x \in \supp, 
\end{equation}
for some constant $\Lambda$. This yields
\begin{equation}
\label{eqn:phi-1}
\psi(x)=\left(\Lambda+\frac{\mu_{\psi}\cdot x}{\|\mu_{\psi}\|^4}\int_{\supp}\rho(y)^{m-2}\psi(y)^2\dy \right)\|\mu_{\psi}\|^2\rho(x)^{2-m}, \qquad x \in \supp.
\end{equation}

As $\psi$ has zero mass, we get
\begin{align*}
0&=\int_{\supp}\psi(x)\dx\\
&=\Lambda \|\mu_{\psi}\|^2\int_{\supp}\rho(x)^{2-m}\dx+\frac{1}{\|\mu_{\psi}\|^2}\int_{\supp}\rho(x)^{m-2}\psi(x)^2\dx\int_{\supp}\mu_{\psi}\cdot x \, \rho(x)^{2-m}\dx.
\end{align*}

Then, we can express $\Lambda$ as 
\begin{align*}
\Lambda&=-\frac{1}{\|\mu_{\psi}\|^4}\frac{\int_{\supp}\rho(x)^{m-2}\psi(x)^2\dx\int_{\supp}\mu_{\psi}\cdot x \, \rho(x)^{2-m}\dx}{\int_{\supp}\rho(x)^{2-m}\dx}\\[2pt]
&=-\frac{\mu_{\psi}\cdot c_{\rho^{2-m}}}{\|\mu_{\psi}\|^4}\int_{\supp}\rho(y)^{m-2}\psi(y)^2\dy,
\end{align*}
which used in \eqref{eqn:phi-1} gives
\begin{equation}
\label{eqn:phi-2}
\psi(x)=\frac{\rho(x)^{2-m}\mu_{\psi}\cdot(x-c_{\rho^{2-m}})}{\|\mu_{\psi}\|^2}\int_{\supp}\rho(y)^{m-2}\psi(y)^2\dy.
\end{equation}
Note that here, 
\begin{equation}
\label{eqn:crhomm2}
c_{\rho^{2-m}}=\frac{\int_{\supp} x \rho^{2-m}(x) \dx}{\int_{\supp}\rho^{2-m}(x)\dx}
\end{equation}
is the centre of mass of the density $\rho^{2-m}$ -- see \eqref{eqn:CM-general}. 

\begin{remark} The integrability of $\rho^{2-m}$ follows from the specific expression of $\rho$. We distinguish two cases:

a) $\rho$ is fully supported, as given by \eqref{eqn:equil-fs}. Unless $\rho$ vanishes at a single point (i.e., the case $\lambda=\kappa\|c_\rho\|$), $\rho^{2-m}$ is well defined and integrable on the whole domain. If $\rho$ vanishes at a single point, then substituting $\lambda=\kappa\|c_\rho\|$ into \eqref{eqn:equil-fs} yields
\[
\rho(x)=\left(\frac{(m-1)\kappa\|c_\rho\|}{m}\right)^{\frac{1}{m-1}}(1+\cos\theta_x)^{\frac{1}{m-1}},\quad x\in\bbs^\dm.
\]
To verify the integrability of $\rho^{2-m}$, we compute
\[
\int_{\bbs^\dm}(1+\cos\theta_x)^{\frac{2-m}{m-1}}\dx
=
\dm w_\dm \int_0^\pi (1+\cos\theta)^{\frac{2-m}{m-1}}\sin^{\dm-1}\theta\d\theta.
\]
Using $1+\cos\theta=2\cos^2\left(\frac{\theta}{2}\right)$ and $\sin\theta=2\sin\left(\frac{\theta}{2}\right)\cos\left(\frac{\theta}{2}\right)$, the above integral can be simplified, up to a multiplicative constant, as
\begin{align*}\int_0^\pi
\cos^{2\left(\frac{2-m}{m-1}\right)+\dm-1}
\left(\frac{\theta}{2}\right)
\sin^{\dm-1}\left(\frac{\theta}{2}\right)\d\theta
=\int_0^\pi
\sin^{2\left(\frac{2-m}{m-1}\right)+\dm-1}
\left(\frac{\theta}{2}\right)
\cos^{\dm-1}\left(\frac{\theta}{2}\right)\d\theta,
\end{align*}
where we also used the change of variable $ \theta\to\pi-\theta$.

The above integral is finite if and only if
\[
2\left(\frac{2-m}{m-1}\right)+\dm-1>-1,
\]
which is equivalent to
\[
\frac{2}{m-1}>2-\dm.
\]
Therefore, if $\dm\geq2$ and $m>1$, then $\rho^{2-m}$ is integrable in this case.
\smallskip

b) $\rho$ is partially supported, as given by \eqref{eqn:equil-cs}. Using \eqref{eqn:phi-notation}, we can write $\rho$ as 
\[
\rho(x)=\begin{cases}
\left(\frac{(m-1)\kappa\|c_\rho\|}{m}\right)^{\frac{1}{m-1}}\left(\cos\theta_x-\cos\phi\right)^{\frac{1}{m-1}},\qquad&\text{ if }0\leq\theta_x\leq \phi,\\[5pt]
0,\qquad&\text{ otherwise}.
\end{cases}
\]
To verify the integrability of $\rho^{2-m}$, we compute
\[
\int_{\supp}(\cos\theta_x-\cos\phi)^{\frac{2-m}{m-1}}\dx=\dm w_\dm
\int_0^\phi
(\cos\theta-\cos\phi)^{\frac{2-m}{m-1}}
\sin^{\dm-1}\theta\d\theta.
\]
Recall that $0<\phi<\pi$ in this case. Since the only possible singularity of the integrand occurs at $\theta=\phi$, we introduce the change of variable
\[
\tilde{\theta}=\phi-\theta,\qquad
0\le\theta\le\phi.
\]
Then the above integral becomes
\[
\int_0^\phi
(\cos\theta-\cos\phi)^{\frac{2-m}{m-1}}
\sin^{\dm-1}\theta\d\theta
=
\int_0^\phi
\left(\cos(\phi-\tilde{\theta})-\cos\phi\right)^{\frac{2-m}{m-1}}
\sin^{\dm-1}(\phi-\tilde{\theta})\d\tilde{\theta}.
\]
Since
\[
\cos(\phi-\tilde{\theta})-\cos\phi
=
\cos\phi(\cos\tilde{\theta}-1)+\sin\phi\sin\tilde{\theta},
\]
we have
\[
\cos(\phi-\tilde{\theta})-\cos\phi=O(\tilde{\theta}),
\qquad\text{as }\tilde{\theta}\searrow 0.
\]
Therefore, the above integral is finite if and only if
\[
\frac{2-m}{m-1}>-1,
\]
which is equivalent to $m>1$. Therefore, $\rho^{2-m}$ is also integrable in this case.
\end{remark}

Now, square \eqref{eqn:phi-2}, multiply the result by $\rho(x)^{m-2}$ and integrate, to get
\begin{multline}
\int_{\supp}\rho(x)^{m-2}\psi(x)^2\dx= \\ 
\frac{1}{\|\mu_{\psi}\|^4}\left(\int_{\supp}\rho(x)^{m-2}\psi(x)^2\dx\right)^2 \int_{\supp}\rho(x)^{2-m}\left(\mu_{\psi}\cdot (x-c_{\rho^{2-m}})\right)^2\dx,
\end{multline}
which then yields
\begin{equation}
\label{eqn:int-rhomm2phis}
\int_{\supp}\rho(x)^{m-2}\psi(x)^2\dx=\frac{\|\mu_{\psi}\|^4}{\int_{\supp}\rho(x)^{2-m}\left(\mu_{\psi}\cdot (x-c_{\rho^{2-m}})\right)^2\dx}.
\end{equation}
Finally, using \eqref{eqn:int-rhomm2phis} in \eqref{eqn:phi-2} we write $\psi$ as 
\begin{equation}
\label{eqn:phi-final}
\psi(x)=\frac{\|\mu_{\psi}\|^2\rho(x)^{2-m}\mu_{\psi}\cdot(x-c_{\rho^{2-m}})}{\int_{\supp}\rho(z)^{2-m}\left(\mu_{\psi}\cdot (z-c_{\rho^{2-m}})\right)^2\dz}.
\end{equation}

\begin{remark}[Minimizing in $\calA$ versus $\tilde{\calA}$.] 
\label{rmk:minA}
When $\rho$ is strictly supported on the sphere and $m>2$, $\psi(x)$ becomes infinite on the boundary of $\supp$ (as $\rho(x)$ vanishes there and $2-m<0$). In this case, the critical points in the form \eqref{eqn:phi-final} do not lie in $\calA$, but in the larger class $\tilde{\calA}$, which is the reason why we introduced this larger set. Next, we will first investigate how to minimize $\calF$ over $\tilde{\calA}$, and then return to minimizing $\calF$ over $\calA$.
\end{remark}


\subsection{Minimizing $\calF[\psi]$ over $\psi \in \tilde{\calA}$}
All critical points (and hence, the candidates for the minimizers of the functional $\calF$) are expressed in the form \eqref{eqn:phi-final}. However, \eqref{eqn:phi-final} defines $\psi$ implicitly, via its first moment $\mu_\psi$, so its consistency needs to be checked.

Let $\mathcal{C} \subset \tilde{\mathcal{A}}$ be the set of all nonzero critical points $\psi$ satisfying \eqref{eqn:phi-final}. For any $y\in\bbs^\dm$, we denote 
\begin{align}\label{eqn:phicy}
\psi_y(x)=\frac{\rho(x)^{2-m}y\cdot(x-c_{\rho^{2-m}})}{\int_{\supp}\rho(z)^{2-m}\left(y\cdot (z-c_{\rho^{2-m}})\right)^2\dz}.
\end{align}
We first note that for any $\psi \in\mathcal{C}$, there exist a constant $\alpha\in\bbr\backslash\{0\}$ and $y\in\bbs^\dm$ such that
\[
\psi=\alpha\psi_{y},
\]
since by \eqref{eqn:phi-final}, we can choose $y=\frac{\mu_\psi}{\|\mu_\psi\|}$ and $\alpha=\|\mu_\psi\|$.

Define
\begin{equation}
\label{eqn:Cprime}
\mathcal{C}':=\{\psi \in\mathcal{C}:\|\mu_\psi\|=1\}.
\end{equation}
We can easily find that if $\psi$ solves \eqref{eqn:phi-final} then so does $\alpha\psi$ for all $\alpha\neq0$. Furthermore, $\mathcal{F}$ is homogeneous, i.e. for any nonzero $\psi \in\tilde{\mathcal{A}}$ and $\alpha\neq0$, we have $\mathcal{F}[\alpha\psi]=\mathcal{F}[\psi]$. This implies that to minimize $\calF$ over $\calC$, it is enough to minimize over $\calC' \subset \calC$. Therefore, we have
\begin{align}
\label{infequiv}
\inf_{\psi \in\mathcal{\tilde{A}}}\mathcal{F}[\psi]=\inf_{\psi\in\mathcal{C}}\mathcal{F}[\psi]=\inf_{ \psi \in\mathcal{C}'}\mathcal{F}[\psi].
\end{align}

\begin{lemma}\label{ppcp}
Let $y \in \bbs^\dm$ be a vector parallel or perpendicular to $x_0$. Then $\psi_y\in\mathcal{C}'$.
\end{lemma}
\begin{proof}
For both cases, we will show that $\psi_y$ satisfies \eqref{eqn:phi-final} and that $\|\mu_{\psi_y}\| =1$.\\

\noindent$\bullet$ (Case 1: $y$ parallel to $x_0$). For this case, we can assume $y=x_0$ since for $y=-x_0$, $\psi_{-x_0}=-\psi_{x_0}$, and the argument is the same. By \eqref{eqn:phicy}, we have
\begin{align}\label{eq:phix0}
\psi_{x_0}(x)=\frac{\rho(x)^{2-m}x_0\cdot(x-c_{\rho^{2-m}})}{\int_{\supp}\rho(z)^{2-m}\left(x_0\cdot (z-c_{\rho^{2-m}})\right)^2\dz},
\end{align}
and then we can calculate $\mu_{\psi_{x_0}}$ as follows:
\begin{equation}
\label{eqn:mu-phic}
\mu_{\psi_{x_0}}=\frac{\int_{\supp}\rho(x)^{2-m}\left(x_0\cdot(x-c_{\rho^{2-m}})\right) x\, \dx}{\int_{\supp}\rho(x)^{2-m}\left(x_0\cdot (x-c_{\rho^{2-m}})\right)^2\dx}.
\end{equation}

Let $x_1 \in \bbs^\dm$ be an arbitrary vector perpendicular to $x_0$. We will show that $\mu_{\psi_{x_0}}$ is orthogonal to $x_1$, i.e., $\mu_{\psi_{x_0}}$ is parallel to $x_0$.  Write
\[
x_1 \cdot \int_{\supp}\rho(x)^{2-m}\left(x_0\cdot(x-c_{\rho^{2-m}})\right)x\dx =\int_{\supp}\rho(x)^{2-m}\left(
x_0\cdot(x-c_{\rho^{2-m}})\right)(x_1\cdot x)\dx.
\]
Denote by $r(x)$ the reflection across the hyperplane orthogonal to $x_1$. Then, by the axial symmetry of $\rho$ around $x_0$, we get
\begin{align*}
&\int_{\supp}\rho(x)^{2-m}\left(
x_0\cdot(x-c_{\rho^{2-m}})\right)(x_1\cdot x)\dx\\
&\qquad = \frac{1}{2} \int_{\supp}\rho(x)^{2-m}\left(
x_0\cdot(x-c_{\rho^{2-m}})\right)(x_1\cdot x)\dx \\
& \qquad \quad + \frac{1}{2} \int_{\supp}\rho(r(x))^{2-m}\left(
x_0\cdot(r(x)-c_{\rho^{2-m}})\right)(x_1\cdot r(x))\dx\\
& \qquad = \frac{1}{2} \int_{\supp}\rho(x)^{2-m}\left(
x_0\cdot(x-c_{\rho^{2-m}})\right)(x_1\cdot x)\dx \\
&\qquad \quad + \frac{1}{2} \int_{\supp}\rho(x)^{2-m}\left(
x_0\cdot(x-c_{\rho^{2-m}})\right)(-x_1\cdot x)\dx\\
&\qquad =0.
\end{align*}
Therefore, by \eqref{eqn:mu-phic} we have $\mu_{\psi_{x_0}}\cdot x_1=0$, for any $x_1$ perpendicular to $x_0$. It implies that
\[
\mu_{\psi_{x_0}}=(\mu_{\psi_{x_0}}\cdot x_0)x_0.
\]
We can further calculate $\mu_{\psi_{x_0}}\cdot x_0$ as follows:
\begin{align*}
\mu_{\psi_{x_0}}\cdot x_0&=\frac{\int_{\supp}\rho(x)^{2-m}\left(x_0\cdot(x-c_{\rho^{2-m}})\right)(x_0\cdot x)\dx}{\int_{\supp}\rho(x)^{2-m}\left(x_0\cdot (x-c_{\rho^{2-m}})\right)^2\dx}\\[2pt]
&=\frac{\int_{\supp}\rho(x)^{2-m}(x_0\cdot x)^2\dx-\int_{\supp}\rho(x)^{2-m}(x_0\cdot x)(x_0\cdot c_{\rho^{2-m}})\dx}{\int_{\supp}\rho(x)^{2-m}(x_0\cdot x)^2\dx-\int_{\supp}\rho(x)^{2-m}(x_0\cdot x)(x_0\cdot c_{\rho^{2-m}})\dx}\\[2pt]
&=1,
\end{align*}
where for the second equal sign we simplified the denominator using (see \eqref{eqn:CM-general} for the definition of $c_{\rho^{2-m}}$):
\[
\int_{\supp}\rho(x)^{2-m}(x_0\cdot x)(x_0\cdot c_{\rho^{2-m}})\dx = \int_{\supp}\rho(x)^{2-m}(x_0\cdot c_{\rho^{2-m}})(x_0\cdot c_{\rho^{2-m}})\dx.
\]

We conclude that $\mu_{\psi_{x_0}}=x_0$, and hence $\psi_{x_0}$ given by \eqref{eq:phix0} satisfies \eqref{eqn:phi-final}. Also, $\|\mu_{\psi_{x_0}}\|=\|x_0\|=1$, and therefore,  $\psi_{x_0}\in\mathcal{C}'$.\\

\noindent$\bullet$ (Case 2: $y$ perpendicular to $x_0$). Take an arbitrary $x_1 \in \bbs^\dm$ perpendicular to $x_0$ and set $y=x_1$. By \eqref{eqn:phicy} we write
\begin{align}\label{eqn:phix1}
\psi_{x_1}(x)=\frac{\rho(x)^{2-m}x_1\cdot(x-c_{\rho^{2-m}})}{\int_{\supp}\rho(z)^{2-m}\left(x_1\cdot (z-c_{\rho^{2-m}})\right)^2\dz}.
\end{align}
Since $c_{\rho^{2-m}}$ is parallel to $x_0$, we have $x_1\cdot c_{\rho^{2-m}}=0$. Therefore, one can simplify \eqref{eqn:phix1} to get:
\[
\psi_{x_1}(x)=\frac{\rho(x)^{2-m}\left(x_1\cdot x\right)}{\int_{\supp}\rho(z)^{2-m}\left(x_1\cdot z\right)^2\dz}.
\]
From this expression, we calculate $\mu_{\psi_{x_1}}$:
\[
\mu_{\psi_{x_1}}=\frac{\int_{\supp}\rho(x)^{2-m}\left(x_1\cdot x\right)x\dx}{\int_{\supp}\rho(x)^{2-m}\left(x_1\cdot x\right)^2\dx},
\]
and hence,
\begin{equation}
\label{eqn:dotp1}
\mu_{\psi_{x_1}}\cdot x_1=\frac{\int_{\supp}\rho(x)^{2-m}\left(x_1\cdot x\right)^2\dx}{\int_{\supp}\rho(x)^{2-m}\left(x_1\cdot x\right)^2\dx}=1.
\end{equation}

Now, let $x_2$ be an arbitrary vector perpendicular to $x_1$ ($x_2$ can be $x_0$), and $r(x)$ be the reflection across the hyperplane orthogonal to $x_1$, as introduced in the previous case. We then compute:
\begin{align*}
&\int_{\supp}\rho(x)^{2-m}\left(x_1\cdot x\right)(x\cdot x_2)\dx\\[2pt]
&=\frac{1}{2}\int_{\supp}\rho(x)^{2-m}\left(x_1\cdot x\right)(x\cdot x_2)\dx+\frac{1}{2}\int_{\supp}\rho(r(x))^{2-m}\left(x_1\cdot r(x)\right)(r(x)\cdot x_2)\dx\\[2pt]
&=\frac{1}{2}\int_{\supp}\rho(x)^{2-m}\left(x_1\cdot x\right)(x\cdot x_2)\dx+\frac{1}{2}\int_{\supp}\rho(x)^{2-m}\left(-x_1\cdot x\right)(x\cdot x_2)\dx \\[2pt]
&=0.
\end{align*}
Therefore, $\mu_{\psi_{x_1}} \cdot x_2 = 0$ for any $x_2 \perp x_1$, and also using \eqref{eqn:dotp1} we find $\mu_{\psi_{x_1}}=x_1$. This  implies that $\|\mu_{\psi_{x_1}}\|=\|x_1\|=1$ and that $\psi_{x_1}$ given by \eqref{eqn:phix1} satisfies \eqref{eqn:phi-final}. We conclude that $\psi_{x_1}\in\mathcal{C}'$.
\end{proof}

\begin{lemma}
\label{lem:F=F}
Let $y,\tilde{y}\in\bbs^\dm$ satisfy $y\cdot x_0=\tilde{y}\cdot x_0$. Then, $\mathcal{F}[\psi_{y}]=\mathcal{F}[\psi_{\tilde{y}}]$.
\end{lemma}

\begin{proof}
Assume $y\neq\tilde{y}$, as otherwise the result holds trivially. Since $y$ and $\tilde{y}$ satisfy the given condition, we can express them as
\[
y=\cos\alpha x_0+\sin\alpha x_1,\qquad \tilde{y}=\cos\alpha x_0+\sin\alpha \tilde{x}_1,
\]
for some $\alpha\in[0,\pi]$ and $x_1, \tilde{x}_1\in\bbs^\dm$ perpendicular to $x_0$. Since $y\neq\tilde{y}$ we also have $x_1\neq\tilde{x}_1$. 

Define 
\[
x_2=\frac{\tilde{x}_1-(\tilde{x}_1\cdot x_1)x_1 }{\|\tilde{x}_1-(\tilde{x}_1\cdot x_1)x_1\|},
\]
which is perpendicular to both $x_0$ and $x_1$. Then there exists $\beta\in[0, 2\pi)$ such that
\[
\tilde{x}_1=\cos\beta x_1+\sin\beta x_2.
\]
Also define
\[
\tilde{x}_2:=-\sin\beta x_1+\cos\beta x_2.
\]
Then, $\{x_1,x_2\}$ and $\{\tilde{x}_1,\tilde{x}_2\}$ are two distinct orthonormal bases of the same affine space. 

Now choose orthonormal bases of $\bbr^{\dm+1}$ ($\bbs^\dm$ is embedded in $\bbr^{\dm+1}$) as follows
\[
\{x_0, x_1, x_2, x_3,\cdots, x_\dm\},\quad \{x_0, \tilde{x}_1, \tilde{x}_2, x_3, \cdots, x_\dm\},
\]
and define the matrix $R$ by
\[
R=x_0x_0^\top+\tilde{x}_1 x_1^\top+\tilde{x}_2x_2^\top+\sum_{k=3}^{\dm} x_k x_k^\top.
\]
The matrix $R$ is orthogonal, and satisfies
\[
Rx_0=x_0,\qquad Ry=\tilde{y}.
\]
Using matrix $R$, we can express $\psi_{\tilde{y}}$ as follows:
\begin{equation}
\label{eqn:phiytilde-phiy}
\begin{aligned}
\psi_{\tilde{y}}(x)&=\frac{\rho(x)^{2-m}\tilde{y}\cdot(x-c_{\rho^{2-m}})}{\int_{\supp}\rho(z)^{2-m}\left(\tilde{y}\cdot (z-c_{\rho^{2-m}})\right)^2\dz}\\[2pt]
&=\frac{\rho(x)^{2-m}(R^{-1}\tilde{y})\cdot(R^{-1}(x-c_{\rho^{2-m}}))}{\int_{\supp}\rho(z)^{2-m}\left((R^{-1}\tilde{y})\cdot (R^{-1}(z-c_{\rho^{2-m}}))\right)^2\dz}\\[2pt]
&=\frac{\rho(x)^{2-m}y\cdot(R^{-1}x-c_{\rho^{2-m}})}{\int_{\supp}\rho(z)^{2-m}\left(y\cdot (R^{-1}z-c_{\rho^{2-m}})\right)^2\dz}\\[2pt]
&=\psi_y(R^{-1}x),
\end{aligned}
\end{equation}
where in the third equality we used that $R^{-1}c_{\rho^{2-m}}=c_{\rho^{2-m}}$, as $c_{\rho^{2-m}}$ is parallel to $x_0$. 

Using \eqref{eqn:phiytilde-phiy} we write 
$\mathcal{F}[\psi_{\tilde{y}}]$ as 
\begin{equation}
\begin{aligned}
\mathcal{F}[\psi_{\tilde{y}}]&=\frac{\int_{\supp}\rho(x)^{m-2}\psi_{\tilde{y}}(x)^2\dx}{\left \| \int_{\bbs^\dm} x \psi_{\tilde{y}}(x) \dx \right \|^2}\\[2pt]
&=\frac{\int_{\supp}\rho(x)^{m-2}\psi_{y}(R^{-1}x)^2\dx}{\left \| \int_{\bbs^\dm} x \psi_{y}(R^{-1}x) \dx \right \|^2}.
\end{aligned}
\label{eqn:Fphiytilde}
\end{equation}
Consider first the numerator of the r.h.s of \eqref{eqn:Fphiytilde}. Since $\rho$ is axially symmetric with respect to $x_0$ and $R$ is a rotation with respect to $x_0$, we have $\rho(x)=\rho(R^{-1}x)$. This implies
\begin{align*}
\int_{\supp}\rho(x)^{m-2}\psi_{y}(R^{-1}x)^2\dx 
& =\int_{\supp}\rho(R^{-1}x)^{m-2}\psi_{y}(R^{-1}x)^2\dx \\
& = \int_{\supp}\rho(x)^{m-2}\psi_{y}(x)^2\dx,
\end{align*}
where for the second equal sign we used that $x\mapsto Rx$ is an isometry on $\bbs^\dm$. 

For the denominator of the r.h.s of \eqref{eqn:Fphiytilde} we use again that $x\mapsto Rx$ is an isometry, to get
\[
\int_{\bbs^\dm} x \psi_{y}(R^{-1}x) \dx=R\int_{\bbs^\dm} (R^{-1}x) \psi_{y}(R^{-1}x) \dx=R\int_{\bbs^\dm} x \psi_{y}(x) \dx.
\]
Then, since $R$ is an orthogonal matrix we find
\[
\left\|\int_{\bbs^\dm} x \psi_{y}(R^{-1}x) \dx\right\|^2=\left\|R\int_{\bbs^\dm} x \psi_{y}(x) \dx\right\|^2=\left\|\int_{\bbs^\dm} x \psi_{y}(x) \dx\right\|^2.
\]
Combining the results, we get from \eqref{eqn:Fphiytilde} that
\[
\mathcal{F}[\psi_{\tilde{y}}]=\frac{\int_{\supp}\rho(x)^{m-2}\psi_{y}(x)^2\dx}{\left\|\int_{\bbs^\dm} x \psi_{y}(x) \dx\right\|^2}=\mathcal{F}[\psi_y].
\]
\end{proof}

\begin{lemma}\label{lem:w-average}
Let $X$ and $Y$ be positive numbers. Then,
\[
\max_{0\leq\alpha\leq \pi}\frac{(\cos^2\alpha) X^2+ (\sin^2\alpha) Y^2}{(\cos^2\alpha) X+(\sin^2\alpha) Y}=\max(X,Y).
\]
\end{lemma}
\begin{proof}
Since $(\cos^2\alpha) X$ and $(\sin^2\alpha) Y$ are both nonnegative (and not simultaneously zero), the given fraction can be considered as a weighted average of $X$ and $Y$:
\[
\frac{(\cos^2\alpha X)X+(\sin^2\alpha Y)Y}{(\cos^2\alpha) X+(\sin^2\alpha) Y}=pX+(1-p)Y,
\]
where
\[
p=\frac{(\cos^2\alpha) X}{(\cos^2\alpha) X+(\sin^2\alpha) Y}\in[0,1].
\]
Therefore, the weighted average satisfies
\[
\frac{(\cos^2\alpha) X^2+(\sin^2\alpha) Y^2}{(\cos^2\alpha) X+(\sin^2\alpha) Y}\leq\max(X,Y).
\]
If we substitute $\alpha=0$ and $\alpha=\frac{\pi}{2}$, we find $X$ and $Y$, respectively. Therefore, the maximum is achieved and we get the desired result.
\end{proof}

A key result which will be used to minimize $\calF$, is given by the following proposition. 
\begin{proposition}\label{prop:minphiy}
Consider $\psi_y$ defined as in \eqref{eqn:phicy} for generic $y \in \bbs^\dm$, and let $x_1 \in \bbs^\dm$ be a fixed (arbitrary) vector orthogonal to $x_0$. Then,
\begin{equation}
\label{eqn:minF=min}
\min_{y\in\bbs^\dm}\mathcal{F}[\psi_y]=\min\left(\mathcal{F}[\psi_{x_0}], \mathcal{F}[\psi_{x_1}]\right).
\end{equation}
\end{proposition}
\begin{proof}
Let $y$ be an arbitrary vector in $\bbs^\dm$. By Lemma \ref{lem:F=F}, we can assume without loss of generality that $y$ can be written as
\begin{equation}
\label{eqn:y-decomp}
y=\cos\alpha \, x_0+\sin\alpha \, x_1,
\end{equation}
for some $\alpha \in [0,\pi]$, where $x_1$ is the fixed vector orthogonal to $x_0$.  By \eqref{eqn:calF} and the definition of $\psi_y$, we calculate
\begin{equation}
\label{eqn:calF-phiy}
\begin{aligned}
\mathcal{F}[\psi_y]&=\frac{\int_{\supp}\rho(x)^{m-2}\psi_y(x)^2\dx}{\|\mu_{\psi_y}\|^2} \\[2pt]
&=\frac{1}{\|\mu_{\psi_y}\|^2\int_{\supp}\rho(x)^{2-m}\left(y\cdot (x-c_{\rho^{2-m}})\right)^2\dx}.
\end{aligned}
\end{equation}

Then, we compute
\begin{equation}
\label{eqn:Fphiy-calc}
\begin{aligned}
&\int_{\supp}\rho(x)^{2-m}\left(y\cdot (x-c_{\rho^{2-m}})\right)^2\dx\\
&=\cos^2\alpha\int_{\supp}\rho(x)^{2-m}(x_0\cdot(x-c_{\rho^{2-m}}))^2\dx+\sin^2\alpha\int_{\supp}\rho(x)^{2-m}(x_1\cdot(x-c_{\rho^{2-m}}))^2\dx\\
&\quad +2\cos\alpha\sin\alpha\int_{\supp}\rho(x)^{2-m}(x_0\cdot(x-c_{\rho^{2-m}}))(x_1\cdot(x-c_{\rho^{2-m}}))\dx \\[2pt]
&=\cos^2\alpha\, \mathcal{F}[\psi_{x_0}]^{-1}+\sin^2\alpha \, \mathcal{F}[\psi_{x_1}]^{-1}\\[2pt]
&\quad +2\cos\alpha\sin\alpha\int_{\supp}\rho(x)^{2-m}(x_0\cdot(x-c_{\rho^{2-m}}))(x_1\cdot(x-c_{\rho^{2-m}}))\dx.
\end{aligned}
\end{equation}

We will show that the last term in the r.h.s. above is zero. Indeed, using that $c_{\rho^{2-m}}$ is parallel to $x_0$, we get
\begin{align*}
&\int_{\supp}\rho(x)^{2-m}(x_0\cdot(x-c_{\rho^{2-m}}))(x_1\cdot(x-c_{\rho^{2-m}}))\dx\\
&\qquad =\int_{\supp}\rho(x)^{2-m}(x_0\cdot(x-c_{\rho^{2-m}}))(x_1\cdot x)\dx.
\end{align*}
Denote by $r(x)$ the reflection across the hyperplane orthogonal to $x_1$. Then, by the axial symmetry of $\rho$ abound $x_0$, we continue the calculation above and further obtain
\begin{align}
\begin{aligned}\label{symmetric0}
& \int_{\supp}\rho(x)^{2-m}(x_0\cdot (x-c_{\rho^{2-m}}))(x_1\cdot x)\dx \\
&\qquad =\frac{1}{2}\int_{\supp}\rho(x)^{2-m}(x_0\cdot (x-c_{\rho^{2-m}}))(x_1\cdot x)\dx \\
& \qquad \quad +\frac{1}{2}\int_{\supp}\rho(r(x))^{2-m}(x_0\cdot (r(x)-c_{\rho^{2-m}}))(x_1\cdot r(x))\dx\\
& \qquad =\frac{1}{2}\int_{\supp}\rho(x)^{2-m}(x_0\cdot (x-c_{\rho^{2-m}}))(x_1\cdot x)\dx \\
& \qquad \quad -\frac{1}{2}\int_{\supp}\rho(x)^{2-m}(x_0\cdot (x-c_{\rho^{2-m}}))(x_1\cdot x)\dx\\
&\qquad =0.
\end{aligned}
\end{align}
Therefore, by \eqref{eqn:Fphiy-calc} we have
\begin{equation}
\label{eqn:denom-t2}
\int_{\supp}\rho(x)^{2-m}\left(y\cdot (x-c_{\rho^{2-m}})\right)^2\dx=\cos^2\alpha \, \mathcal{F}[\psi_{x_0}]^{-1}+\sin^2\alpha \, \mathcal{F}[\psi_{x_1}]^{-1}.
\end{equation}

By \eqref{eqn:denom-t2}, we have simplified one of the terms in the denominator of \eqref{eqn:calF-phiy}. Next, we look into $\|\mu_{\psi_y}\|$. Using the definition \eqref{eqn:phicy} of $\psi_y$, we have
\begin{equation}
\label{eqn:muphiy}
\begin{aligned}
\mu_{\psi_y}&=\int_{\supp} \psi_y(x)x\, \dx\\
&=\frac{\int_{\supp}\rho(x)^{2-m}y\cdot(x-c_{\rho^{2-m}})x\, \dx}{\int_{\supp}\rho(x)^{2-m}\left( y\cdot (x-c_{\rho^{2-m}})\right)^2\dx}.
\end{aligned}
\end{equation}
The denominator in \eqref{eqn:muphiy} was calculated in \eqref{eqn:denom-t2}. For the numerator, we use \eqref{eqn:y-decomp} to write it as
\begin{align*}
\mathcal{N}=&\cos\alpha\int_{\supp}\rho(x)^{2-m}x_0\cdot(x-c_{\rho^{2-m}})x\dx+\sin\alpha\int_{\supp}\rho(x)^{2-m}x_1\cdot(x-c_{\rho^{2-m}})x\dx.
\end{align*}

We first compute
\begin{align*}
\mathcal{N}\cdot x_0&=\cos\alpha\int_{\supp}\rho(x)^{2-m}\left(x_0\cdot(x-c_{\rho^{2-m}})\right)(x\cdot x_0)\dx\\[2pt]
&\quad +\sin\alpha\int_{\supp}\rho(x)^{2-m}\left(x_1\cdot(x-c_{\rho^{2-m}})\right)(x\cdot x_0)\dx.
\end{align*}
The second term in the r.h.s. is zero from a calculation similar to \eqref{symmetric0}. Also, the first term can be rewritten using \eqref{eqn:crhomm2} to get
\begin{equation}
\begin{aligned}
\label{eqn:Ndotx0}
\mathcal{N}\cdot x_0&=\cos\alpha\int_{\supp}\rho(x)^{2-m}\left(x_0\cdot(x-c_{\rho^{2-m}})\right)(x\cdot x_0)\dx\\
&=\cos\alpha\int_{\supp}\rho(x)^{2-m}\left(x_0\cdot(x-c_{\rho^{2-m}})\right)^2\dx \\
&=\cos\alpha\mathcal{F}[\psi_{x_0}]^{-1},
\end{aligned}
\end{equation}
where for the last equal sign we used that $\|\mu_{\psi_{x_0}}\|=1$ (cf. Lemma \ref{ppcp}); see also \eqref{eqn:calF-phiy}.

Next, we consider the dot product with $x_1$:
\begin{align*}
\mathcal{N}\cdot x_1&=\cos\alpha\int_{\supp}\rho(x)^{2-m}\left(x_0\cdot(x-c_{\rho^{2-m}})\right)(x\cdot x_1)\dx\\
&\quad +\sin\alpha\int_{\supp}\rho(x)^{2-m}\big(x_1\cdot (x-c_{\rho^{2-m}})\big)(x\cdot x_1)\dx.
\end{align*}
The first term in the r.h.s. is zero by \eqref{symmetric0}. Then, we use $x_1\cdot c_{\rho^{2-m}}=0$ to get
\begin{equation}
\label{eqn:Ndotx1}
\begin{aligned}
\mathcal{N}\cdot x_1 &=\sin\alpha\int_{\supp}\rho(x)^{2-m}\big(x_1\cdot (x-c_{\rho^{2-m}})\big)^2\dx\\
&=\sin\alpha\mathcal{F}[\psi_{x_1}]^{-1},
\end{aligned}
\end{equation}
where we used again Lemma \ref{ppcp} (which gives $\|\mu_{\psi_{x_1}}\|=1$) and \eqref{eqn:calF-phiy}.

Now, choose a direction $x_2$ perpendicular to both $x_0$ and $x_1$. We then have
\begin{align*}
\mathcal{N}\cdot x_2&=\cos\alpha\int_{\supp}\rho(x)^{2-m}\left(x_0\cdot(x-c_{\rho^{2-m}})\right)(x\cdot x_2)\dx\\
&\quad +\sin\alpha\int_{\supp}\rho(x)^{2-m}\left(x_1\cdot(x-c_{\rho^{2-m}})\right)(x\cdot x_2)\dx.
\end{align*}
Each of the integrals in the r.h.s above is zero. This can be shown by a very similar calculation as in \eqref{symmetric0}, but performing instead a symmetry with respect to the plane perpendicular to $x_2$. We leave this simple exercise to the reader. Hence, we obtain
\[
\mathcal{N}\cdot x_2=0.
\]

Combining the above with \eqref{eqn:Ndotx0} and \eqref{eqn:Ndotx1}, we find
\[
\mathcal{N}=\left(\cos\alpha\mathcal{F}[\psi_{x_0}]^{-1}\right)x_0+\left(\sin\alpha\mathcal{F}[\psi_{x_1}]^{-1}\right)x_1,
\]
which together with \eqref{eqn:denom-t2} and \eqref{eqn:muphiy} yield
\begin{equation}
\label{eqn:norm-muphiy-s}
\|\mu_{\psi_y}\|^2=\frac{\left(\cos\alpha\mathcal{F}[\psi_{x_0}]^{-1}\right)^2+\left(\sin\alpha\mathcal{F}[\psi_{x_1}]^{-1}\right)^2}{\left(\cos^2\alpha \, \mathcal{F}[\psi_{x_0}]^{-1}+\sin^2\alpha \, \mathcal{F}[\psi_{x_1}]^{-1}\right)^2}.
\end{equation}

From \eqref{eqn:calF-phiy}, \eqref{eqn:denom-t2} and \eqref{eqn:norm-muphiy-s}, we now find
\[
\mathcal{F}[\psi_y]=\frac{\cos^2\alpha \, \mathcal{F}[\psi_{x_0}]^{-1}+\sin^2\alpha \, \mathcal{F}[\psi_{x_1}]^{-1}}{\left(\cos\alpha\mathcal{F}[\psi_{x_0}]^{-1}\right)^2+\left(\sin\alpha\mathcal{F}[\psi_{x_1}]^{-1}\right)^2}.
\]
We want to find the minimum value of $\mathcal{F}[\psi_y]$. Consider instead the complementary problem, which is to find the maximum of 
\[
\mathcal{F}[\psi_y]^{-1}=\frac{ \cos^2\alpha \left( \mathcal{F}[\psi_{x_0}]^{-1}\right)^2+ \sin^2\alpha \left( \mathcal{F}[\psi_{x_1}]^{-1}\right)^2}{\cos^2\alpha \, \mathcal{F}[\psi_{x_0}]^{-1}+\sin^2\alpha \, \mathcal{F}[\psi_{x_1}]^{-1}}
\]
over $\alpha \in [0,\pi]$. We can find the maximum value of $\mathcal{F}[\psi_y]^{-1}$ using Lemma \ref{lem:w-average} with $X=\mathcal{F}[\psi_{x_0}]^{-1}$ and $Y=\mathcal{F}[\psi_{x_1}]^{-1}$. We find
\[
\max_{y\in\bbs^\dm}\mathcal{F}[\psi_y]^{-1}=\max\left(\mathcal{F}[\psi_{x_0}]^{-1},\mathcal{F}[\psi_{x_1}]^{-1}\right),
\]
which in turn implies \eqref{eqn:minF=min}.
\end{proof}

Combining various of the results above we can state the following proposition.
\begin{proposition}
\label{prop:minF-tildeA} 
Consider an equilibrium $\rho$ with normalized centre of mass at $x_0$ (see \eqref{eqn:equil} and \eqref{eqn:x0}) and  the functional $\calF[\psi]$ defined by \eqref{eqn:calF}. Let $x_1 \in \bbs^\dm$ be a fixed vector orthogonal to $x_0$. Then, we have
\[
\min_{\psi\in\tilde{\mathcal{A}}}\mathcal{F}[\psi]=\min\left(\mathcal{F}[\psi_{x_0}], \mathcal{F}[\psi_{x_1}]\right),
\]
where $\tilde{\calA}$ was defined in \eqref{eqn:calAtilde}; see also \eqref{eqn:phicy}.
\end{proposition}
\begin{proof}
By \eqref{infequiv}, it is enough to minimize $\calF$ over $\calC'$, where $\calC'$ was defined in \eqref{eqn:Cprime}. Let $\psi\in\mathcal{C}'$; then $\psi$ solves \eqref{eqn:phi-final} and $\|\mu_\psi\|=1$. Therefore, if we set $z=\mu_\psi$ then $\psi$ can be expressed as
\[
\psi=\psi_z,
\]
with $\psi_z$ defined by \eqref{eqn:phicy}. In particular, $\psi\in\{\psi_y:y\in\bbs^\dm\}$ and hence,
\begin{align}\label{cpsub}
\mathcal{C}'\subset\{\psi_y:y\in\bbs^\dm\}.
\end{align}

We can now establish the following relationships:
\[
\min\left(\mathcal{F}[\psi_{x_0}], \mathcal{F}[\psi_{x_1}]\right)\geq\inf_{\psi \in\mathcal{C}'}\mathcal{F}[\psi]\geq\inf_{y\in\bbs^\dm}\mathcal{F}[\psi_y]=\min\left(\mathcal{F}[\psi_{x_0}], \mathcal{F}[\psi_{x_1}]\right),
\]
where for the first inequality we used Lemma \ref{ppcp}, the second inequality follows from \eqref{cpsub}, and for the last equality we used Proposition \ref{prop:minphiy}. From here we infer 
\[
\min_{\psi \in\mathcal{C}'}\mathcal{F}[\psi]=\min\left(\mathcal{F}[\psi_{x_0}], \mathcal{F}[\psi_{x_1}]\right),
\]
where we used $\min$ instead of $\inf$ since $\psi_{x_0}$ and $\psi_{x_1}$ are elements of $\mathcal{C}'$. Together with \eqref{infequiv}, this shows the desired result.
\end{proof}

\begin{proposition} 
\label{prop:calF-x0x1-ABC}
Consider the equilibrium $\rho$, and $\psi_{x_0}$, $\psi_{x_1}$ as in Proposition \ref{prop:minF-tildeA}. Then, using the notations \eqref{eqn:ABC}, $\mathcal{F}[\psi_{x_0}]$ and $\mathcal{F}[\psi_{x_1}]$ can be expressed as
\begin{equation}
\label{eqn:calF-x0x1-ABC}
\mathcal{F}[\psi_{x_0}]^{-1}=C(\rho)-\frac{B(\rho)^2}{A(\rho)},\quad \text{ and } \quad \mathcal{F}[\psi_{x_1}]^{-1}=\frac{1}{\dm}(A(\rho)-C(\rho)).
\end{equation}
\end{proposition}
\begin{proof}
Using that $c_{\rho^{2-m}}$ is parallel to $x_0$, we simplify $\mathcal{F}[\psi_{x_0}]$ as (see \eqref{eqn:calF-phiy} and note that $\|\mu_{\psi_{x_0}}\|= 1$):
\begin{align*}
\mathcal{F}[\psi_{x_0}]^{-1}&=\int_{\supp}\rho(x)^{2-m}(x_0\cdot(x-c_{\rho^{2-m}}))^2\dx\\
&=\int_{\supp}\rho(x)^{2-m}\left((x_0\cdot x)^2-2(x_0\cdot x)(x_0\cdot c_{\rho^{2-m}})+(x_0\cdot c_{\rho^{2-m}})^2\right)\dx\\
&=\int_{\supp}\rho(x)^{2-m}(x_0\cdot x)^2\dx-(x_0\cdot c_{\rho^{2-m}})^2\int_{\supp}\rho(x)^{2-m}\dx\\
&=\int_{\supp}\rho(x)^{2-m}(x_0\cdot x)^2\dx-\|c_{\rho^{2-m}}\|^2\int_{\supp}\rho(x)^{2-m}\dx,
\end{align*}
where for the third equal sign we used \eqref{eqn:crhomm2}. Then, using the notations \eqref{eqn:ABC}, we write
\begin{align*}
\mathcal{F}[\psi_{x_0}]^{-1} &=C(\rho)-\left(\frac{B(\rho)}{A(\rho)}\right)^2A(\rho)\\
&=C(\rho)-\frac{B(\rho)^2}{A(\rho)}.
\end{align*}

Also, we find from \eqref{eqn:calF-phiy} (use $\|\mu_{\psi_{x_1}}\| = 1$):
\begin{align*}
\mathcal{F}[\psi_{x_1}]^{-1}&=\int_{\supp}\rho(x)^{2-m}(x_1\cdot(x-c_{\rho^{2-m}}))^2\dx\\
&=\int_{\supp}\rho(x)^{2-m}(x_1\cdot x)^2\dx.
\end{align*}
To simplify $\mathcal{F}[\psi_{x_1}]$ further, assume that $x_0, x_1, ..., x_\dm$ is an orthonormal basis of $\bbr^{\dm+1}$. By the axial symmetry of $\rho(x)^{2-m}$, we have
\begin{align*}
\int_{\supp}\rho(x)^{2-m}\dx&=\int_{\supp}\rho(x)^{2-m}\left((x_0\cdot x)^2+(x_1\cdot x)^2+\cdots+(x_\dm\cdot x)^2\right)\dx\\
&=\int_{\supp}\rho(x)^{2-m}(x_0\cdot x)^2\dx+\dm\int_{\supp}\rho(x)^{2-m}(x_1\cdot x)^2\dx,
\end{align*}
which yields
\begin{align*}
\mathcal{F}[\psi_{x_1}]^{-1}&=\frac{1}{\dm}\left(\int_{\supp}\rho(x)^{2-m}\dx-\int_{\supp}\rho(x)^{2-m}(x_0\cdot x)^2\dx\right)\\[2pt]
&=\frac{1}{\dm}\left(A(\rho)-C(\rho)\right).
\end{align*}
\end{proof}

\subsection{Minimizing $\calF[\psi]$ over $\psi \in \calA$} 

As noted in Remark \ref{rmk:minA}, if $\rho$ is strictly supported on the sphere and $m>2$, then $\psi_{x_0}$ and $\psi_{x_1}$ are not in $\mathcal{A}$ (as they are not bounded from below). Recall that our goal is to assess the sign of the second variation \eqref{eqn:d2Edeps-0}, and for this we need to minimize $\calF$ over $\psi \in \calA$.

The following lemma is key in connecting the minimizers on $\calA$ versus $\tilde{\calA}$.

\begin{lemma}\label{lem:approx} Let $y\in \bbs^\dm$ be either parallel or perpendicular to $x_0$. Then, there exists a sequence $\{\psi_{n, y}\}_{n\geq1}\subset \mathcal{A}$ such that
\[
\lim_{n\to\infty}\mathcal{F}[\psi_{n,y}]=\mathcal{F}[\psi_{y}].
\]
\end{lemma}
\begin{proof}
Let $y$ be a fixed vector either parallel or perpendicular to $x_0$. Recall \eqref{eqn:phicy}, and set
\[
\alpha:=\int_{\supp}\rho(x)^{2-m}\left(y\cdot (x-c_{\rho^{2-m}})\right)^2\dx.
\]
Define 
\[
\tilde{\psi}_y(x):=\alpha\psi_y(x)=\rho(x)^{2-m}y\cdot(x-c_{\rho^{2-m}}).
\]
Since the functional $\calF$ is homogeneous, we have
\begin{equation}
\label{eqn:Ftildephi=Fphi}
\mathcal{F}[\tilde{\psi}_y]=\mathcal{F}[\psi_y].
\end{equation}
As the function $\tilde{\psi}_y$ is much simpler in form and yields the same functional value as $\psi_y$, we will work with it instead.

For $R>0$ define
\[
\tilde{\psi}_y^R(x)=\rho(x)\min(\rho(x)^{1-m}, R)\left(y\cdot\left(x-c_{\rho\min(\rho^{1-m}, R)}\right)\right).
\]
By their construction, $\tilde{\psi}_y^R$ is expected to approximate $\tilde{\psi}_y$ for large $R$; this will be shown formally below. Also, $\tilde{\psi}_y^R$ satisfies
\begin{equation}
\label{eqn:int-tildephiy}
\int_{\supp}\tilde{\psi}_y^R(x)\dx=0,
\end{equation}
and furthermore,
\begin{equation}
\label{eqn:tildephiy-ineq}
\tilde{\psi}_y^R(x)\geq -2R\rho(x),
\end{equation}
as $\min(\rho(x)^{1-m},R)\leq R$ and 
\begin{align*}
\left|y\cdot\left(x-c_{\rho \min(\rho^{1-m}, R)}\right)\right| &\leq \|y\| \left( \|x\| + \| c_{\rho \min(\rho^{1-m}, R)}\| \right) \\[2pt]
&\leq 2.
\end{align*}
By \eqref{eqn:int-tildephiy}, \eqref{eqn:tildephiy-ineq}, and the definition of $\calA$ in \eqref{eqn:calA}, we have $\frac{1}{2R\epsilon_0}\tilde{\psi}_y^R\in\mathcal{A}$; this observation will be used later to define the sequence $\psi_{n,y} \in \calA$ needed to show the lemma.

To compare $\tilde{\psi}_y^R$ and $\tilde{\psi}_y$, we start by estimating the difference between $c_{\rho\min(\rho^{1-m},R)}$ and $c_{\rho^{2-m}}$. We compute
\begin{align*}
&c_{\rho\min(\rho^{1-m}, R)}-c_{\rho^{2-m}}\\[3pt]
&\quad = \frac{\int_{\supp}\rho(x)\min(\rho(x)^{1-m},R)x \, \dx}{\int_{\supp}\rho(x)\min(\rho(x)^{1-m},R)\, \dx}-\frac{\int_{\supp}\rho(x)^{2-m}x \, \dx}{\int_{\supp}\rho(x)^{2-m} \, \dx}\\[3pt]
&\quad = \frac{\int_{\supp} \rho(x)\min(\rho(x)^{1-m},R)x\dx \; \int_{\supp} \left(\rho(x)^{2-m}-\rho(x)\min(\rho(x)^{1-m},R) \right)\dx
}{\int_{\supp}\rho(x)\min(\rho(x)^{1-m},R)\dx \, \int_{\supp}\rho(x)^{2-m}\dx}\\[3pt]
&\qquad + \frac{\int_{\supp}\rho(x)\min(\rho(x)^{1-m},R)\dx \,\int_{\supp} \left(\rho(x)\min(\rho(x)^{1-m},R)-\rho(x)^{2-m} \right)x \dx
}{\int_{\supp}\rho(x)\min(\rho(x)^{1-m},R)\dx \, \int_{\supp}\rho(x)^{2-m}\dx}.
\end{align*}
From the triangle inequality, we then get
\begin{align*}
&\left \|c_{\rho\min(\rho^{1-m}, R)}-c_{\rho^{2-m}} \right\|\\[3pt]
&\leq\left\|\frac{\int_{\supp} \rho(x)\min(\rho(x)^{1-m},R)x\dx \, \int_{\supp} \left(\rho(x)^{2-m}-\rho(x)\min(\rho(x)^{1-m},R)\right)\dx
}{\int_{\supp}\rho(x)\min(\rho(x)^{1-m},R)\dx\, \int_{\supp}\rho(x)^{2-m}\dx}\right\|\\[3pt]
&\quad +\left\|\frac{\int\rho(x)\min(\rho(x)^{1-m},R)\dx\int_{\supp}\left(\rho(x)\min(\rho(x)^{1-m},R)-\rho(x)^{2-m} \right) x \dx
}{\int_{\supp}\rho(x)\min(\rho(x)^{1-m},R)\dx\int_{\supp}\rho(x)^{2-m}\dx}\right\|.
\end{align*}

If we use $\|x\|=1$ for all $x\in\bbs^\dm$, we can further estimate the r.h.s. above and find
\begin{align*}
&\left \|c_{\rho\min(\rho^{1-m}, R)}-c_{\rho^{2-m}} \right\|\\[3pt]
&\leq 2 \, \frac{\int_{\supp}\rho(x)\min(\rho(x)^{1-m},R)\dx \, \int_{\supp}|\rho(x)^{2-m}-\rho(x)\min(\rho(x)^{1-m},R)|\dx}{\int_{\supp}\rho(x)\min(\rho(x)^{1-m},R)\dx\int_{\supp}\rho(x)^{2-m}\dx}.
\end{align*}
By the dominated convergence theorem, we have
\begin{equation}
\label{eqn:2limR}
\begin{aligned}
&\lim_{R\to\infty}\int_{\supp}\rho(x)\min(\rho(x)^{1-m},R)\dx =\int_{\supp}\rho(x)^{2-m}\dx,\\
&\lim_{R\to\infty}\int_{\supp}|\rho(x)^{2-m}-\rho(x)\min(\rho(x)^{1-m},R)|\dx=0.
\end{aligned}
\end{equation}
Therefore, we can conclude that
\begin{equation}
\label{eqn:limR-normc}
\lim_{R\to\infty}\|c_{\rho\min(\rho^{1-m}, R)}-c_{\rho^{2-m}}\|=0.
\end{equation}

Now, we compare $\tilde{\psi}_y^R$ and $\tilde{\psi}_y$ directly:
\begin{align*}
&\tilde{\psi}_y(x)-\tilde{\psi}_y^R(x)\\[2pt]
&=\rho(x)^{2-m}y\cdot(x-c_{\rho^{2-m}})-\rho(x)\min(\rho(x)^{1-m}, R)\left(y\cdot\left(x-c_{\rho\min(\rho^{1-m}, R)}\right)\right)\\[2pt]
&=\rho(x)^{2-m}y\cdot(c_{\rho\min(\rho^{1-m},R)}-c_{\rho^{2-m}})-\rho(x)\left(\min(\rho(x)^{1-m}, R)-\rho^{1-m}\right)\left(y\cdot\left(x-c_{\rho\min(\rho^{1-m}, R)}\right)\right).
\end{align*}
Then, we estimate
\begin{align*}
&\|\tilde{\psi}_y-\tilde{\psi}_y^R\|_{1}\\
&\leq\int_{\supp}|\rho(x)^{2-m}y\cdot (c_{\rho\min(\rho^{1-m}, R)}-c_{\rho^{2-m}})|\dx\\
&+\int_{\supp}|(\rho^{2-m}(x)-\rho(x)\min(\rho(x)^{1-m},R))( y\cdot(x-c_{\rho\min(\rho^{1-m}, R)}))|\dx.
\end{align*}
We use again $\|y\|\leq1$ and $\|x-c_{\rho\min(\rho^{1-m},R)}\|\leq2$ to get
\begin{align*}
&\|\tilde{\psi}_y-\tilde{\psi}_y^R\|_{1}\\
&\leq \|c_{\rho\min(\rho^{1-m},R)}-c_{\rho^{2-m}}\|\int_{\supp}\rho(x)^{2-m}\dx\\
&+2\int_{\supp}|\rho^{2-m}(x)-\rho(x)\min(\rho(x)^{1-m},R)|\dx.
\end{align*}
By \eqref{eqn:2limR} and \eqref{eqn:limR-normc}, we finally find
\begin{equation}
\label{eqn:limphi-L1}
\lim_{R\to\infty}\|\tilde{\psi}_y-\tilde{\psi}_y^R\|_{1}=0.
\end{equation}

Next, we calculate the difference between $\mathcal{F}[\tilde{\psi}_y]$ and $\mathcal{F}[\psi_{y}^R]$:
\begin{align*}
&\mathcal{F}[\tilde{\psi}_y]-\mathcal{F}[\tilde{\psi}_y^R]\\[2pt]
&\quad =\frac{\int_{\supp}\rho(x)^{m-2}\tilde{\psi}_y(x)^2\dx}{\|\mu_{\tilde{\psi}_y}\|^2}-\frac{\int_{\supp}\rho(x)^{m-2}\tilde{\psi}_y^R(x)^2\dx}{\|\mu_{\tilde{\psi}_y^R}\|^2}\\
&\quad =\frac{1}{\|\mu_{\tilde{\psi}_y}\|^2}\int_{\supp}\rho(x)^{m-2}\left(\tilde{\psi}_y(x)^2-\tilde{\psi}_y^R(x)^2\right)\dx\\[2pt]
&\qquad +\frac{\|\mu_{\tilde{\psi}_{y}^R}\|^2-\|\mu_{\tilde{\psi}_y}\|^2}{\|\mu_{\tilde{\psi}_y}\|^2\|\mu_{\tilde{\psi}_{y}^R}\|^2}\int_{\supp}\rho(x)^{m-2}\tilde{\psi}_y^R(x)^2\dx.
\end{align*}
We further estimate
\begin{equation}
\label{eqn:F-diff}
\begin{aligned}
\left|\mathcal{F}[\tilde{\psi}_y]-\mathcal{F}[\tilde{\psi}_y^R]\right|&\leq
\frac{1}{\|\mu_{\tilde{\psi}_y}\|^2}\left|
\int_{\supp}\rho(x)^{m-2}(\tilde{\psi}_y(x)^2-\tilde{\psi}_y^R(x)^2)\dx\right|
\\
&+\frac{\left|\|\mu_{\tilde{\psi}_{y}^R}\|^2-\|\mu_{\tilde{\psi}_y}\|^2\right|}{\|\mu_{\tilde{\psi}_y}\|^2\|\mu_{\tilde{\psi}_{y}^R}\|^2}
\int_{\supp}\rho(x)^{m-2}\tilde{\psi}_y^R(x)^2\dx.
\end{aligned}
\end{equation}
The second term in the r.h.s. of \eqref{eqn:F-diff} tends to zero as $R\to\infty$ since
\[
\lim_{R\to\infty}\|\mu_{\tilde{\psi}_y^R}-\mu_{\tilde{\psi}_y}\|=\lim_{R\to\infty}\left\|\int_{\supp}(\tilde{\psi}_y^R(x)-\tilde{\psi}_y(x))x\dx\right\|\leq \lim_{R\to\infty} \|\tilde{\psi}_y^R-\tilde{\psi}_y\|_1=0,
\]
by using \eqref{eqn:limphi-L1}. We will also show that as $R \to \infty$, the first term in the r.h.s. of \eqref{eqn:F-diff} tends to zero as well. 

Indeed, as $x, y\in\bbs^\dm$ and $\|c_{\rho^{2-m}}\|,\|c_{\rho\min(\rho^{1-m},R)}\|\leq1$, we have
\[
|\tilde{\psi}_y(x)|,|\tilde{\psi}_y^R(x)|\leq 2\rho(x)^{2-m}, \qquad\forall x\in\supp,\quad R>0,
\]
from which we can infer
\begin{align*}
\rho(x)^{m-2}\left|\tilde{\psi}_y(x)^2-\tilde{\psi}_y^R(x)^2\right|&=\rho(x)^{m-2}\left|\tilde{\psi}_y(x)+\tilde{\psi}_y^R(x)\right|\times\left|\tilde{\psi}_y(x)-\tilde{\psi}_y^R(x)\right|\\
&\leq4\left|\tilde{\psi}_y(x)-\tilde{\psi}_y^R(x)\right|.
\end{align*}
Hence, we get
\begin{align*}
\left|
\int_{\supp}\rho(x)^{m-2}(\tilde{\psi}_y(x)^2-\tilde{\psi}_y^R(x)^2)\dx\right|& \leq 4\int_{\supp}|\tilde{\psi}_y(x)-\tilde{\psi}_y^R(x)|\dx \\
&=4\|\tilde{\psi}_y-\tilde{\psi}_y^R\|_1,
\end{align*}
and by \eqref{eqn:limphi-L1} we conclude that the first term in the r.h.s. of \eqref{eqn:F-diff} tends to zero as $R\to\infty$. Combining the results, we find from \eqref{eqn:F-diff} that
\[
\lim_{R\to\infty}\left|\mathcal{F}[\tilde{\psi}_y]-\mathcal{F}[\tilde{\psi}_y^R]\right|=0,
\]
and hence,
\begin{equation}
\label{eqn:limF-Rinf}
\lim_{R\to\infty}\mathcal{F}[\tilde{\psi}_y^R]=\mathcal{F}[\tilde{\psi}_y].
\end{equation}

Now, we choose a sequence 
\[
R_n\nearrow\infty,\quad\text{ as }\quad n\nearrow\infty.
\]

Since $\tilde{\psi}_y^{R_n}$ satisfies (see \eqref{eqn:tildephiy-ineq})
\[
\tilde{\psi}_y^{R_n}(x)\geq-2R_n\rho(x),\qquad\forall x\in\bbs^\dm,
\]
we define
\[
\psi_{n, y}=\frac{1}{2R_n\epsilon_0}\tilde{\psi}_y^{R_n},
\]
to ensure
\[
\psi_{n, y}(x)\geq -\frac{1}{\epsilon_0}\rho(x), \qquad\forall x\in\bbs^\dm.
\]
The functions $\psi_{n,y}$ also integrate to $0$ (see \eqref{eqn:int-tildephiy}), and therefore $\psi_{n, y}\in\mathcal{A}$ for all $n \geq 1$. Also, since the functional $\calF$ is homogeneous, we have 
\[
\mathcal{F}[\psi_{n,y}]=\mathcal{F}[\tilde{\psi}_y^{R_n}].
\]

Finally, using the above, along with \eqref{eqn:Ftildephi=Fphi} and \eqref{eqn:limF-Rinf}, we conclude that $\{\psi_{n, y}\}_{n\geq1}\subset\mathcal{A}$ satisfies 
\[
\lim_{n\to\infty}\mathcal{F}[\psi_{n, y}]= \lim_{n\to\infty}\mathcal{F}[\tilde{\psi}_y^{R_n}] = \mathcal{F}[\tilde{\psi}_y]=\mathcal{F}[\psi_y].
\]
\end{proof}

\begin{remark}
If $\rho$ is fully supported, then $\rho^{2-m}\in L^\infty(\bbs^\dm)$ and hence $\psi_y\in L^\infty(\bbs^\dm)$. Therefore, the proof of Lemma \ref{lem:approx} for this case can be directly obtained by the scaling argument.
\end{remark}

We now have all the preparatory results to characterize the stability of the energy equilibria.
\begin{theorem}[Criteria for stability]
\label{thm:minF}
Consider an energy equilibrium $\rho$ in the form \eqref{eqn:equil}, with normalized centre of mass at $x_0$. Let $x_1 \in \bbs^\dm$ be a fixed vector orthogonal to $x_0$, as in Proposition \ref{prop:minF-tildeA}, and the functional $\calF$ defined by \eqref{eqn:calF}. Then,
\smallskip

(a) If $\min(\mathcal{F}[\psi_{x_0}],\mathcal{F}[\psi_{x_1}])<\frac{\kappa}{m}$, the equilibrium $\rho$ is unstable. \smallskip

(b) If $\min(\mathcal{F}[\psi_{x_0}],\mathcal{F}[\psi_{x_1}])>\frac{\kappa}{m}$, the equilibrium $\rho$ is stable. \smallskip

(c) If $\min(\mathcal{F}[\psi_{x_0}],\mathcal{F}[\psi_{x_1}])=\frac{\kappa}{m}$ and $\rho$ is strictly supported, then $\rho$ is stable.
\end{theorem}
\begin{proof}
We investigate the sign of the second variation of the energy given by \eqref{eqn:d2Edeps-0}. By \eqref{eqn:calF}, this reduces to investigating the sign of $\mathcal{F}[\psi]-\frac{\kappa}{m}$. If $\mathcal{F}[\psi]-\frac{\kappa}{m}<0$ for some $\psi \in\mathcal{A}$ then $\rho$ is unstable, and if $\mathcal{F}[\psi]-\frac{\kappa}{m}>0$ for any $\psi \in\mathcal{A}$ then $\rho$ is stable. \\

\noindent(a) In this case, by Proposition \ref{prop:minF-tildeA} we have $\inf_{\psi \in\mathcal{A}}\mathcal{F}[\psi]=\min(\mathcal{F}[\psi_{x_0}],\mathcal{F}[\psi_{x_1}])<\frac{\kappa}{m}$. From Lemma \ref{lem:approx}, for $\epsilon=\frac{\kappa}{m}-\min(\mathcal{F}[\psi_{x_0}],\mathcal{F}[\psi_{x_1}])>0$, there exists $\psi \in\mathcal{A}$ such that
\[
\left|\mathcal{F}[\psi]-\min(\mathcal{F}[\psi_{x_0}],\mathcal{F}[\psi_{x_1}])\right|<\epsilon.
\]
Therefore, we get $\mathcal{F}[\psi]<\frac{\kappa}{m}$ and we can directly conclude the instability of $\rho$.\\

\noindent(b) This case is trivial, as by Proposition \ref{prop:minF-tildeA} we have $\mathcal{F}[\psi]\geq\min(\mathcal{F}[\psi_{x_0}],\mathcal{F}[\psi_{x_1}])>\frac{\kappa}{m}$ for any $\psi \in\mathcal{A}$. \\

\noindent(c) If $\rho$ is strictly supported, then $\psi_{x_0}, \psi_{x_1}\in\tilde{\mathcal{A}}\backslash\mathcal{A}$. Therefore, $\inf_{\psi \in\mathcal{A}}\mathcal{F}[\psi]=\min(\mathcal{F}[\psi_{x_0}],\mathcal{F}[\psi_{x_1}])=\frac{\kappa}{m}$, but the minimum of $\mathcal{F}$ is not achieved in $\mathcal{A}$. We can conclude then that $\mathcal{F}[\psi]>\frac{\kappa}{m}$ for any $\psi \in\mathcal{A}$, which implies that $\rho$ is stable. 
\end{proof}

\section{Stability of the equilibria ($m>2$)}
\label{sect:2ndorder-fscs}
In this section, we use Theorem \ref{thm:minF} to establish the stability/instability of the equilibria \eqref{eqn:equil}. 

\subsection{Fully supported equilibria}
\label{subsect:mg2:fs}
We consider first the fully supported equilibria -- equilibria of type a) in Section \ref{subsect:cp}, see \eqref{eqn:equil-fs}. Recall notation \eqref{eqn:eta-notation} for $\eta$. As for fully supported equilibria we have $\lambda \geq \kappa \|c_\rho\|$, $\eta$ ranges in $\eta \leq -1$. Written in terms of $\eta$, \eqref{eqn:equil-fs} reads
\begin{equation}
\label{eqn:equil-fs-eta}
\rho(x)=\left(\frac{(m-1)\kappa \|c_\rho\|}{m}\right)^{\frac{1}{m-1}}\left(\cos\theta_x-\eta\right)^{\frac{1}{m-1}},\qquad\forall x\in \bbs^\dm,
\end{equation}
where by \eqref{eqn:system-fs-mg1}, $\| c_\rho\|$ and $\eta$ satisfy
\begin{subequations}
\label{eqn:1s-comp}
\begin{align}
1&=\dm w_\dm\left(\frac{(m-1) \kappa \|c_\rho\|}{m}\right)^{\frac{1}{m-1}}\int_0^\pi (\cos\theta-\eta)^{\frac{1}{m-1}}\sin^{\dm-1}\theta\d\theta, \label{eqn:1-comp}\\
\|c_\rho\|&=\dm w_\dm\left(\frac{(m-1) \kappa \|c_\rho\|}{m}\right)^{\frac{1}{m-1}}\int_0^\pi (\cos\theta-\eta)^{\frac{1}{m-1}}\sin^{\dm-1}\theta\cos\theta\d\theta. \label{eqn:s-comp}
\end{align}
\end{subequations}

In Theorem \ref{thm:minF}, we classified the stability of equilibria in terms of the values of $\mathcal{F}[\psi_{x_0}]$ and $\mathcal{F}[\psi_{x_1}]$. To calculate these values we will use the following relationships between $A(\rho)$, $B(\rho)$ and $C(\rho)$ defined in \eqref{eqn:ABC}.
\begin{lemma}
\label{lemma:ABC-fs}
Let $\rho$ be a fully supported equilibrium in the form \eqref{eqn:equil-fs-eta}. Then there exist the following relationships between $A(\rho)$, $B(\rho)$ and $C(\rho)$ defined in \eqref{eqn:ABC}:
\begin{align}
\label{ABC-system-fs}
\begin{cases}
\displaystyle B(\rho) = \frac{m}{(m-1)\kappa \|c_\rho\|} + \eta A(\rho),\vspace{0.3cm}\\
\displaystyle C(\rho)=\frac{m}{(m-1)\kappa}+\eta B(\rho),\vspace{0.3cm}\\
\displaystyle A(\rho)-C(\rho)=\frac{\dm m}{\kappa}.
\end{cases}
\end{align}
Also, we can solve the above system of linear equations to get
\begin{equation}
\label{eqn:A-fs-calc2}
A(\rho)=\frac{1}{\eta^2-1}\frac{m}{\kappa}\left(-\frac{1}{m-1}-\frac{\eta}{(m-1)\|c_\rho\|}-\dm\right).
\end{equation}
\end{lemma}
\begin{proof}
See Appendix \ref{appendix:fs}.
\end{proof}

The value of $\mathcal{F}[\psi_{x_1}]$ follows directly from \eqref{eqn:calF-x0x1-ABC} and \eqref{ABC-system-fs} (third equation), which yield
\begin{equation}
\label{eqn:calFi-x1-fs}
\mathcal{F}[\psi_{x_1}]^{-1} =\frac{m}{\kappa},
\quad\text{or equivalently}\quad
\mathcal{F}[\psi_{x_1}] =\frac{\kappa}{m}.
\end{equation}
Therefore, to compare $\min(\mathcal{F}[\psi_{x_0}],\mathcal{F}[\psi_{x_1}])$ with $\frac{\kappa}{m}$ (cf., Theorem \ref{thm:minF}), it suffices to determine the sign of
\[
\mathcal{F}[\psi_{x_0}]-\frac{\kappa}{m}.
\]
Equivalently, since $\mathcal{F}[\psi_{x_0}]>0$ and $\kappa/m>0$, we may determine the sign of
\[
\mathcal{F}[\psi_{x_0}]^{-1}-\frac{m}{\kappa},
\]
which has the opposite sign. 

Now use \eqref{eqn:calF-x0x1-ABC}, \eqref{ABC-system-fs} (first and third equation) and \eqref{eqn:A-fs-calc2} to compute $\mathcal{F}[\psi_{x_0}]$. We find
\begin{align*}
\mathcal{F}[\psi_{x_0}]^{-1}
&=A(\rho)-\frac{m\dm}{\kappa}-\frac{\left(\eta A(\rho)+\frac{m}{(m-1)\kappa \|c_\rho\|}\right)^2}{A(\rho)}\\
&=A(\rho)-\frac{m\dm}{\kappa}-\eta^2 A(\rho)-\frac{2\eta m }{(m-1)\kappa \|c_\rho\|}-\frac{m^2}{(m-1)^2\kappa^2 \|c_\rho\|^2 A(\rho)}\\
&=(1-\eta^2)A(\rho)-\frac{m\dm}{\kappa}-\frac{2\eta m }{(m-1)\kappa \|c_\rho\|}-\frac{m^2}{(m-1)^2\kappa^2 \|c_\rho\|^2 A(\rho)}\\
&=-\frac{m}{\kappa}\left(-\frac{1}{m-1}-\frac{\eta}{(m-1)\|c_\rho\|}-\dm\right)-\frac{m\dm}{\kappa}-\frac{2\eta m }{(m-1)\kappa \|c_\rho\|}-\frac{m^2}{(m-1)^2\kappa^2 \|c_\rho\|^2 A(\rho)}\\
&=\frac{m}{\kappa(m-1)}-\frac{\eta m}{(m-1)\kappa \|c_\rho\|}-\frac{m^2}{(m-1)^2\kappa^2 \|c_\rho\|^2 A(\rho)}.
\end{align*}
Then, using \eqref{eqn:A-fs-calc2} again, we compute 
\begin{align*}
\mathcal{F}[\psi_{x_0}]^{-1}-\frac{m}{\kappa}&=\frac{m(2-m)}{\kappa(m-1)}-\frac{\eta m}{(m-1)\kappa \|c_\rho\|}-\frac{m^2}{(m-1)^2\kappa^2 \|c_\rho\|^2 A(\rho)}\\
&=\frac{m}{(m-1)\kappa}\left(2-m-\frac{\eta}{\|c_\rho\|}-\frac{m}{(m-1)\kappa \|c_\rho\|^2 A(\rho)}\right)\\
&=\frac{m}{(m-1)\kappa}\left(2-m-\frac{\eta}{\|c_\rho\|}-\frac{\eta^2-1}{ \|c_\rho\|^2 \left( -1-\frac{\eta}{\|c_\rho\|}-\dm(m-1)\right)}\right)\\
&=-\frac{m}{(m-1)\kappa \|c_\rho\|}\left( \|c_\rho\|(m-2)+\eta+\frac{\eta^2-1}{ -\|c_\rho\| -\eta-\|c_\rho\|\dm(m-1)}\right)\\
&=-\frac{m}{(m-1)\kappa \|c_\rho\|(-\|c_\rho\|-\eta- \|c_\rho\|\dm(m-1))}\\[5pt]
&\hspace{2cm}\times\left((\|c_\rho\|(m-2)+\eta)(-\|c_\rho\|-\eta-\|c_\rho\|\dm(m-1))+\eta^2-1\right).
\end{align*}

Therefore,
\begin{align*}
&\mathrm{sgn}\left(\mathcal{F}[\psi_{x_0}]^{-1}-\frac{m}{\kappa}\right) = \\
&\quad \mathrm{sgn}\left(-\frac{1}{-\|c_\rho\|-\eta- \|c_\rho\|\dm(m-1)}\left((\|c_\rho\|(m-2)+\eta)(-\|c_\rho\|-\eta-\|c_\rho\|\dm(m-1))+\eta^2-1\right)\right).
\end{align*}
As $A(\rho)>0$, by $\eqref{eqn:A-fs-calc2}$ we infer that $-\|c_\rho\|-\eta-\|c_\rho\|\dm(m-1)>0$. Hence,
\begin{equation}
\label{eqn:sign-diff}
\mathrm{sgn}\left(\mathcal{F}[\psi_{x_0}]^{-1}-\frac{m}{\kappa}\right)=-\mathrm{sgn}\left((\|c_\rho\|(m-2)+\eta)(-\|c_\rho\|-\eta-\|c_\rho\|\dm(m-1))+\eta^2-1\right).
\end{equation}
It remains to investigate the sign of the right-hand-side of \eqref{eqn:sign-diff}.

By Lemma \ref{lem:H-monotone}, the function $H(\eta)$ given by \eqref{eqn:H} is increasing when $m>2$.  Also, we can show that $H(\eta)>0$, by the following argument. The constants outside the integrals are positive, and the second integral is clearly positive. Therefore, it suffices to show that the first integral is positive. To do this, we split the integral into two parts and make the change of variable $\theta\to\pi-\theta$ in the second one, to get
\begin{align*}
&\int_0^\pi (\cos\theta-\eta)^{\frac{1}{m-1}}\sin^{\dm-1}\theta \cos\theta\d\theta\\
=&\int_0^{\frac{\pi}{2}} (\cos\theta-\eta)^{\frac{1}{m-1}}\sin^{\dm-1}\theta \cos\theta\d\theta
+\int_{\frac{\pi}{2}}^\pi (\cos\theta-\eta)^{\frac{1}{m-1}}\sin^{\dm-1}\theta \cos\theta\d\theta\\
=&\int_0^{\frac{\pi}{2}} (\cos\theta-\eta)^{\frac{1}{m-1}}\sin^{\dm-1}\theta \cos\theta\d\theta
-\int_{0}^{\frac{\pi}{2}} (-\cos\theta-\eta)^{\frac{1}{m-1}}\sin^{\dm-1}\theta \cos\theta\d\theta\\
=&\int_0^{\frac{\pi}{2}}\left((\cos\theta-\eta)^{\frac{1}{m-1}}-(-\cos\theta-\eta)^{\frac{1}{m-1}}\right)\sin^{\dm-1}\theta\cos\theta\d\theta.
\end{align*}
Since $\frac{1}{m-1}>0$ and
\[
\cos\theta-\eta\ge -\cos\theta-\eta,
\qquad
0\le\theta\le\frac{\pi}{2},
\]
the integrand is nonnegative. Moreover, it is strictly positive except at $\theta=\frac{\pi}{2}$. Therefore, the first integral is strictly positive, and consequently $H(\eta)>0$.

These facts imply that
\begin{align*}
0&<\frac{\d}{\d\eta}\ln H(\eta) \\
& =-\frac{1}{m-1}\frac{\int_0^\pi (\cos\theta-\eta)^{\frac{1}{m-1}-1}\sin^{\dm-1}\theta \cos\theta\d\theta}{\int_0^\pi (\cos\theta-\eta)^{\frac{1}{m-1}}\sin^{\dm-1}\theta \cos\theta\d\theta}-\frac{m-2}{m-1}\frac{\int_0^\pi (\cos\theta-\eta)^{\frac{1}{m-1}-1}\sin^{\dm-1}\theta \d\theta}{\int_0^\pi (\cos\theta-\eta)^{\frac{1}{m-1}}\sin^{\dm-1}\theta \d\theta} \\[2pt]
& = -\frac{1}{m-1} \left( \frac{B(\rho)}{C(\rho)-B(\rho)\eta}+(m-2)\frac{A(\rho)}{B(\rho)-A(\rho)\eta}\right),
\end{align*}
where for the second equal sign we used \eqref{eqn:A-fs-calc1}, \eqref{eqn:B-fs-calc1} (first line) and \eqref{eqn:C-fs-calc1} (first line), and for the integrals in the denominators we wrote 
\begin{align*}
(\cos \theta - \eta)^{\frac{1}{m-1}} &= (\cos \theta - \eta)^{\frac{2-m}{m-1}} (\cos \theta -\eta) \\[2pt]
&= (\cos \theta - \eta)^{\frac{2-m}{m-1}} \cos \theta - (\cos \theta - \eta)^{\frac{2-m}{m-1}} \eta.
\end{align*}

We infer that
\[
\frac{B(\rho)}{C(\rho)-B(\rho)\eta}+(m-2)\frac{A(\rho)}{B(\rho)-A(\rho)\eta}<0,
\]
which by \eqref{ABC-system-fs} (first and second equations) can be simplified into
\[
\frac{(m-1)\kappa B(\rho)}{m}+(m-2)\frac{(m-1)\kappa \|c_\rho\| A(\rho)}{m}<0.
\]
From the above we infer
\[
(m-2)\|c_\rho\| A(\rho)+B(\rho)<0,
\]
and by using the first equation in \eqref{ABC-system-fs} again, we get
\[
A(\rho)(\eta+(m-2)\|c_\rho\|)+\frac{m}{(m-1)\kappa \|c_\rho\|}<0.
\]

Finally, use \eqref{eqn:A-fs-calc2} to find
\[
\frac{1}{\eta^2-1}\frac{m}{\kappa}\left(-\frac{1}{m-1}-\frac{\eta}{(m-1)\|c_\rho\|}-\dm\right)(\eta+(m-2)\|c_\rho\|)+\frac{m}{(m-1)\kappa \|c_\rho\|}<0,
\]
which in turn can be simplified into
\[
(\eta+(m-2)\|c_\rho\|)\left(-\|c_\rho\|-\eta-\dm(m-1)\|c_\rho\|\right)+\eta^2-1<0.
\]
We conclude that
\[
\mathrm{sgn}\left((\eta+(m-2) \|c_\rho\|)(-\|c_\rho\|-\eta-\|c_\rho\|\dm(m-1))+\eta^2-1\right)=-1,
\]
and back to \eqref{eqn:sign-diff}, we find
\[
\mathrm{sgn}\left(\mathcal{F}[\psi_{x_0}]^{-1}-\frac{m}{\kappa}\right)=+1.
\]
Hence, 
\[
\mathcal{F}[\psi_{x_0}]^{-1}>\frac{m}{\kappa},\quad\text{or equivalently}\quad \mathcal{F}[\psi_{x_0}]<\frac{\kappa}{m}.
\]
Finally, we conclude that
\begin{equation}
\label{eqn:minlkm}
\min(\mathcal{F}[\psi_{x_0}],\mathcal{F}[\psi_{x_1}])<\frac{\kappa}{m}.
\end{equation}

The considerations above can be applied to the fully supported equilibrium $\rho_{\kappa,2}$ from Proposition \ref{prop:bif-mg2} to infer the following result.

\begin{theorem}[Stability of fully supported equilibria]
\label{thm:mg2:fs-stab}
The fully supported equilibrium $\rho_{\kappa,2}$ from Theorem \ref{thm:equigenerald} (see Proposition \ref{prop:bif-mg2}, equation \eqref{eqn:rhok-fs-mg2} and also the red dashed line in Figure \ref{fig:m35-splot}(b)) is unstable; note that this equilibrium exists for $\kappa_2<\kappa<\kappa_1$. 
\end{theorem}
\begin{proof}
The proof follows from the considerations above, replacing $\rho$ by $\rho_{\kappa,2}$. By \eqref{eqn:minlkm} and Theorem \ref{thm:minF} part (a), we conclude that $\rho_{\kappa,2}$ is unstable.
\end{proof}


\subsection{Strictly supported equilibria}
\label{subsect:mg2:cs}

We now investigate the stability of the strictly supported equilibria -- equilibria of type b) in Section \ref{subsect:cp}. Here, $\rho$ is given by  \eqref{eqn:equil-cs} or equivalently by \eqref{eqn:equil-cs-eta}, with the latter expression being more convenient as it uses the $\phi$ notation (i.e., the size of the support).

Similar to the fully-supported case, we will investigate $\mathcal{F}[\psi_{x_0}]$ and $\mathcal{F}[\psi_{x_1}]$. From \eqref{eqn:calF-x0x1-ABC} and \eqref{ABC-system} (third equation), we can obtain $\mathcal{F}[\psi_{x_1}]^{-1}$ directly:
\[
\mathcal{F}[\psi_{x_1}]^{-1}=\frac{1}{\dm}\left(A(\rho)-C(\rho)\right)=\frac{m}{\kappa}.
\]
Now, we use \eqref{ABC-system}, \eqref{eqn:cs-A}, and \eqref{eqn:calF-x0x1-ABC} to simplify $\mathcal{F}[\psi_{x_0}]^{-1}$ as follows:
\begin{align*}
\mathcal{F}[\psi_{x_0}]^{-1}&=C(\rho)-\frac{B(\rho)^2}{A(\rho)}\\
&=A(\rho)-\frac{m\dm}{\kappa}-\frac{\left(\cos\phi A(\rho)+\frac{m}{(m-1)\kappa \|c_\rho\|}\right)^2}{A(\rho)}\\
&=A(\rho)-\frac{m\dm}{\kappa}-\cos^2\phi A(\rho)-\frac{2m \cos\phi }{(m-1)\kappa \|c_\rho\|}-\frac{m^2}{(m-1)^2\kappa^2 \|c_\rho\|^2 A(\rho)}\\
&=(1-\cos^2\phi)A(\rho)-\frac{m\dm}{\kappa}-\frac{2m\cos\phi }{(m-1)\kappa \|c_\rho\|}-\frac{m^2}{(m-1)^2\kappa^2 \|c_\rho\|^2 A(\rho)}\\
&=-\frac{m}{\kappa}\left(-\frac{1}{m-1}-\frac{\cos\phi}{(m-1)\|c_\rho\|}-\dm\right)-\frac{m\dm}{\kappa}-\frac{2m \cos\phi }{(m-1)\kappa \|c_\rho\|}-\frac{m^2}{(m-1)^2\kappa^2 \|c_\rho\|^2 A(\rho)}\\
&=\frac{m}{\kappa(m-1)}-\frac{m \cos\phi }{(m-1)\kappa \|c_\rho\|}-\frac{m^2}{(m-1)^2\kappa^2 \|c_\rho\|^2 A(\rho)}.
\end{align*}
Then, we have
\begin{align*}
\mathcal{F}[\psi_{x_0}]^{-1}-\frac{m}{\kappa}&=\frac{m(2-m)}{\kappa(m-1)}-\frac{m \cos\phi }{(m-1)\kappa \|c_\rho\|}-\frac{m^2}{(m-1)^2\kappa^2 \|c_\rho\|^2 A(\rho)}\\
&=\frac{m}{(m-1)\kappa}\left(2-m-\frac{\cos\phi}{\|c_\rho\|}-\frac{m}{(m-1)\kappa \|c_\rho\|^2 A(\rho)}\right)\\
&=\frac{m}{(m-1)\kappa}\left(2-m-\frac{\cos\phi}{\|c_\rho\|}-\frac{\sin^2 \phi}{ \|c_\rho\|^2 \left(1+\frac{\cos\phi}{\|c_\rho\|}+\dm(m-1)\right)}\right)\\
&=-\frac{m}{(m-1)\kappa \|c_\rho\|}\left((m-2) \| c_\rho \|+\cos\phi+\frac{\sin^2 \phi}{ \|c_\rho\|+\cos\phi+\dm(m-1)\|c_\rho\|}\right)\\
&=-\frac{m}{(m-1)\kappa \|c_\rho\|(\|c_\rho\|+\cos\phi+ \dm(m-1)\|c_\rho\|)}\\[5pt]
&\hspace{2cm}\times\left(((m-2)\|c_\rho\|+\cos\phi)(\|c_\rho\|+\cos\phi+\dm(m-1)\|c_\rho\|)+\sin^2\phi\right),
\end{align*}
where for the third equal sign we used \eqref{eqn:cs-A}.

From \eqref{eqn:cs-A}, we have
\[
\mathrm{sgn}\left(\|c_\rho\|+\cos\phi+\|c_\rho\|\dm(m-1)\right)=\mathrm{sgn}(A(\rho))=+1,
\]
which combined with the calculation above yields
\begin{align}
\begin{aligned}\label{eqn:F-m/k}
\mathrm{sgn}\left(\mathcal{F}[\psi_{x_0}]^{-1}-\frac{m}{\kappa}\right)
&=-\mathrm{sgn}\left(( (m-2)\|c_\rho\|+\cos\phi)(\|c_\rho\|+\cos\phi+\dm(m-1) \|c_\rho\|)+\sin^2\phi\right).
\end{aligned}
\end{align}

Remarkably, the sign of the expression on the r.h.s. of \ref{eqn:F-m/k} was investigated in the proof of Proposition \ref{prop:sgnFp} (see \eqref{eqn:G} and \eqref{eqn:signG}) to determine the monotonicity of the function $F(\phi)$. By \eqref{eqn:F-m/k} and \eqref{eqn:sgnF-2} we have
\begin{equation}
\label{eqn:sgnFpG}
\mathrm{sgn}\left(\mathcal{F}[\psi_{x_0}]^{-1}-\frac{m}{\kappa}\right)=-\mathrm{sgn}\left(F'(\phi)\right),
\end{equation}
which implies
\begin{equation}
\label{eqn:signdiff-sgnFp}
\mathrm{sgn}\left(\mathcal{F}[\psi_{x_0}]-\frac{\kappa}{m}\right)=\mathrm{sgn}\left(F'(\phi)\right).
\end{equation}

\begin{theorem}[Stability of strictly supported equilibria]
\label{thm:mg2:cs-stab}
The strictly supported equilibrium $\rho_{\kappa,1}$ from Theorem \ref{thm:equigenerald} (see Proposition \ref{prop:bif-mg2}, equation \eqref{eqn:rhok-cs-mg2} and also the black solid line in Figure \ref{fig:m35-splot}(b)) is stable; note that this equilibrium exists for all $\kappa>\kappa_3$. On the other hand, the strictly supported equilibrium $\rho_{\kappa,2}$, which exists for $\kappa_3<\kappa<\kappa_2$ (see equation \eqref{eqn:rhok-cs-mg2} and the black dashed line in Figure \ref{fig:m35-splot}(b)), is unstable.
\end{theorem}
\begin{proof}
The proof follows from \eqref{eqn:signdiff-sgnFp} and Proposition \ref{prop:sgnFp}. We distinguish two cases.
\smallskip

\noindent(Case 1: $\mathrm{sgn}\left(F'(\phi)\right)>0$) This corresponds to the equilibrium $\rho_{\kappa,1}$, for which $0<\phi < \bar{\phi}$. By the calculations above (replace $\rho$ by $\rho_{\kappa,1}$), we get from \eqref{eqn:signdiff-sgnFp} that
\[
\min(\mathcal{F}[\psi_{x_0}],\mathcal{F}[\psi_{x_1}])=\frac{\kappa}{m}.
\]
Using Theorem \ref{thm:minF} part (c) we infer that $\rho_{\kappa,1}$ is stable.
\medskip

\noindent(Case 2: $\mathrm{sgn}\left(F'(\phi)\right)<0$) This corresponds to the equilibrium $\rho_{\kappa,2}$, for which $\bar{\phi}<\phi < \pi$, as $\kappa$ ranges in $\kappa_3<\kappa<\kappa_2$. In this case (use calculations above with $\rho$ replaced by $\rho_{\kappa,2}$), from \eqref{eqn:signdiff-sgnFp} we find
\[
\min(\mathcal{F}[\psi_{x_0}],\mathcal{F}[\psi_{x_1}])<\frac{\kappa}{m}.
\]
Then, by Theorem \ref{thm:minF} part (a) we conclude that $\rho_{\kappa,2}$ is unstable.
\end{proof}

\begin{remark}
For $m>2$, at $\kappa=\kappa_1$ we have a subcritical pitchfork bifurcation, whereas at $\kappa=\kappa_3$ we note a saddle-node bifurcation. At $\kappa=\kappa_1$, the uniform distribution becomes unstable by Proposition \ref{prop:stab-unif}, while a branch of unstable equilibria emerges (as $\kappa$ decreases through $\kappa_1$) by Theorem \ref{thm:mg2:fs-stab}. This is in contrast to the case $1<m<2$ described in Remark \ref{rmk:m12}, where the bifurcation at $\kappa=\kappa_1$ is supercritical pitchfork. On the other hand, at $\kappa=\kappa_3$, a pair of strictly supported equilibria is born, one stable and the other unstable (see Theorem \ref{thm:mg2:cs-stab}).
\end{remark}


\appendix
\section{$A(\rho)$, $B(\rho)$, and $C(\rho)$ for fully supported equilibria $\rho$}\label{appendix:fs}

These calculations support results in Section \ref{subsect:mg2:fs}, in particular the proof of Lemma \ref{lemma:ABC-fs}. Take a fully supported equilibrium in the form \eqref{eqn:equil-fs-eta}, and compute using \eqref{eqn:ABC}:
\begin{equation}
\label{eqn:A-fs-calc1}
A(\rho)
=\dm w_\dm\int_0^\pi\left(\frac{(m-1)\kappa \|c_\rho\|}{m}\right)^{\frac{2-m}{m-1}}\left(\cos\theta-\eta\right)^{\frac{2-m}{m-1}}\sin^{\dm-1}\theta \, \d\theta,
\end{equation}
and
\begin{equation}
\label{eqn:B-fs-calc1}
\begin{aligned}
B(\rho)
&=\dm w_\dm\int_0^\pi\left(\frac{(m-1)\kappa \|c_\rho\|}{m}\right)^{\frac{2-m}{m-1}}\left(\cos\theta-\eta\right)^{\frac{2-m}{m-1}}\sin^{\dm-1}\theta \cos\theta \, \d\theta\\
&=\dm w_\dm\int_0^\pi\left(\frac{(m-1)\kappa \|c_\rho\|}{m}\right)^{\frac{2-m}{m-1}}\left(\cos\theta-\eta\right)^{\frac{1}{m-1}}\sin^{\dm-1}\theta \, \d\theta\\
&\quad +\eta \, \dm w_\dm\int_0^\pi\left(\frac{(m-1)\kappa \|c_\rho\|}{m}\right)^{\frac{2-m}{m-1}}\left(\cos\theta-\eta\right)^{\frac{2-m}{m-1}}\sin^{\dm-1}\theta \,\d\theta\\
&=\frac{m}{(m-1)\kappa \|c_\rho\|} + \eta A(\rho),
\end{aligned}
\end{equation}
where in the calculation of $B(\rho)$ we wrote $\cos \theta = \cos \theta - \eta + \eta$ to get from the first to the second line, and used \eqref{eqn:1-comp} for the last line. This shows the first equation in \eqref{ABC-system-fs}.

Also, 
\begin{equation}
\label{eqn:C-fs-calc1}
\begin{aligned}
C(\rho)
&=\dm w_\dm\int_0^\pi\left(\frac{(m-1)\kappa \|c_\rho\|}{m}\right)^{\frac{2-m}{m-1}}\left(\cos\theta-\eta\right)^{\frac{2-m}{m-1}}\sin^{\dm-1}\theta \cos^2\theta\d\theta \\
&= A(\rho) - \dm w_\dm\int_0^\pi\left(\frac{(m-1)\kappa \|c_\rho\|}{m}\right)^{\frac{2-m}{m-1}}\left(\cos\theta-\eta\right)^{\frac{2-m}{m-1}}\sin^{\dm+1}\theta \d\theta,
\end{aligned}
\end{equation}
where we wrote $\cos^2 \theta = 1- \sin^2 \theta$, and used \eqref{eqn:A-fs-calc1}. 
By integration by parts, we get
\begin{align*}
&\int_0^\pi (\cos\theta-\eta)^{\frac{2-m}{m-1}}\sin^{\dm+1}\theta\d\theta\\
&\quad =\left[-(m-1)(\cos\theta-\eta)^{\frac{1}{m-1}}\sin^\dm\theta\right]_0^\pi+(m-1)\dm\int_0^\pi (\cos\theta-\eta)^{\frac{1}{m-1}}\sin^{\dm-1}\theta\cos\theta\d\theta\\
&\quad =(m-1)\dm\int_0^\pi (\cos\theta-\eta)^{\frac{1}{m-1}}\sin^{\dm-1}\theta\cos\theta\d\theta,
\end{align*}
which used together with \eqref{eqn:s-comp} in \eqref{eqn:C-fs-calc1}, leads to
\begin{equation}
\label{eqn:C-fs-calc2}
\begin{aligned}
   C(\rho) &= A(\rho) - \dm w_\dm \left(\frac{(m-1)\kappa \|c_\rho\|}{m}\right)^{\frac{2-m}{m-1}} (m-1) \dm \int_0^\pi (\cos\theta-\eta)^{\frac{1}{m-1}}\sin^{\dm-1}\theta\cos\theta\d\theta \\
   & = A(\rho) - \frac{m}{(m-1)\kappa \|c_\rho\|} (m-1) \dm \|c_\rho\| \\
   & = A(\rho) - \frac{m \dm }{\kappa}.
\end{aligned}
\end{equation}
This is the third identity in \eqref{ABC-system-fs}.

We can also establish a relationship between $B(\rho)$ and $C(\rho)$ in the following way. In the expression of $C(\rho)$ in \eqref{eqn:C-fs-calc1} (the first line), write $\cos^2 \theta = (\cos \theta - \eta + \eta) \cos \theta$ and compute using  \eqref{eqn:s-comp} and \eqref{eqn:B-fs-calc1} (first line):
\begin{equation}
\label{eqn:C-fs-calc3}
\begin{aligned}
C(\rho)&= \dm w_\dm\int_0^\pi\left(\frac{(m-1)\kappa \|c_\rho\|}{m}\right)^{\frac{2-m}{m-1}}\left(\cos\theta-\eta\right)^{\frac{1}{m-1}}\sin^{\dm-1}\theta \cos\theta\d\theta \\
&\quad + \eta \, \dm w_\dm\int_0^\pi\left(\frac{(m-1)\kappa \|c_\rho\|}{m}\right)^{\frac{2-m}{m-1}}\left(\cos\theta-\eta\right)^{\frac{2-m}{m-1}}\sin^{\dm-1}\theta \cos\theta\d\theta \\
&=\frac{m}{(m-1)\kappa}+\eta B(\rho).
\end{aligned}
\end{equation}
The second relationship in \eqref{ABC-system-fs} is now shown as well.

Finally, by \eqref{eqn:C-fs-calc3} and \eqref{eqn:B-fs-calc1} (last line), we obtain
\[
C(\rho)=\frac{m}{(m-1)\kappa}+\frac{m\eta}{(m-1)\kappa \|c_\rho\|}+\eta^2 A(\rho),
\]
which combined with \eqref{eqn:C-fs-calc2}, yields
\begin{equation}
A(\rho)-\frac{m\dm}{\kappa}=\frac{m}{(m-1)\kappa}+\frac{m\eta}{(m-1)\kappa \|c_\rho\|}+\eta^2 A(\rho).
\end{equation}
Solving for $A(\rho)$ in the above then gives \eqref{eqn:A-fs-calc2}.
\smallskip

\section{$A(\rho)$, $B(\rho)$, and $C(\rho)$ for strictly supported equilibria $\rho$}
\label{appendix:ps}

These calculations are in support of Lemma \ref{lemma:ABC-cs}. Consider a strictly supported equilibrium in the form \eqref{eqn:equil-cs-eta}, and substitute \eqref{eqn:equil-cs-eta} into \eqref{eqn:ABC} to compute $A(\rho)$, $B(\rho)$, and $C(\rho)$. We get
\begin{equation}
\label{eqn:calcA-int}
A(\rho)=\dm w_\dm\int_0^\phi\left(\frac{(m-1)\kappa \|c_\rho\|}{m}\right)^{\frac{2-m}{m-1}}\left(\cos\theta-\cos\phi\right)^{\frac{2-m}{m-1}}\sin^{\dm-1}\theta\d\theta
\end{equation}
and
\begin{align*}
B(\rho)&=\dm w_\dm\int_0^\phi\left(\frac{(m-1)\kappa \|c_\rho\|}{m}\right)^{\frac{2-m}{m-1}}\left(\cos\theta-\cos\phi\right)^{\frac{2-m}{m-1}}\sin^{\dm-1}\theta\cos\theta\d\theta\\
&=\dm w_\dm\int_0^\phi\left(\frac{(m-1)\kappa \|c_\rho\|}{m}\right)^{\frac{1}{m-1}}\left(\cos\theta-\cos\phi\right)^{\frac{2-m}{m-1}}\sin^{\dm-1}\theta\d\theta\\
&+\dm w_\dm\int_0^\phi\left(\frac{(m-1)\kappa \|c_\rho\|}{m}\right)^{\frac{2-m}{m-1}}\left(\cos\theta-\cos\phi\right)^{\frac{2-m}{m-1}}\sin^{\dm-1}\theta\cos\phi\d\theta\\
&=\frac{m}{(m-1)\kappa\|c_\rho\|}+\cos\phi A(\rho),
\end{align*}
where in the calculation of $B(\rho)$ we used $\cos\theta=(\cos\theta-\cos\phi)+\cos\phi$ in the second equality, together with \eqref{eqn:equil-cs-1} and \eqref{eqn:calcA-int} in the third equality. The first relationship in \eqref{ABC-system} is now shown.

For the second identity in \eqref{ABC-system}, we use again $\cos\theta=(\cos\theta-\cos\phi)+\cos\phi$ to get
\begin{align*}
C(\rho)&=\dm w_\dm\int_0^\phi\left(\frac{(m-1)\kappa \|c_\rho\|}{m}\right)^{\frac{2-m}{m-1}}\left(\cos\theta-\cos\phi\right)^{\frac{2-m}{m-1}}\sin^{\dm-1}\theta\cos^2\theta\d\theta\\
&=\dm w_\dm\int_0^\phi\left(\frac{(m-1)\kappa \|c_\rho\|}{m}\right)^{\frac{2-m}{m-1}}\left(\cos\theta-\cos\phi\right)^{\frac{1}{m-1}}\sin^{\dm-1}\theta\cos\theta\d\theta\\
&+\dm w_\dm\int_0^\phi\left(\frac{(m-1)\kappa \|c_\rho\|}{m}\right)^{\frac{2-m}{m-1}}\left(\cos\theta-\cos\phi\right)^{\frac{2-m}{m-1}}\sin^{\dm-1}\theta\cos\theta\cos\phi\d\theta.
\end{align*}
If we further use \eqref{eqn:equil-cs-1} and the definition of $B(\rho)$, we get
\[
C(\rho)=\frac{m}{(m-1)\kappa}+\cos\phi B(\rho).
\]

For the third relationship, we calculate
\begin{align*}
A(\rho)-C(\rho)&=\dm w_\dm\left(\frac{m}{(m-1)\kappa \|c_\rho\|}\right)\int_0^\phi\left(\frac{(m-1)\kappa \|c_\rho\|}{m}\right)^{\frac{1}{m-1}}\left(\cos\theta-\cos\phi\right)^{\frac{2-m}{m-1}}\sin^{\dm+1}\theta\d\theta.
\end{align*}
To simplify the above, we use integration by parts:
\begin{align*}
&\int_0^\phi (\cos\theta-\cos\phi)^{\frac{1}{m-1}-1}\sin^{\dm+1}\theta\d\theta\\
&=-\left[(m-1)(\cos\theta-\cos\phi)^{\frac{1}{m-1}}\sin^\dm\theta\right]^{\phi}_0+(m-1)\dm\int_0^\phi(\cos\theta-\cos\phi)^{\frac{1}{m-1}}\sin^{\dm-1}\theta\cos\theta\d\theta\\
&=(m-1)\dm\int_0^\phi(\cos\theta-\cos\phi)^{\frac{1}{m-1}}\sin^{\dm-1}\theta\cos\theta\d\theta,
\end{align*}
which yields
\begin{align*}
A(\rho)-C(\rho) &=    \dm^2 w_\dm\left(\frac{m}{\kappa \|c_\rho\|}\right)\int_0^\phi \left(\frac{(m-1)\kappa \|c_\rho\|}{m}\right)^{\frac{1}{m-1}} (\cos\theta-\cos\phi)^{\frac{1}{m-1}}\sin^{\dm-1}\theta\cos\theta\d\theta \\
&=\frac{\dm m}{\kappa},
\end{align*}
where for the second equal sign we used \eqref{eqn:equil-cs-s}.

Finally, \eqref{eqn:cs-A} follows by a direct calculation from \eqref{ABC-system}. These calculations prove Lemma \ref{lemma:ABC-cs}.

\section{Supporting results for the proof of Proposition \ref{prop:sgnFp}}
\label{appendix:signG}
\begin{lemma}
\label{appendix:lemma1}
Let $s(\phi)$ be defined as \eqref{def:s}. Then, we have
\begin{align}\label{eqn:sprime}
\frac{\d s}{\d\phi}=\frac{1+\dm(m-1)s \cos\phi -(1+\dm(m-1))s^2}{(m-1)\sin\phi},\qquad\forall \phi\in(0, \pi).
\end{align}
\end{lemma}
\begin{proof}
First, we use the quotient rule to get from \eqref{def:s}:
\begin{align}
\begin{aligned}\label{eqn:sprimefirst}
\frac{\d s}{\d\phi}&=\frac{\sin\phi}{m-1}\frac{\int_0^\phi(\cos\theta-\cos\phi)^{\frac{2-m}{m-1}}\sin^{\dm-1}\theta\cos\theta\d\theta}{\int_0^\phi
(\cos\theta-\cos\phi)^{\frac{1}{m-1}}
\sin^{d-1}\theta\,d\theta}\\
&\quad -\frac{\sin\phi}{m-1}\frac{\left(\int_0^\phi(\cos\theta-\cos\phi)^{\frac{1}{m-1}}\sin^{\dm-1}\theta\cos\theta\d\theta\right)\left(\int_0^\phi(\cos\theta-\cos\phi)^{\frac{2-m}{m-1}}\sin^{\dm-1}\theta\d\theta\right)}{\left(\int_0^\phi
(\cos\theta-\cos\phi)^{\frac{1}{m-1}}
\sin^{d-1}\theta\,d\theta\right)^2}.
\end{aligned}
\end{align}
From \eqref{eqn:equil-cs-eta} and the definitions of $A(\rho)$, $B(\rho)$, and $C(\rho)$ in \eqref{eqn:ABC}, we can write the r.h.s. of \eqref{eqn:sprimefirst}
as
\begin{align*}
&\frac{\sin\phi}{m-1}\left(\frac{B(\rho)}{B(\rho)-\cos\phi A(\rho)}-\frac{(C(\rho)-\cos\phi B(\rho))A(\rho)}{(B(\rho)-\cos\phi A(\rho))^2}\right)\\
&=\frac{\sin\phi }{(m-1)(B(\rho)-\cos\phi A(\rho))^2}\left(
B(\rho)(B(\rho)-\cos\phi A(\rho))-A(\rho)(C(\rho)-\cos\phi B(\rho))
\right).
\end{align*}

Now simplify $B(\rho)-\cos\phi A(\rho)$ and $C(\rho)-\cos\phi B(\rho)$ from the first two equations in \eqref{ABC-system} to write the r.h.s. above as
\begin{align*}
\frac{\sin\phi }{(m-1)\left(\frac{m}{(m-1)\kappa s}\right)^2}\left(
\frac{m B(\rho)}{(m-1)\kappa s}-\frac{m A(\rho)}{(m-1)\kappa}
\right)&=\frac{\kappa\sin\phi }{m}(sB(\rho)-s^2 A(\rho)) \\
&=\frac{\kappa \sin\phi}{m}\left(\frac{m}{(m-1)\kappa}+(s\cos\phi-s^2)A(\rho)\right),
\end{align*}
where for the second equal sign we used the first equation in \eqref{ABC-system} to substitute for $B(\rho)$.

Finally, we use $\eqref{eqn:cs-A}$ to get
\begin{align*}
&\frac{\kappa \sin\phi}{m}\left(\frac{m}{(m-1)\kappa}+(s\cos\phi-s^2)A(\rho)\right) \\
&\qquad = \frac{\kappa \sin\phi}{m}\left(\frac{m}{(m-1)\kappa}+(s\cos\phi-s^2)\times \frac{m}{\kappa\sin^2\phi}\left(\frac{1}{m-1}+\frac{\cos\phi}{(m-1)s}+\dm\right)\right)\\
&\qquad =\sin\phi\left(\frac{1}{(m-1)}+(s\cos\phi-s^2)\times \frac{1}{\sin^2\phi}\left(\frac{1}{m-1}+\frac{\cos\phi}{(m-1)s}+\dm\right)\right).
\end{align*}

Going back to \eqref{eqn:sprimefirst}, from these calculations we find
\begin{equation}
\label{eqn:dsdphi-int}
\frac{\d s}{\d\phi}=\sin\phi\left(\frac{1}{(m-1)}+(s\cos\phi-s^2)\times \frac{1}{\sin^2\phi}\left(\frac{1}{m-1}+\frac{\cos\phi}{(m-1)s}+\dm\right)\right).
\end{equation}
Now multiply $(m-1)\sin\phi$ to both sides of \eqref{eqn:dsdphi-int} and simplify, to get
\begin{align*}
(m-1)\sin\phi\frac{\d s}{\d\phi}&=\sin^2\phi+(s\cos\phi-s^2)\left(1+\frac{\cos\phi}{s}+\dm(m-1)\right)\\
&=\sin^2\phi+(\cos\phi-s)((1+\dm(m-1))s+\cos\phi)\\
&=\sin^2\phi+\cos^2\phi+\dm(m-1)s\cos\phi-s^2(1+\dm(m-1))\\
&=1+\dm(m-1)s \cos\phi-(1+\dm(m-1))s^2.
\end{align*}
Therefore, we get the desired result.
\end{proof}

\begin{lemma}\label{lem:svd}
Let $s(\phi)$ be defined as \eqref{def:s}. Then, $s(\phi)$ is decreasing and
\begin{align}\label{slimvalue}
\lim_{\phi\to0+}s(\phi)=1,\qquad\lim_{\phi\to\pi-}s(\phi)=\frac{1}{1+\dm(m-1)}.
\end{align}
Furthermore, we get
\begin{align}\label{srange}
\frac{1}{1+\dm(m-1)}< s(\phi)< 1, \qquad\forall\phi\in(0, \pi).
\end{align}
\end{lemma}

\begin{proof}
\noindent {(\em The value of $s(\phi)$ near $0+$)} From the definition of $s$, we get
\begin{align*}
0\leq|1-s(\phi)|&=\left|
\frac{ \int_0^\phi
(\cos\theta-\cos\phi)^{\frac{1}{m-1}}
\sin^{d-1}\theta (1-\cos\theta)\d\theta}
{\int_0^\phi
(\cos\theta-\cos\phi)^{\frac{1}{m-1}}
\sin^{d-1}\theta\d\theta}\right|\\
&\leq 
\frac{ \int_0^\phi
(\cos\theta-\cos\phi)^{\frac{1}{m-1}}
\sin^{d-1}\theta (1-\cos\phi)\d\theta}
{\int_0^\phi
(\cos\theta-\cos\phi)^{\frac{1}{m-1}}
\sin^{d-1}\theta\d\theta}\\
&=1-\cos\phi.
\end{align*}
Therefore, we get
\[
\lim_{\phi\to0+}|1-s(\phi)|=0,
\]
which yields the first part of \eqref{slimvalue}.
\medskip

\noindent {(\em The value of $s(\phi)$ near $\pi-$)} Again, from the definition of $s$, we get
\begin{align*}
\lim_{\phi\to\pi}s(\phi)&=\frac{\int_0^\pi (\cos\theta+1)^{\frac{1}{m-1}}\sin^{\dm-1}\theta \cos\theta\d\theta}{\int_0^\pi (\cos\theta+1)^{\frac{1}{m-1}}\sin^{\dm-1}\theta\d\theta}\\
&=\frac{\int_0^\pi (\cos\theta+1)^{\frac{1}{m-1}+1}\sin^{\dm-1}\theta \d\theta}{\int_0^\pi (\cos\theta+1)^{\frac{1}{m-1}}\sin^{\dm-1}\theta\d\theta}-1.
\end{align*}
We rewrite the above form using
\[
\cos\theta+1=2\cos^2\left(\frac{\theta}{2}\right) \qquad\text{ and }\qquad \sin\theta=2\cos\left(\frac{\theta}{2}\right)\sin\left(\frac{\theta}{2}\right),
\]
to get
\begin{equation}
\label{eqn:limphipi}
\lim_{\phi\to\pi}s(\phi)=2\times\frac{\int_0^\pi \cos^{\frac{2}{m-1}+2+\dm-1}\left(\frac{\theta}{2}\right)\sin^{\dm-1}\left(\frac{\theta}{2}\right)\d\theta}{\int_0^\pi \cos^{\frac{2}{m-1}+\dm-1}\left(\frac{\theta}{2}\right)\sin^{\dm-1}\left(\frac{\theta}{2}\right)\d\theta}-1.
\end{equation}

Using the following formula:
\[
\int_0^\pi
\sin^a\left(\frac{\theta}{2}\right)
\cos^b\left(\frac{\theta}{2}\right)\,d\theta
=
\frac{
\Gamma\left(\frac{a+1}{2}\right)
\Gamma\left(\frac{b+1}{2}\right)
}{
\Gamma\left(\frac{a+b+2}{2}\right)}, \qquad\forall a, b>-1,
\] 
and properties of the Gamma function, we can further simplify the two integrals in \eqref{eqn:limphipi} as follows:
\begin{align*}
\int_0^\pi \cos^{\frac{2}{m-1}+2+\dm-1}\left(\frac{\theta}{2}\right)\sin^{\dm-1}\left(\frac{\theta}{2}\right)\d\theta &=\frac{\Gamma\left(\frac{1}{m-1}+1+\frac{\dm}{2}\right)\Gamma\left(\frac{\dm}{2}\right)}{\Gamma\left(\frac{1}{m-1}+1+\dm\right)}\\
&=\frac{\frac{1}{m-1}+\frac{\dm}{2}}{\frac{1}{m-1}+\dm}\times \frac{\Gamma\left(\frac{1}{m-1}+\frac{\dm}{2}\right)\Gamma\left(\frac{\dm}{2}\right)}{\Gamma\left(\frac{1}{m-1}+\dm\right)},
\end{align*}
and
\[
\int_0^\pi \cos^{\frac{2}{m-1}+\dm-1}\left(\frac{\theta}{2}\right)\sin^{\dm-1}\left(\frac{\theta}{2}\right)\d\theta=\frac{\Gamma\left(\frac{1}{m-1}+\frac{\dm}{2}\right)\Gamma\left(\frac{\dm}{2}\right)}{\Gamma\left(\frac{1}{m-1}+\dm\right)}.
\]
Finally, we have
\begin{align*}
\lim_{\phi\to\pi}s(\phi)&=2\times\frac{\frac{1}{m-1}+\frac{\dm}{2}}{\frac{1}{m-1}+\dm}-1=\frac{1}{1+\dm(m-1)}.
\end{align*}
\medskip

\noindent {(\em Monotonicity of $s$)} From \eqref{eqn:sprimefirst}, we have 
\begin{align*}
\mathrm{sgn}(s'(\phi))&=\mathrm{sgn}\bigg(
\left(\int_0^\phi(\cos\theta-\cos\phi)^{\frac{2-m}{m-1}}\sin^{\dm-1}\theta\cos\theta\d\theta\right)\left(\int_0^\phi(\cos\theta-\cos\phi)^{\frac{1}{m-1}}\sin^{\dm-1}\theta\d\theta\right)\\
&\hspace{1cm}-\left(\int_0^\phi(\cos\theta-\cos\phi)^{\frac{2-m}{m-1}}\sin^{\dm-1}\theta\d\theta\right)\left(\int_0^\phi(\cos\theta-\cos\phi)^{\frac{1}{m-1}}\sin^{\dm-1}\theta\cos\theta\d\theta\right)
\bigg).
\end{align*}
Let $\d\mu(\theta)=(\cos\theta-\cos\phi)^{\frac{2-m}{m-1}}\sin^{\dm-1}\theta \, \d\theta$ on $\theta\in[0, \phi]$, then we get
\begin{align*}
\mathrm{sgn}(s'(\phi))&=\mathrm{sgn}\bigg(
\left(\int\cos\theta\d\mu(\theta)\right)\left(\int(\cos\theta-\cos\phi)\d\mu(\theta)\right)\\
&\hspace{2cm}-\left(\int\d\mu(\theta)\right)\left(\int(\cos\theta-\cos\phi)\cos\theta\d\mu(\theta)\right)
\bigg)\\
&=\mathrm{sgn}\left(-\int\cos^2\theta\d\mu(\theta)\int\d\mu(\theta)+\left(\int\cos\theta\d\mu(\theta)\right)^2\right)\\
&=-1,
\end{align*}
where the last equal sign can be justified by Cauchy-Schwarz inequality. Therefore, we conclude that $s(\phi)$ is decreasing.
\end{proof}

\begin{lemma}\label{lem:negativeprime}
Let $m>2$ and $G(\phi)$ be defined by \eqref{eqn:G}. Assume that $G(\phi_0)=0$ for some $\phi_0\in(0, \pi)$. Then, $G'(\phi_0)<0$.
\end{lemma}
\begin{proof}
By rearranging its r.h.s., we can write \eqref{eqn:G} as
\begin{equation}
\label{eqn:G-mod}
G(\phi) = 1+(m-1)(\dm+1)\cos\phi s(\phi) 
+(m-2)(1+\dm(m-1))s(\phi)^2.
\end{equation}
Therefore, a root $\phi_0$ of $G$ satisfies
\[
1+(m-1)(\dm+1)\cos\phi_0 s(\phi_0) 
+(m-2)(1+\dm(m-1))s(\phi_0)^2 = 0,
\]
which yields
\begin{align}\label{cosphi0}
\cos\phi_0=-\frac{1+(m-2)(1+\dm(m-1))s(\phi_0)^2}{(m-1)(\dm+1)s(\phi_0)}.
\end{align}

From \eqref{eqn:sprime} and \eqref{cosphi0}, we calculate $s'(\phi_0)$ as follows:
\begin{align}
\begin{aligned}\label{sprime0}
s'(\phi_0)&=\frac{1+\dm(m-1)\cos\phi_0 s(\phi_0)-(1+\dm(m-1))s(\phi_0)^2}{(m-1)\sin\phi_0}\\
&=\frac{1}{(m-1)\sin\phi_0}\left(1-\frac{\dm}{\dm+1}\left(1+(m-2)(1+\dm(m-1))s(\phi_0)^2\right)-(1+\dm(m-1))s(\phi_0)^2\right)\\
&=\frac{1}{(m-1)\sin\phi_0}\left(\frac{1}{\dm+1}-(1+\dm(m-1))\left(\frac{\dm(m-2)}{\dm+1}+1\right) s(\phi_0)^2 \right) \\
&=\frac{1}{(m-1)\sin\phi_0}\left(\frac{1-(1+\dm(m-1))^2s(\phi_0)^2}{\dm+1}\right)\\
&=\frac{1-(1+\dm(m-1))^2s(\phi_0)^2}{(\dm+1)(m-1)\sin\phi_0}.
\end{aligned}
\end{align}

By \eqref{eqn:G-mod}, we calculate the derivative of $G$ at $\phi_0$ to get
\begin{equation}
\label{eqn:Gp-phi0}
\begin{aligned}
G'(\phi_0)&=(m-1)(\dm+1)\left(-\sin\phi_0 s(\phi_0)+\cos\phi_0 s'(\phi_0)\right)\\
&\hspace{3cm}+2(m-2)(1+\dm(m-1))s(\phi_0)s'(\phi_0)\\
&=-(m-1)(\dm+1)\sin\phi_0 s(\phi_0)\\
&\hspace{3cm}+\left((m-1)(\dm+1)\cos\phi_0+2(m-2)(1+\dm(m-1))s(\phi_0)\right)s'(\phi_0)\\
&:=\mathcal{I}_1+\mathcal{I}_2.
\end{aligned}
\end{equation}
In the second term $\mathcal{I}_2$, remove $\cos\phi_0$ and $s'(\phi_0)$ by using \eqref{cosphi0} and \eqref{sprime0}, to find
\begin{align*}
\mathcal{I}_2&=\left(-\frac{1+(m-2)(1+\dm(m-1))s(\phi_0)^2}{s(\phi_0)}+2(m-2)(1+\dm(m-1))s(\phi_0)\right) \\
&\qquad \times\left(\frac{1-(1+\dm(m-1))^2s(\phi_0)^2}{(\dm+1)(m-1)\sin\phi_0}\right)\\[5pt]
&=\frac{
\left(-1+(m-2)(1+\dm(m-1))s(\phi_0)^2\right)\left(1-(1+\dm(m-1))^2s(\phi_0)^2\right)
}{(\dm+1)(m-1)s(\phi_0)}\times\frac{1}{\sin\phi_0}.
\end{align*}

On the other hand, we can rewrite $\frac{1}{\sin\phi_0}$ using \eqref{cosphi0} as follows:
\begin{align*}
\frac{1}{\sin\phi_0}&=\frac{\sin\phi_0}{(1-\cos\phi_0)(1+\cos\phi_0)}\\
&=\sin\phi_0\left(\frac{(m-1)(\dm+1)s(\phi_0)}{
(m-1)(\dm+1)s(\phi_0)+1+(m-2)(1+\dm(m-1))s(\phi_0)^2
}\right)\\[2pt]
&\hspace{2cm}\times\left(\frac{(m-1)(\dm+1)s(\phi_0)}{
(m-1)(\dm+1)s(\phi_0)-1-(m-2)(1+\dm(m-1))s(\phi_0)^2
}\right)\\[2pt]
&=-\frac{(m-1)^2(\dm+1)^2s(\phi_0)^2\sin\phi_0}{\big((m-2)s(\phi_0)+1\big)\big((1+\dm(m-1))s(\phi_0)+1\big)\big((m-2)s(\phi_0)-1\big)\big((1+\dm(m-1))s(\phi_0)-1\big)}\\
&=-\frac{(m-1)^2(\dm+1)^2s(\phi_0)^2\sin\phi_0}{\big(1-(m-2)^2s(\phi_0)^2\big)\big(1-(1+\dm(m-1))^2s(\phi_0)^2\big)}.
\end{align*}
Combining the two calculations, we obtain
\[
\mathcal{I}_2=-\frac{(m-1)(\dm+1)s(\phi_0)\sin\phi_0\big(-1+(m-2)(1+\dm(m-1))s(\phi_0)^2\big)}{1-(m-2)^2s(\phi_0)^2}.
\]

Next, we calculate the sum of $\mathcal{I}_1$ and $\mathcal{I}_2$ as follows:
\begin{align}
\begin{aligned}\label{eqn:Gprime}
G'(\phi_0)&=\mathcal{I}_1+\mathcal{I}_2\\
&=-(m-1)(\dm+1)s(\phi_0)\sin\phi_0-\frac{(m-1)(\dm+1)s(\phi_0)\sin\phi_0\big(-1+(m-2)(1+\dm(m-1))s(\phi_0)^2\big)}{1-(m-2)^2s(\phi_0)^2}\\[3pt]
&=-(m-1)(\dm+1)s(\phi_0)\sin\phi_0\left(
1+\frac{-1+(m-2)(1+\dm(m-1))s(\phi_0)^2}{1-(m-2)^2s(\phi_0)^2}
\right)\\[3pt]
&=-\frac{(m-1)(\dm+1)s(\phi_0)\sin\phi_0}{1-(m-2)^2s(\phi_0)^2}\left(
1-(m-2)^2s(\phi_0)^2-1+(m-2)(1+\dm(m-1))s(\phi_0)^2
\right)\\[3pt]
&=-\frac{(m-1)(\dm+1)(m-2)\sin\phi_0 s(\phi_0)^3\left((m-1)(\dm-1)+2\right)}{1-(m-2)^2s(\phi_0)^2}.
\end{aligned}
\end{align}
From \eqref{cosphi0}, we have
\[
-\frac{1+(m-2)(1+\dm(m-1))s(\phi_0)^2}{(m-1)(\dm+1)s(\phi_0)}=\cos\phi_0\geq -1,
\]
which yields
\begin{align}
\begin{aligned}\label{phi0ineq}
0&\geq 1-(m-1)(\dm+1)s(\phi_0)
+(m-2)(1+\dm(m-1))s(\phi_0)^2\\
&=(m-2)(1+\dm(m-1))\left(
s(\phi_0)-\frac{1}{m-2}
\right)\left(
s(\phi_0)-\frac{1}{1+\dm(m-1)}
\right).
\end{aligned}
\end{align}

Finally, by \eqref{srange} we know
\[
s(\phi)>\frac{1}{1+\dm(m-1)},\qquad\forall \phi\in(0, \pi),
\]
which we combine with \eqref{phi0ineq} to get
\[
s(\phi_0)<\frac{1}{m-2}.
\]
Therefore, 
\[
1-(m-2)^2s(\phi_0)^2>0,
\]
and it implies that the denominator of the expression of $G'(\phi_0)$ in  \eqref{eqn:Gprime} is positive. Furthermore, the numerator of the expression is also positive since $m>2$, $\sin\phi_0>0$, and $s(\phi_0)>0$, and we can conclude that
\[
G'(\phi_0)<0.
\]
\end{proof}

\begin{lemma}\label{lem:Gnearpi}
Let $m>2$. The function $G(\phi)$ given by \eqref{eqn:G} satisfies
\[
\lim_{\phi\to\pi-}\frac{G(\phi)}{\sin^2\phi}<0.
\]
\end{lemma}

\begin{proof}
To investigate the given limit, we investigate the limit of $\frac{s'(\phi)}{\sin\phi}$ first. Note that by \eqref{eqn:sprime} and \eqref{slimvalue}, we have
\[
\lim_{\phi\to\pi-}s'(\phi)=0.
\]
By \eqref{eqn:sprime} and l'H\^{o}pital's rule, we get
\begin{align*}
\lim_{\phi\to\pi^-}\frac{s'(\phi)}{\sin\phi}
&=\lim_{\phi\to\pi^-}\frac{1+\dm(m-1) s(\phi) \cos\phi -(1+\dm(m-1))s(\phi)^2}{
(m-1)\sin^2\phi}\\
&=\lim_{\phi\to\pi^-}
\frac{ -\dm(m-1)s(\phi) \sin\phi + \dm(m-1) s'(\phi) \cos\phi -2(1+\dm(m-1))s(\phi)s'(\phi)}{2(m-1)\sin\phi\cos\phi}\\
&=\lim_{\phi\to\pi^-}\frac{-\dm(m-1)s(\phi)+\left(
\dm(m-1)\cos\phi-2(1+\dm(m-1))s(\phi)\right)\frac{s'(\phi)}{\sin\phi}}{
2(m-1)\cos\phi}.
\end{align*}
Since $\lim_{\phi\to\pi-}s(\phi)=\frac{1}{1+\dm(m-1)}$ (cf., \eqref{slimvalue}), we get
\begin{align*}
\lim_{\phi\to\pi^-}\frac{s'(\phi)}{\sin\phi}=\frac{-\frac{\dm(m-1)}{1+\dm(m-1)}+\left(-\dm(m-1)-2\right)\lim_{\phi\to\pi-}\frac{s'(\phi)}{\sin\phi}}{-2(m-1)},
\end{align*}
from which we can find 
\begin{equation}
\label{eqn:sposinphi}
\lim_{\phi\to\pi-}\frac{s'(\phi)}{\sin\phi}=-\frac{\dm(m-1)}{\left(1+\dm(m-1)\right)\left((\dm-2)(m-1)+2\right)}.
\end{equation}

Now, by l'H\^{o}pital's rule,
\[
\lim_{\phi\to\pi^-}\frac{G(\phi)}{\sin^2\phi}=\lim_{\phi\to\pi^-}
\frac{G'(\phi)}{2\sin\phi\cos\phi}.
\]
By \eqref{eqn:G-mod}, we can compute the derivative of $G$ as
\[
G'(\phi)=(m-1)(\dm+1)\left(-\sin\phi s(\phi)+\cos\phi s'(\phi)\right)
+2(m-2)(1+d(m-1))s(\phi)s'(\phi).
\]
Therefore, we get
\begin{align*}
\frac{G'(\phi)}{\sin\phi}&=-(m-1)(\dm+1)s(\phi)\\
&\qquad+\bigg((m-1)(\dm+1)\cos\phi+2(m-2)(1+\dm(m-1))s(\phi)\bigg)\frac{s'(\phi)}{\sin\phi}.
\end{align*}
From the above, together with \eqref{slimvalue} and \eqref{eqn:sposinphi}, we find
\begin{align*}
\lim_{\phi\to\pi-}\frac{G'(\phi)}{\sin\phi}&=-(m-1)(\dm+1)\lim_{\phi\to\pi-}s(\phi)\\
&\qquad+\lim_{\phi\to\pi-}\bigg((m-1)(\dm+1)\cos\phi+2(m-2)(1+\dm(m-1))s(\phi)\bigg)\lim_{\phi\to\pi-}\frac{s'(\phi)}{\sin\phi}\\
&=-\frac{(m-1)(\dm+1)}{1+\dm(m-1)}\\
&\qquad+\left(-(m-1)(\dm+1)+2(m-2)\right)\times \frac{-\dm(m-1)}{\left(1+\dm(m-1)\right)\left((\dm-2)(m-1)+2\right)}\\
&=\frac{2(m-1)(m-2)}{(1+\dm(m-1))((\dm-2)(m-1)+2)}.
\end{align*}
Finally, we get
\[
\lim_{\phi\to\pi^-}\frac{G(\phi)}{\sin^2\phi}=\lim_{\phi\to\pi^-}
\frac{G'(\phi)}{2\sin\phi\cos\phi}=-\frac{(m-1)(m-2)}{(1+\dm(m-1))((\dm-2)(m-1)+2)}<0
\]
\end{proof}


\bibliographystyle{abbrv}
\def\url#1{}
\bibliography{lit.bib}

\end{document}